\documentclass[11pt]{amsart}

 \usepackage{amsfonts,graphics,amsmath,amsthm,amsfonts,amscd,amssymb,amsmath,latexsym,multicol,
 mathrsfs}
\usepackage{epsfig,url}
\usepackage{flafter}
\usepackage{fancyhdr}
\usepackage{hyperref}
\hypersetup{colorlinks=true, linkcolor=black}

 \usepackage[matrix, arrow]{xy}

\DeclareMathOperator{\Pic}{Pic}
\DeclareMathOperator{\Proj}{Proj}

\DeclareMathOperator{\Spec}{Spec}
\DeclareMathOperator{\Supp}{Supp}

\DeclareMathOperator{\vol}{vol}

 \numberwithin{equation}{subsection}
 \numberwithin{footnote}{subsection}

 \newtheorem{cor}[subsection]{Corollary}
 \newtheorem{lem}[subsection]{Lemma}
 \newtheorem{prop}[subsection]{Proposition}
 \newtheorem{thm}[subsection]{Theorem}

{
\theoremstyle{upright}

 \newtheorem{exa}[subsection]{Example}

}

 \newcommand{\N}{\mathbb N}
 \newcommand{\PP}{\mathbb P}
 \newcommand{\A}{\mathbb A}
 \newcommand{\Q}{\mathbb Q}
 \newcommand{\R}{\mathbb R}
 \newcommand{\Z}{\mathbb Z}
  
 \newcommand{\bir}{\dashrightarrow}
 \newcommand{\rddown}[1]{\left\lfloor{#1}\right\rfloor} 

\title{\large S\MakeLowercase{ingularities on} F\MakeLowercase{ano fibrations and beyond}}
\author{\large C\MakeLowercase{aucher} B\MakeLowercase{irkar}}
\date{\today}
\begin{document}
\maketitle

\begin{abstract}
In this paper, we investigate singularities on fibrations and related topics. We prove conjectures of M$^{\rm c}$Kernan and Shokurov on singularities on Fano type fibrations and a conjecture of the author on singularities on log Calabi-Yau fibrations. From these we derive a variant of a conjecture of M$^{\rm c}$Kernan and Prokhorov on rationally connected varieties  with nef anti-canonical divisor. 

We present further applications to other problems including boundedness of klt complements for Fano fibrations over curves, torsion index of rationally connected Calabi-Yau pairs, and gonality of fibres of del Pezzo fibrations. 

We prove a general result on controlling multiplicities of fibres of certain fibrations (not necessarily of Fano type) which is the key ingredient of the proofs of the above results.   
\end{abstract}

\tableofcontents


\section{\bf Introduction}

We work over an algebraically closed field $k$ of characteristic zero unless stated otherwise.\\

We prove various statements regarding behaviour of singularities on fibrations $f\colon X\to Z$ where $X,Z$ are algebraic varieties and $f$ is a surjective projective morphism. This involves singularities of $X$, of the fibres of $f$, of the base $Z$, and other structures associated with the fibration. Understanding singularities on fibrations is a natural and important aspect of algebraic geometry, particularly birational geometry, as it frequently appears in inductive understanding of algebraic varieties. It is also important in other areas of mathematics and beyond, e.g. in arithmetic geometry (e.g. [\ref{Silverman}, Chapter IV, Section 8]) and mathematical physics (e.g. [\ref{BLee-1}]).

In this paper, we treat singularities in general settings (e.g. \ref{t-bnd-mult-lc-places-fib-main}) but we also pay particular attention to the case when the fibration $f$ is a Fano fibration or a log Calabi-Yau fibration. Assume for now that $f\colon X\to Z$ is a \emph{Fano fibration}, i.e. $X$ has log canonical (lc) singularities, $-K_X$ is ample over $Z$, and $f$ is a contraction. There are three main cases to consider: 
\begin{enumerate}
\item $\dim Z=0$: global case,
\item $0<\dim Z<\dim X$: fibre space case,
\item $\dim Z=\dim X$: birational case.
\end{enumerate}
There have been many fundamental problems associated with these different cases.
For years, a key focus of our work has been a program dedicated to resolving these problems.
Some of the main problems in case (1) were settled in [\ref{B-Fano}][\ref{B-BAB}]. The aim of this paper is to settle some of the main problems in case (2). The final phase of this program is case (3) with multiple open problems that are being pursued elsewhere. 

We will then focus on case (2).
Assume again that $f\colon X\to Z$ is a {Fano fibration}, and for simplicity of this discussion, assume $X,Z$ are $\Q$-factorial. Iskovskikh conjectured that if $f$ is a conic bundle (i.e. $\dim X-\dim Z=1$) and if $X$ is a 3-fold with terminal singularities, then $Z$ has canonical singularities. This conjecture was proved by Mori and Prokhorov [\ref{MP2}]. Moreover, Shokurov conjectured that if $f$ is a conic bundle and if $X$ has canonical singularities (with $X$ of arbitrary dimension), then $Z$ has $\frac{1}{2}$-lc
singularities which was proved by Han, Jiang and Luo [\ref{HCY}]. Both the canonical singularities
in Iskovskikh conjecture and the $\frac{1}{2}$-lc singularities in Shokurov conjecture are optimal.

M$^{\rm c}$Kernan proposed a generalisation of Iskovskikh conjecture where he conjectured that given $d\in \N$ and $\epsilon\in \R^{>0}$, there is $\delta\in \R^{>0}$ such that if $f\colon X\to Z$ is a Fano fibration where $X$ is $\epsilon$-lc of dimension $d$, then $Z$ is $\delta$-lc. Note that the class of $\epsilon$-lc singularities is much wider than that of terminal and canonical singularities that were traditionally used (see \ref{ss-pairs}). Alexeev and Borisov [\ref{AB}] proved the toric case of this conjecture, that is, when $X,Z$ are toric varieties and $f$ is a toric morphism. 

On the other hand,  Shokurov independently proposed an even more general conjecture regarding singularities on Fano type fibrations. Roughly speaking it says that given $d\in \N$ and $\epsilon\in \R^{>0}$, there is $\delta\in \R^{>0}$ such that if $f\colon (X,B)\to Z$ is a fibration where $(X,B)$ is $\epsilon$-lc of dimension $d$, $K_X+B\equiv 0/Z$, and $-K_X$ is big over $Z$, then the discriminant b-divisor defined by the canonical bundle formula has coefficients $\le 1-\delta$. As we will see, the discriminant b-divisor measures singularities on the fibration.

Shokurov conjecture implies M$^{\rm c}$Kernan conjecture and it contains more information about singularities of fibrations. A consequence of Shokurov conjecture is that if $f\colon X\to Z$ is a Fano fibration  
where $X$ is $\epsilon$-lc of dimension $d$, then multiplicities of the fibres of $f$ over codimension one points of $Z$ are bounded depending only on $d,\epsilon$. When $f$ is a del Pezzo fibration from a 3-fold $X$ with terminal singularities onto a curve $Z$, this boundedness of multiplicities of fibres was proved by Mori and Prokhorov [\ref{MP}].  The toric case was proved by Birkar and Chen [\ref{BC-toric}].

Shokurov conjecture was proved in Birkar [\ref{B-sing-fano-fib}] assuming $(F,\Supp B|_F)$ belongs to a fixed bounded family where $F$ is a general fibre of $f$. In particular, we can deduce that Shokurov conjecture holds when the coefficients of $B$ are bounded away from zero because in this case, by the BAB [\ref{B-BAB}, Theorem 1.1], $(F,\Supp B|_F)$ belongs to a bounded family. This is very useful in many situations, but to utilise the full power of the conjecture we need to work with arbitrary boundary coefficients. But then the difficulties we face are of an entirely different magnitude and we have to use a completely different approach to solve the problem. The issue here is not simply a combinatorial difficulty with the coefficients but instead it is a geometric problem and solving it requires a deeper understanding of fibrations and their singularities. 

One of the main results of this paper is a proof of Shokurov conjecture (see \ref{t-mc-sh-conj}). As mentioned above, this implies M$^{\rm c}$Kernan conjecture (see \ref{cor-mckernan-conj}) and boundedness of multiplicities of fibres over codimension one points (see \ref{cor-mult-fib}). It also implies another conjecture of Shokurov on boundedness of klt complements for Fano fibrations over curves (see \ref{cor-klt-comp}). 

Going beyond the Fano case, in [\ref{B-lcyf}, Conjecture 2.6], we proposed a conjecture concerning singularities on log Calabi-Yau fibrations that do not necessarily have a Fano type structure. We prove this conjecture (see \ref{t-cb-sing-usual-fib}) which is stronger than Shokurov conjecture and it will be important for future work on families of Calabi-Yau pairs. 

On the other hand, M$^{\rm c}$Kernan and Prokhorov [\ref{Mc-Pr}] conjectured that if $X$ is a projective rationally connected variety of dimension $d$ with $\epsilon$-lc singularities ($\epsilon>0)$ and if $-K_X$ is nef, then $X$ is bounded. This conjecture has been studied in recent years where boundedness has been replaced by boundedness up to isomorphism in codimension one. In this sense, it was proved in dimension 3 in Birkar-Di Cerbo-Svaldi [\ref{BDS}] (also see [\ref{CDCH+}]). We prove the conjecture in every dimension (see \ref{t-mc-pro-conj}) where again we use boundedness up to isomorphism in codimension one. A consequence of this is a proof of the index conjecture of Shokurov for rationally connected Calabi-Yau pairs (see \ref{t-cy-index-conj}).

The main tools that we will use in this paper are the minimal model program, the theory of complements, the theory of generalised pairs, and toroidal and toric geometries. We also make crucial use of some of the ideas of the proof of BAB and results in [\ref{B-BAB}][\ref{B-Fano}].

Let us also emphasize that, with a view towards the future, the importance of this work is not only in the results proved here but equally in the techniques and ideas introduced that make other fundamental problems accessible, e.g. Shokurov conjecture on boundedness of klt complements, ACC for mlds. Also their impact will likely go beyond birational and algebraic geometry, e.g. some of the toric results proved here have consequences for the geometry of integers [\ref{B-pos-sing-bnd}], and applications are expected in F-theory in mathematical physics in the context of elliptic Calabi-Yau varieties [\ref{BLee-1}][\ref{BLee-2}] and in K\"ahler geometry in the context of K\"ahler-Ricci shrinkers [\ref{SZ}]. 

In the rest of this introduction we will state our results more precisely and sometimes in more general settings than those mentioned above. \\

{\textbf{\sffamily{Singularities on Fano type fibrations.}}}
We start with Shokurov conjecture on singularities. 
Let $f\colon X\to Z$ be a contraction of normal varieties and
$(X,B)$ be a klt pair such that $K_X+B\sim_\R 0/Z$. By a construction of Kawamata [\ref{ka97}][\ref{ka98}] (also see Ambro [\ref{am99}]) we may write a canonical bundle formula (also called adjunction)
$$
K_X+B\sim_\R f^*(K_Z+B_Z+M_Z)
$$
where $B_Z$ is called the \emph{discriminant divisor} and 
$M_Z$ is called the \emph{moduli divisor}. The discriminant part is canonically determined as a Weil $\R$-divisor by the singularities of $(X,B)$ and the fibres over codimension one points of $Z$; the moduli part is then 
automatically determined as an $\R$-linear equivalence class. The divisor $B_Z$ measures the singularities on the fibration and the divisor $M_Z$ measures the variation of the log fibres in their ``moduli space". 
For each birational contraction $Z'\to Z$, one can similarly define $B_{Z'},M_{Z'}$ so that their pushdown to $Z$ coincide with $B_Z,M_Z$. One can then view $(Z,B_Z+M_Z)$ as a so-called generalised pair.  See \ref{s-adjunction} for more details.

\begin{thm}
	\label{t-mc-sh-conj}
	Let $d\in \mathbb{N}$ and let $\epsilon \in \mathbb{R}^{>0}$. Then there is $\delta \in \mathbb{R}^{>0}$ depending only on $d,\epsilon$ satisfying the following. 
	Let $(X,B)$ be a pair where $B$ is a $\Q$-boundary and let \(f\colon X\to Z\) be a contraction 
	such that 
\begin{itemize}
\item $(X,B)$ is $\epsilon$-lc and $\dim X-\dim Z=d$,
\item $K_X+B\sim_\Q 0/Z$, and 
\item $-K_X$ is big over $Z$.
\end{itemize}	
Then the generalised pair $(Z,B_Z+M_Z)$ induced by the canonical bundle formula 
$$
K_X+B\sim_\Q f^*(K_Z+B_Z+M_Z)
$$ 
is generalised $\delta$-lc.
\end{thm}

The theorem is equivalent to showing that for any birational contraction $Z'\to Z$, the coefficients of the discriminant divisor $B_{Z'}$ do not exceed $1-\delta$. 

The theorem as stated is quite general and enough for many applications. We will, however, prove a more general version of the theorem for $\R$-boundaries, for generalised pairs, and when $X,Z$ are over a base variety (see \ref{t-sh-sing-gen-fib}). Using the results of this paper, an even more general version of the theorem, also conjectured by Shokurov, is proved in [\ref{BC-vertical-elc}] where the $\epsilon$-lc assumption on $(X,B)$ is replaced by weaker conditions such as being vertical $\epsilon$-lc meaning that the log discrepancy $a(D,X,B)\ge \epsilon$ for divisors $D$ over $X$ which are vertical over $Z$.

The theorem was previously known when $(F,\Supp B_F)$ belongs to a fixed  bounded family where $F$ is a general fibre of $f$ and $B_F=B|_F$ [\ref{B-sing-fano-fib}], e.g. when the horizontal coefficients of $B$ are $\ge t$ for some fixed $t>0$, by [\ref{B-BAB}][\ref{B-sing-fano-fib}]. The method used in [\ref{B-sing-fano-fib}] is completely different from the one developed in this paper. The theorem was also known when $d\le 1$ [\ref{B-sing-fano-fib}]. Also, a global variant of the theorem is established in [\ref{B-FT-fib}] where one assumes that $X,Z$ are projective and that $A_Z-(K_Z+B_Z+M_Z)$ is ample for some very ample divisor $A_Z$  with bounded $A_Z^{\dim Z}$. 

In the general case, a well-known argument reduces the theorem to the case when $Z$ is a curve (cf. [\ref{B-sing-fano-fib}]). Thus it is reasonable to assume $\dim X-\dim Z=d$ rather than $\dim X=d$ (although Shokurov conjecture was commonly stated assuming the latter). See [\ref{HCY}][\ref{BC-vertical-elc}] for other variations.

An important open problem relating to the theorem is to give an explicit description of $\delta$ in terms of $d,\epsilon$. This is known for $d=1$ [\ref{BChen-d=1}].

We are ready to state some immediate consequences of the theorem for Fano fibrations.
Recall that in general by a Fano fibration \(f\colon X\to Z\) we mean $f$ is a contraction between normal varieties, $X$ has lc singularities, and $-K_X$ is ample over $Z$.

\begin{cor}
	\label{cor-mckernan-conj}
	Let $d\in \mathbb{N}$ and let $\epsilon \in \mathbb{R}^{>0}$. Then there is $\delta \in \mathbb{R}^{>0}$ depending only on $d,\epsilon$ satisfying the following. 
	Let \(f\colon X\to Z\) be a Fano fibration where $X$ is $\epsilon$-lc, $\dim X-\dim Z=d$, and $K_Z$ is $\Q$-Cartier. Then $Z$ is $\delta$-lc.
\end{cor}

This is proved by taking a boundary $B$ so that $(X,B)$ is $\epsilon$-lc and $K_X+B\sim_\Q 0/Z$, and then applying Theorem \ref{t-mc-sh-conj}. Note that if $X$ is $\Q$-factorial and if $X\to Z$ is a Mori fibre space, then $Z$ is automatically $\Q$-factorial.

\begin{cor}
	\label{cor-mult-fib}
	Let $d\in \mathbb{N}$ and let $\epsilon \in \mathbb{R}^{>0}$. Then there is $l\in \mathbb{N}$ depending only on $d,\epsilon$ satisfying the following. 
	Let \(f\colon X\to Z\) be a Fano fibration where $X$ is $\epsilon$-lc of dimension $d$ and $Z$ is a  curve.
	Then for each $z\in Z$, each coefficient of $f^*z$ is $\le l$.
\end{cor}

Theorem \ref{t-mc-sh-conj} is actually proved by reducing it to a variant of this corollary. This allows us to relate \ref{t-mc-sh-conj} to a general problem about boundedness of multiplicities of fibres of certain fibrations which are not necessarily of Fano type (see \ref{t-bnd-mult-lc-places-fib-main}).

\bigskip

{\textbf{\sffamily{Klt complements on Fano fibrations over curves.}}}
Another consequence of \ref{t-mc-sh-conj} and \ref{cor-mult-fib} regards existence of bounded klt complements. This is a case of a conjecture of Shokurov.

\begin{cor}
	\label{cor-klt-comp}
	Let $d\in \mathbb{N}$ and let $\epsilon \in \mathbb{R}^{>0}$. Then there is $n\in \mathbb{N}$ depending only on $d,\epsilon$ satisfying the following. Let \(f\colon X\to Z\) be a Fano fibration where $X$ is $\epsilon$-lc of dimension $d$ and $Z$ is a curve. Then for each $z\in Z$, there exists a klt $n$-complement of $K_X$ over $z$.
\end{cor}

Shokurov conjectured that the same is true even if $\dim Z>1$. If we replace klt with lc, then the conjecture was proved in  [\ref{B-Fano}]. 
For the definition of complements, see \ref{ss-compl}. Boundedness of complements is about understanding the asymptotic behaviour of the relative linear systems $|-mK_X|_{/Z}$ for $m\ge 0$. 

Using Theorem \ref{t-mc-sh-conj}, Chen [\ref{BChen-complements}] has recently shown that Shokurov conjecture on boundedness of klt complements can be reduced to the case when $X\to Z$ is birational. As a consequence he proves the conjecture when $\dim Z=2$. 

\bigskip

{\textbf{\sffamily{Fibres of del Pezzo fibrations.}}}
Another application of Theorem \ref{t-mc-sh-conj} is to special fibres of del Pezzo fibrations. 
This answers [\ref{BCDP}, Question 4.6] in the case of gonality in a more general setting (the question considers terminal singularities only). 

\begin{thm}\label{t-gonality}
For each $\epsilon\in \R^{>0}$, there is $g\in \N$ satisfying the following. Let $f\colon X\to Z$ be a Fano fibration where $X$ is a 3-fold with $\epsilon$-lc singularities and $Z$ is a curve. Let $S$ be an irreducible component of a fibre of $f$ over a closed point. Then $S$ is birational to $\PP^1\times C$ for a smooth projective curve $C$ with gonality $\le g$.
\end{thm}

The main point is that, putting $z=f(S)$, by Theorem \ref{t-mc-sh-conj}, $(X,tf^*z)$ is klt for some fixed $t>0$. This allows us to apply the results of [\ref{BL}]. Actually it is possible to bound the genus of $C$ and treat similar statements in higher dimension; see [\ref{BQ}] for more on this direction.

\bigskip

{\textbf{\sffamily{Rationally connected varieties with nef anti-canonical divisor.}}}
Another main application of Theorem \ref{t-mc-sh-conj} is the boundedness of certain rationally connected varieties which was conjectured by M$^{\rm c}$Kernan and Prokhorov [\ref{Mc-Pr}] (the two dimensional case is due to Alexeev [\ref{Alexeev-K2}]). 

\begin{thm}
	\label{t-mc-pro-conj}
	Let $d\in \mathbb{N}$ and let $\epsilon \in \mathbb{R}^{>0}$. Consider projective varieties $X$  
    where 
	\begin{itemize}
	\item $(X,B)$ is $\epsilon$-lc of dimension $d$ for some $B$,
	\item $-(K_X+B)$ is nef, and 
	\item $X$ is rationally connected.
	\end{itemize}
	Then such $X$ form a bounded family up to isomorphism in codimension one.
\end{thm}

This was known up to dimension $3$ [\ref{BDS}] which crucially relies on [\ref{B-FT-fib}] and [\ref{Jiang}]. Also see [\ref{CDCH+}] for a similar result in dimension $3$. To replace ``bounded family up to isomorphism in codimension one" with ``bounded family", one has to use techniques that are perhaps not related to this paper. It is likely more related to the Morrison-Kawamata-Totaro cone conjecture. On the other hand, the theorem is stronger than the original conjecture because here the coefficients of $B$ are arbitrary in the range $[0,1-\epsilon]$ while the conjecture assumed coefficients in $[\epsilon,1-\epsilon]$. 

Theorem \ref{t-mc-pro-conj} follows from Theorem \ref{t-mc-sh-conj}, ideas from [\ref{BDS}], and some additional arguments introduced in this paper. 

A more general version of the theorem in the context of generalised pairs holds. In fact, the next theorem is proved first and \ref{t-mc-pro-conj} is derived as an immediate corollary. The use of generalised pairs is crucial for the proof of the next result, hence also crucial for the proof of \ref{t-mc-pro-conj}.  

\begin{thm}\label{t-bnd-rc-gen-pairs}
	Let $d\in \mathbb{N}$ and let $\epsilon \in \mathbb{R}^{>0}$.
Consider projective varieties $X$ such that 
\begin{itemize}
\item $(X,B+M)$ is  generalised $\epsilon$-lc of dimension $d$ for some $B,M$,
\item $K_X+B+M\sim_\R 0$, and 
\item $X$ is rationally connected.
\end{itemize}
Then the set of such $X$ forms a bounded family up to isomorphism in codimension one. 
\end{thm}
  
Burt Totaro pointed out to us that the following result, which is a special case of the so-called index conjecture for Calabi-Yau pairs (due to Shokurov), follows from either of Theorems \ref{t-mc-pro-conj} and \ref{t-bnd-rc-gen-pairs}. 

\begin{cor}
	\label{t-cy-index-conj}
	Let $d\in \mathbb{N}$ and let $\Phi\subset [0,1]$ be a DCC set of rational numbers. Then there exists $I\in \N$ depending only on $d,\Phi$ satisfying the following. Consider pairs $(X,B)$  
    where 
	\begin{itemize}
	\item $(X,B)$ is projective klt of dimension $d$, 
	\item the coefficients of $B$ are in $\Phi$,
	\item $K_X+B\sim_\Q 0$, and 
	\item $X$ is rationally connected.
	\end{itemize}
	Then $I(K_X+B)\sim 0$.
\end{cor}

\bigskip

{\textbf{\sffamily{Singularities on log Calabi-Yau fibrations.}}}
For applications it is important to go beyond Fano type fibrations. This was treated in [\ref{B-lcyf}] under certain assumptions. The general version was proposed as a conjecture [\ref{B-lcyf}, Conjecture 2.6]. Our next result is a proof of this conjecture which also holds for $\R$-boundaries and generalised pairs (see \ref{t-cb-sing-gen-fib}). 

\begin{thm}\label{t-cb-sing-usual-fib}
	Let $d,v\in \mathbb{N}$ and let $\epsilon \in \mathbb{R}^{>0}$.
 Then there is $\delta\in \R^{>0}$ depending only on $d,v,\epsilon$ satisfying the following. 
 Assume that $(X,B)$ is a pair where $B$ is a $\Q$-boundary, and that $f\colon X\to Z$ is a contraction such that
\begin{itemize}
\item $(X,B)$ is $\epsilon$-lc of dimension $d$,

\item $K_X+B\sim_\Q 0/Z$, and  

\item there is an integral divisor $N$ on $X$ which is big over $Z$ with $\vol(N|_F)<v$ 
for the general fibres $F$ of $f$.
\end{itemize}
Then the generalised pair $(Z,B_Z+M_Z)$ induced by the canonical bundle formula, is generalised $\delta$-lc.
\end{thm} 

Note that in [\ref{B-lcyf}, Conjecture 2.6] it is assumed that $N\ge 0$ and that 
$$
0<\vol(B|_F+M|_F+N|_F)<v.
$$ 
In this setting, we can apply the above theorem to $-K_X+N$ in place of $N$.
On the other hand, 
the theorem is clearly stronger than Theorem \ref{t-mc-sh-conj} because in the case of Fano type fibrations we can take $N=-K_X$ and use the fact that $\vol(-K_F)$ is bounded from above, by [\ref{B-BAB}]. However, we prove \ref{t-mc-sh-conj} first and then derive \ref{t-cb-sing-usual-fib}.

A simple example where the theorem applies is when $f\colon X\to Z$ is a minimal elliptic fibration and $N$ is a rational multi-section of degree $<v$. More sophisticated examples appear in the context of variations of polarised Calabi-Yau pairs (cf. [\ref{BH}, Lemma 6.7]).

\bigskip

{\textbf{\sffamily{Multiplicities along lc places.}}}
A key ingredient of the proof of Theorem \ref{t-mc-sh-conj} is the following result on controlling multiplicities of fibres. Much of the paper is in fact devoted to the proof of this. Roughly speaking, it is saying that in fibrations that are generically relatively bounded, we can control special fibres under suitable singularity and numerical minimality conditions. Although this has no apparent connection with Fano fibrations, but interestingly, its proof is eventually reduced to understanding toric Fano fibrations.

\begin{thm}\label{t-bnd-mult-lc-places-fib-main}
	Let $d,r\in \mathbb{N}$ and let $\epsilon \in \mathbb{R}^{>0}$.
Then there is $l\in \N$ depending only on $d,r,\epsilon$ satisfying the 
following. Assume that 
\begin{itemize}

\item $(W,B_W)$ is an $\epsilon$-lc pair and $(X,\Lambda)$ is an lc pair, both of dimension $d$,

\item $W\to X$ is a birational contraction and $X\to Z$ is a surjective projective morphism onto a smooth curve, 

\item $K_W+B_W$ is nef$/X$, 

\item  $A$ is a very ample$/Z$ divisor on $X$ such that $\deg_{A/Z}A\le r$,

\item $A-B$ and $A-\Lambda$ are pseudo-effective over $Z$ where $B$ is the pushdown of $B_W$, 

\item $r\Lambda$ is integral, and 

\item $T$ is a prime divisor over $X$ mapping to a closed point $z\in Z$ with 
$$
a(T,W,B_W)\le 1 ~~~\mbox{and}~~~ a(T,X,\Lambda)=0.
$$
\end{itemize}
Then $\mu_TF\le l$ where $F$ is the fibre of $X\to Z$ over $z$. That is, if $\phi\colon V\to X$ is a 
resolution so that $T$ is a divisor on $V$, then $\mu_T\phi^*F\le l$.
\end{thm}

This can be viewed as a much stronger relative version of [\ref{B-BAB}, Theorem 1.8] over curves. The setting of [\ref{B-BAB}, Theorem 1.8] is global where we have $W=X$ but in the above theorem $W\to X$ is a birational morphism and this makes the proof much harder (it is important to allow this flexibility for applications and also for the proof to go through). 
The proof of [\ref{B-BAB}, Theorem 1.8] goes via reduction to the log smooth case making use of the fact that bounded families of projective varieties have bounded resolutions, and then reducing to the toric case, that is, when $(X,\Lambda)$ is toric. 

In the relative case above, we cannot simply reduce to the log smooth case because relative bounded resolutions do not exist even in dimension $d=2$. Instead, we use bounded toroidal models (\ref{t-bnd-torification}). In turn, we use the toroidal models to reduce to the toric case. Treating this toric case turned out to be unexpectedly very hard and time-consuming. Note that by toric case we mean $(X,\Lambda)$ is toric while $(W,B_W)$ might be far from toric.
It is worth mentioning that despite the similarities, the proofs of \ref{t-bnd-mult-lc-places-fib-main} and [\ref{B-BAB}, Theorem 1.8] are overall very different.

\bigskip

{\textbf{\sffamily{Bounded toroidalisation of fibrations over curves.}}}
The result on bounded toroidalisations (\ref{t-bnd-torification}) mentioned above roughly speaking says that given a relatively bounded family over a curve, we can alter it to a relatively bounded and toroidal family. This was proved in an earlier version of this paper using families of nodal curves, but its proof has been moved to the new paper [\ref{B-toroid-toric-model-fib}].

\bigskip

{\textbf{\sffamily{A few words about the proofs of Theorems \ref{t-mc-sh-conj} and \ref{t-bnd-mult-lc-places-fib-main}.}}}
We present a brief discussion of some of the ideas of the proofs of \ref{t-mc-sh-conj} and \ref{t-bnd-mult-lc-places-fib-main}.
Assume that we are in the setting of Theorem  \ref{t-mc-sh-conj}.

Step 1. Taking hyperplane sections on $Z$ (and on appropriate birational models), it is possible to reduce the theorem to the case when $\dim Z=1$. Fixing a closed point $z\in Z$, the theorem is equivalent to finding a real number $\tau>0$ depending only on $d,\epsilon$ so that $(X,B+\tau f^*z)$ is lc. If necessary, we will shrink $Z$ around $z$ without notice. Let $t$ be the maximal real number so that $(X,B+tf^*z)$ is $\frac{\epsilon}{2}$-lc. It is enough to show that $t$ is bounded from below away from zero. There is a prime divisor $T$ over $X$ and mapping to $z$ with 
$$
a(T,X,B+tf^*z)=\frac{\epsilon}{2}.
$$ 
It is enough to show that the multiplicity of $T$ in the fibre of $V\to Z$ over $z$ is bounded from above for any resolution $V\to X$ on which $T$ appears. It is then possible to replace $\epsilon$ and replace $X\to Z$ with a Fano fibration so that $T=\Supp f^*z$ and $\mu_TB=1-{\epsilon}$. We want to show that the multiplicity $\mu_Tf^*z$ is bounded from above (this is the setting of \ref{cor-mult-fib}).

Step 2. Applying [\ref{B-Fano}], there is a fixed $n$ so that $K_X$ has an $n$-complement $K_X+\Lambda$ over $z$ which means $(X,\Lambda)$ is lc and $n(K_X+\Lambda)\sim 0$. 
Moreover, starting with $\epsilon$ small enough and using $\mu_TB=1-{\epsilon}$, we can ensure $\mu_T\Lambda=1$.
For simplicity, let's assume that $-K_X$ is ample over $Z$. Applying BAB [\ref{B-BAB}], the general log fibres $(F,\Lambda_F)$ belong to a bounded family, and replacing $n$, the linear system $|-nK_F|$ is very ample. We can use this to find a birational map $X\bir Y/Z$ so that if $D$ is the birational transform of the horizontal part of $\Supp \Lambda$, then $(Y,D)\to Z$ is a relatively bounded family (essentially this means all the log fibres are bounded including the fibre over $z$). However, $Y,D$ may have very bad singularities. 

Step 3. Applying toroidalisation, we alter the family $(Y,D)\to Z$ to a toroidal family $(Y',D')\to (Z',E')$ which is relatively bounded. The advantage of this is that now $(Y',D')$ is an lc pair with formally toric singularities. Moreover,
if $K_{Y'}+\Lambda_{Y'}$ is the pullback of $K_X+\Lambda$ under the rational map $Y'\bir X$, then we can ensure that $\Lambda_{Y'}\le D'$. Also 
we can find a prime divisor $T'$ over $Y'$, derived from $T$, with $a(T',Y',\Lambda_{Y'})=0$ which implies $a(T',Y',D')=0$ meaning $T'$ is toroidal with respect to $(Y',D')$. Let $K_{Y'}+B_{Y'}$ be the pullback of $K_X+B$ under $Y'\bir X$. In general, $B_{Y'}$ may have very negative coefficients. So it is difficult to work on $Y'$. But after some modifications, we can construct a birational model $(W,B_W)\to Y'$ where $(W,B_W)$ is $\epsilon$-lc, $K_W+B_W$ is nef over $Y'$, $a(T',W,B_W)\le 1$, and the pushdown of $B_W$ to $Y'$ is the effective part of $B_{Y'}$. Applying induction on dimension, we can assume the centre of $T'$ on $Y'$ is a closed point $y'$.

Step 4. Although we have $(Y',D')$ which is toroidal and relatively bounded over $Z'$ but $Y'$ might be far from being Fano type over $Z'$. The idea is then to pass to a model which is not only toroidal but actually toric over a formal neighbourhood of $z'$ where $z'$ is the image of $y'$.  We will construct a model $(Y'',D'')\to Z'$ with a point $y''\in Y''$ so that $(Y',D')$ and $(Y'',D'')$ are formally isomorphic near $y',y''$. We will use this to find a prime divisor $T''$ over $Y''$ with centre $y''$, corresponding to $T'$. Moreover, we ensure that there is a birational map $Y''\bir P=\PP^{d-1}_{Z'}/Z'$ so that if $G$ is the sum of the coordinate hyperplanes on $P$ (induced by $\PP^{d-1}$) plus the fibre over $z'$, then $T''$ is toroidal with respect to $(P,G)$. Additionally, we will use the formal isomorphism between $(Y',D'),y'$ and $(Y'',D''),y''$ to find a birational model $(V,B_V)\to P$ which is $\epsilon$-lc, $K_V+B_V$ is nef over $P$, and $T''$ is a divisor on $V$ (which in particular means $a(T'',V,B_V)\le 1$). We then reduce the original problem to showing that the multiplicity of $T''$ in the fibre of $V\to Z'$ is bounded from above.  

Step 5. We will modify our setting so that we can assume $Z'=\A^1$. In particular, $(P,G)\to Z'$ is toric. Using the fact that $T''$ is toroidal with respect to $(P,G)$, we will translate the problem of bounding the multiplicity of $T''$ in the fibre of $V\to Z'$ to a purely toric problem (see \ref{p-Nsigma-Y-model}). We will solve the toric problem which completes the proof of \ref{t-mc-sh-conj}. 

Now assume we are in the setting of Theorem \ref{t-bnd-mult-lc-places-fib-main}. The idea is to use toroidalisation to reduce to the case when $(X,\Lambda)$ is toroidal similar to Step 3 above. The rest of the argument is then similar to those in Steps 4,5 above. In practice though we first prove Theorem \ref{t-bnd-mult-lc-places-fib-main} and then derive Theorem \ref{t-mc-sh-conj} following arguments as in Steps 1,2,3.

\bigskip

{\textbf{\sffamily{Plan of the paper.}}}
In Section 2 we collect some preliminary results. In Sections 3,4 we discuss couples, toroidal geometry, and relatively bounded toroidalisation of relatively bounded families together with their toric models. In Section 5 we introduce and study strict transform of divisors in not necessarily birational settings. In Section 6 we reduce Theorem \ref{t-bnd-mult-lc-places-fib-main} first to the toroidal case and then to the relatively projective space case. In Section 7 we reduce Theorem \ref{t-bnd-mult-lc-places-fib-main} to a toric problem and then settle the problem. In Section 8 we prove Theorem \ref{t-mc-sh-conj} and its corollaries \ref{cor-mckernan-conj}, \ref{cor-mult-fib}, \ref{cor-klt-comp} together with Theorem \ref{t-gonality}. In Section 9 we prove Theorem \ref{t-mc-pro-conj} and its version for generalised pairs along with Corollary \ref{t-cy-index-conj}. In Section 10 we prove Theorem \ref{t-cb-sing-usual-fib} and its version for generalised pairs.   
\bigskip

{\textbf{\sffamily{Acknowledgements.}}}
This work was partially done at the University of Cambridge. It was completed at Tsinghua University with support of a grant from Tsinghua University and a grant of the National Program of Overseas High Level Talent. Thanks to Florin Ambro, Bingyi Chen, Jingjun Han, Xiaowei Jiang, Junpeng Jiao, Santai Qu, Roberto Svaldi, Burt Totaro, and Yu Zou for their valuable comments. 
And thanks to the participants of activities devoted to this work, including a workshop in June 2023 and a seminar series in March--May 2024 at Tsinghua University and a workshop in May 2024 at Fudan University. 


\section{\bf Preliminaries}

We work over an algebraically closed field $k$ of characteristic zero. Varieties are all quasi-projective over $k$ unless stated otherwise.

\subsection{Morphisms}

An \emph{alteration} is a surjective projective morphism $Y\to X$ of varieties of the same dimension, hence it is generically finite. 
A \emph{contraction} is a projective morphism $f\colon X\to Z$ with $f_*\mathcal{O}_X=\mathcal{O}_Z$, hence it is surjective with connected fibres.  

Given a morphism $g\colon Y\to X$ of schemes and a subset $T\subset X$, $g^{-1}T$ denotes the set-theoretic inverse image of $T$. If $T$ is a closed subscheme, we then consider $g^{-1}T$ with its induced reduced scheme structure. But if we consider the scheme-theoretic inverse image of $T$, we will say so explicitly. 

\subsection{Divisors, degree, and volume}\label{ss-divs-deg}
Let $X$ be a normal variety and let $D$ be an $\R$-divisor.  Writing $D=\sum d_iD_i$ where $D_i$ are the distinct irreducible components of $D$, for each real number $a$ we define 
$D^{\le a}=\sum\min\{a,d_i\}D_i$. For a prime divisor $T$ on $X$, $\mu_TD$ denotes the coefficient of $T$ in $D$. If $D$ is $\R$-Cartier and if $T$ is a prime divisor over $X$, i.e., on some birational modification $g\colon Y\to X$, then by $\mu_TD$ we mean $\mu_Tg^*D$. Here and elsewhere,
by a \emph{birational modification}, we mean a birational contraction $Y\to X$ from a normal variety.

Let $f\colon X\to Z$ be a surjective projective morphism of varieties. 
For an $\R$-divisor $D$ on $X$, we define 
$$
|D|_{\R/Z}=\{D'  \mid  0\le D'\sim_\R D/Z\}.
$$
Now let $A$ be a $\Q$-Cartier divisor on $X$. 
For a Weil divisor $D$ on $X$ we define the \emph{relative degree} of $D$ over $Z$ with respect to $A$ 
as 
$$
\deg_{A/Z}D:=(D|_F)\cdot (A|_F)^{n-1}
$$ 
where $F$ is a general fibre of $f$ and $n=\dim F$. 
It is clear that this is a generic degree,  so the vertical$/Z$ components of $D$ do not contribute 
to the degree. Note that $F$ may not be irreducible: by a general fibre we mean fibre over a general point of $Z$.
In practice, we take $A$ to be ample over $Z$.
A related notion is the \emph{relative volume} of $D$ over $Z$ which we define as $\vol_{/Z}(D):=\vol(D|_F)$.

For a morphism $g\colon V\to X$ of varieties (or schemes) and an $\R$-Cartier $\R$-divisor $N$ on $X$, 
we sometimes write $N|_V$ instead of $g^*N$. 

For a birational map $X\bir X'$ (resp. $X\bir X''$)(resp. $X\bir X'''$)(resp. $X\bir Y$) of varieties 
whose inverse does not contract divisors, and for 
an $\R$-divisor $D$ on $X$, we usually denote the pushdown of $D$ to $X'$ (resp. $X''$)(resp. $X'''$)(resp. $Y$) 
by $D'$ (resp. $D''$)(resp. $D'''$)(resp. $D_Y$).

\subsection{Pairs and singularities}\label{ss-pairs}
A \emph{pair} $(X,B)$ consists of a normal variety $X$ and an $\R$-divisor $B\ge 0$ such that $K_X+B$ is $\R$-Cartier. We call $B$ the \emph{boundary divisor}.

Let $\phi\colon W\to X$ be a log resolution of a pair $(X,B)$. Let $K_W+B_W$ be the 
pullback of $K_X+B$. The \emph{log discrepancy} of a prime divisor $D$ on $W$ with respect to $(X,B)$ 
is defined as 
$$
a(D,X,B):=1-\mu_DB_W.
$$
A \emph{non-klt place} of $(X,B)$ is a prime divisor $D$ over $X$, that is, 
on birational modifications of $X$,  such that $a(D,X,B)\le 0$, and a \emph{non-klt centre} is the image of such a $D$ on $X$. 

We say $(X,B)$ is \emph{lc} (resp. \emph{klt})(resp. \emph{$\epsilon$-lc}) if 
 $a(D,X,B)\ge 0$ (resp. $>0$)(resp. $\ge \epsilon$) for every $D$. This means that  
every coefficient of $B_W$ is $\le 1$ (resp. $<1$)(resp. $\le 1-\epsilon$). 
Note that since $a(D,X,B)=1$ for most prime divisors, we necessarily have $\epsilon\le 1$.

A \emph{log smooth} pair is a pair $(X,B)$ where $X$ is smooth and $\Supp B$ has simple 
normal crossing singularities. Assume $(X,B)$ is a log smooth pair and assume $B=\sum_{i=1}^r B_i$ is reduced  
where $B_i$ are the irreducible components of $B$. 
A \emph{stratum} of $(X,B)$ is an irreducible component of $\bigcap_{i\in I}B_i$ for some non-empty $I\subseteq \{1,\dots,r\}$.  
Since $B$ is reduced, a stratum is nothing but a non-klt centre of $(X,B)$.

\subsection{b-divisors}\label{ss-b-div}
A \emph{b-$\R$-divisor} $\bf M$ over a normal variety $X$ consists of an $\R$-divisor ${\bf M}_Y$ on each birational modification $Y\to X$ so that if $Y\to X$ and $Y'\to X$ are two birational modifications with corresponding divisors ${\bf M}_Y, {\bf M}_{Y'}$ and so that the induced map $Y'\bir Y$ is a morphism, then ${\bf M}_Y$ is the pushdown of ${\bf M}_{Y'}$. 
A \emph{b-$\R$-divisor} $\bf M$ is \emph{b-$\R$-Cartier} if there is a birational modification $Y\to X$ so that ${\bf M}_Y$ is $\R$-Cartier and so that for any birational modification $Y'\to Y$, ${\bf M}_{Y'}$ is the pullback of ${\bf M}_Y$. 

To define generalised pairs below we will use b-divisors but with different notation. We just need to keep in mind that a b-$\R$-Cartier b-$\R$-divisor over $X$ is determined by the choice of a birational modification $Y\to X$ and an $\R$-Cartier $\R$-divisor $M$ on $Y$ up to the following equivalence: 
 another birational modification $Y'\to X$ and $\R$-Cartier $\R$-divisor
$M'$ defines the same b-$\R$-Cartier  b-$\R$-divisor if there is a common resolution $W\to Y$ and $W\to Y'$ on which the pullbacks of $M$ and $M'$ coincide.  

A b-$\R$-Cartier  b-$\R$-divisor  represented by some $Y\to X$ and $M$ is \emph{b-Cartier} if  $M$ is 
b-Cartier, i.e. its pullback to some resolution is Cartier.

\subsection{Generalised pairs}\label{ss-gpp}
A \emph{generalised pair} consists of 
\begin{itemize}
\item a normal variety $X$ equipped with a projective
morphism $X\to Z$, 

\item an $\R$-divisor $B\ge 0$ on $X$, and 

\item a b-$\R$-Cartier  b-$\R$-divisor over $X$ represented 
by some birational modification $X' \overset{\phi}\to X$ and $\R$-Cartier $\R$-divisor
$M'$ on $X$
\end{itemize}
such that $M'$ is nef$/Z$ and $K_{X}+B+M$ is $\R$-Cartier,
where $M:= \phi_*M'$. 

To specify such a generalised pair we say that $(X,B+M)$ is a generalised pair with data $X'\overset{\phi}\to X\to Z$ and $M'$, and refer to $M'$ as the nef part. But sometimes we just say $(X,B+M)$ is a generalised pair when $X'\overset{\phi}\to X\to Z$ and $M'$ are understood and there is no danger of confusion. Also, since a b-$\R$-Cartier b-divisor is defined birationally, 
in practice we will often replace $X'$ with a resolution and replace $M'$ with its pullback.
When $Z$ is a point we drop it but say the pair is projective. 

Now we define singularities of a generalised pair $(X,B+M)$.
Replacing $X'$, we can assume $\phi$ is a log resolution of $(X,B)$. We can write 
$$
K_{X'}+B'+M'=\phi^*(K_{X}+B+M)
$$
for some uniquely determined $B'$. For a prime divisor $D$ on $X'$, the \emph{generalised log discrepancy} 
$a(D,X,B+M)$ is defined to be $1-\mu_DB'$. 

We say $(X,B+M)$ is 
\emph{generalised lc} (resp. \emph{generalised klt})(resp. \emph{generalised $\epsilon$-lc}) 
if for each $D$ the generalised log discrepancy $a(D,X,B+M)\ge 0$ (resp. $>0$)(resp. $\ge \epsilon$).

For the basic theory of generalised pairs, see [\ref{BZh}, Section 4] and [\ref{B-gen-pairs}].

\subsection{Minimal models, Mori fibre spaces, and MMP}
Let $ X\to Z$ be a
projective morphism of normal  varieties and $D$ be an $\R$-Cartier $\R$-divisor
on $X$. Let $Y$ be a normal  variety, projective over $Z$, and $\phi\colon X\bir Y/Z$
be a birational map whose inverse does not contract any divisor. 
Assume $D_Y:=\phi_*D$ is also $\R$-Cartier and that 
there is a common resolution $g\colon W\to X$ and $h\colon W\to Y$ such that
$E:=g^*D-h^*D_Y$ is effective and exceptional$/Y$, and
$\Supp g_*E$ contains all the exceptional divisors of $\phi$.

Under the above assumptions, we call $Y$ 
a \emph{minimal model} of $D$ over $Z$ if $D_Y$ is nef$/Z$.
On the other hand, we call $Y$ a \emph{Mori fibre space} of $D$ over $Z$ if there is an extremal contraction
$Y\to T/Z$ with $-D_Y$  ample$/T$ and $\dim Y>\dim T$.

If one can run a \emph{minimal model program} (MMP) on $D$ over $Z$ which terminates 
with a model $Y$, then $Y$ is either a minimal model or a Mori fibre space of 
$D$ over $Z$. 

\subsection{Fano type varieties}
Assume that $X$ is a variety and $X\to Z$ is a contraction. We say $X$ is \emph{of Fano type over} $Z$ 
if there is a boundary $C$ such that $(X,C)$ is klt and $-(K_X+C)$ is ample over $Z$ 
(or equivalently if there is a boundary $D$ such that $(X,D)$ is klt and $-(K_X+D)$ is nef and big over $Z$). This is equivalent to 
having a boundary $B$ such that $(X,B)$ is klt, $K_X+B\sim_\R 0/Z$, and $B$ is big over $Z$. 
By [\ref{BCHM}], we can run an MMP over $Z$ on any $\R$-Cartier $\R$-divisor $D$ on $X$ and the MMP ends 
with a minimal model or a Mori fibre space for $D$.

\subsection{Nakayama divisor}
Let $X$ be a normal projective variety and let $D$ be a pseudo-effective $\R$-Cartier $\R$-divisor on $X$.
When $X$ is smooth, Nakayama [\ref{Nakayama}] associates an $\R$-divisor $N_\sigma(D)\ge 0$ to $D$, by a limiting process, which intuitively captures the numerically negative part of $D$. If $X$ is not smooth, we can define $N_\sigma(D)$ to be $\psi_*N_\sigma(\psi^*D)$ for a resolution $\psi\colon W\to X$ (this does not depend on the choice of the resolution, by [\ref{Nakayama}, Chapter III, Theorem 5.16]).
If we can run an MMP on $D$ ending with a minimal model $X'$ for $D$, then  denoting $X\bir X'$ by $\phi$, one can check that $N_\sigma(D)=D-\phi^*D'$ where $D'$ is the pushdown of $D$ and $\phi^*D'$ means pulling back $D'$ to a common resolution of $X,X'$ and then push down to $X$.
   
We define a relative version but only in the Fano type setting in which case we can run MMP. Let $X\to Z$ be a projective morphism of normal varieties and assume $X$ is of Fano type over $Z$. Let $D$ be an $\R$-Cartier $\R$-divisor on $X$ which is pseudo-effective$/Z$. By the Fano type assumption, we can run an MMP on $D$ over $Z$ which terminates with a minimal model $X'$ for $D$ over $Z$. Denoting $X\bir X'$ by $\phi$, we define $N_\sigma(D/Z):=D-\phi^*D'$ where $D'$ is the pushdown of $D$. In particular, $N_\sigma(D/Z)\ge 0$ by the negativity lemma. If $Z$ is a point, then this definition is consistent with the previous paragraph.

Actually, in the relative setting, we are mainly interested in the support of $N_\sigma(D/Z)$ rather than its coefficients but it is convenient to use the notation $N_\sigma(D/Z)$.

\subsection{Complements}\label{ss-compl}
Let $(X,B)$ be a pair  and let $X\to Z$ be a contraction. 
A \emph{strong $n$-complement} of $K_{X}+B$ over a point $z\in Z$ is of the form 
$K_{X}+{B}^+$ where over some neighbourhood of $z$ we have the following properties:
\begin{itemize}
\item $(X,{B}^+)$ is lc, 

\item $n(K_{X}+{B}^+)\sim 0$, and 

\item ${B}^+\ge B$.
\end{itemize}

When $Z$ is a point, we just say that $K_X+B^+$ is a strong $n$-complement of $K_X+B$. We recall one of the main 
results of [\ref{B-Fano}] on complements.

\begin{thm}[\ref{B-Fano}, Theorem 1.8]\label{t-bnd-compl-usual-local}
Let $d$ be a natural number and $\mathfrak{R}\subset [0,1]$ be a finite set of rational numbers.
Then there exists a natural number $n$
depending only on $d$ and $\mathfrak{R}$ satisfying the following.  
Assume $(X,B)$ is a pair and $X\to Z$ is a contraction such that 
\begin{itemize}
\item $(X,B)$ is lc of dimension $d$ and $\dim Z>0$,

\item the coefficients of $B$ are in $\Phi(\mathfrak{R})$, 

\item $X$ is of Fano type over $Z$, and 

\item $-(K_{X}+B)$ is nef over $Z$.
\end{itemize}
Then for any point $z\in Z$, there is a strong $n$-complement $K_{X}+{B}^+$ of $K_{X}+{B}$ 
over $z$. Moreover, the complement is also an $mn$-complement for any $m\in \N$. 
\end{thm}

Recall that 
$
\Phi(\mathfrak{R})=\left\{1-\frac{r}{m} \mid r\in \mathfrak{R},~ m\in \N\right\}.
$

\subsection{Adjunction for fibrations}\label{s-adjunction} 
We recall the generalised adjunction formula for fibrations in a setting suitable for this paper. This formula has been developed gradually in different degrees of generality, see [\ref{ka97}][\ref{ka98}][\ref{am05}][\ref{am04}][\ref{Filipazzi-18}][\ref{CHLX-folitation-g-pair}] (also see [\ref{B-lcyf}]).  

(1) 
Assume that 
\begin{itemize}
\item $(X,B+M)$ is a generalised klt generalised pair with data $X'\to X\to S$ and $M'$,

\item $f\colon X\to Z/S$ is a contraction with $\dim Z>0$, 

\item $K_{X}+B+M\sim_\R 0/Z$.
\end{itemize} 
Then the \emph{generalised adjunction formula} (also called canonical bundle formula) says that we can write 
$$
K_{X}+B+M\sim_\R f^*(K_Z+B_Z+M_Z)
$$
where the \emph{discriminant divisor} $B_Z$ is canonically defined and the \emph{moduli divisor} $M_Z$ is determined up to $\R$-linear equivalence.
 
More precisely, $B_Z$ can be defined as follows. Let $D$ be a prime divisor on $Z$. Let $t$ be the generalised lc threshold of $f^*D$ with respect to $(X,B+M)$ 
over the generic point of $D$, that is, $t$ is the largest real number so that $(X,B+M+tf^*D)$ is generalised lc over the generic point of $D$. This makes sense even if $D$ is not $\Q$-Cartier because we only need 
the pullback $f^*D$ over the generic point of $D$ where $Z$ is smooth. 
We then let the coefficient of $D$ in $B_Z$ to be $1-t$. 
Having defined $B_Z$, we can find $M_Z$ giving 
$$
K_{X}+B+M\sim_\R f^*(K_Z+B_Z+M_Z)
$$
where $M_Z$ is determined up to $\R$-linear equivalence over $S$.  

Let $Z'\to Z$ be a birational modification. There is 
a birational modification $X'\to X$  so that the induced map $X'\bir Z'$ is a morphism.  
Let $K_{X'}+B'+M'$ be the 
pullback of $K_{X}+B+M$. We can similarly define $B_{Z'},M_{Z'}$ for $(X',B'+M')$ over $Z'$. In this way, we get the \emph{discriminant b-divisor ${\bf{B}}_Z$} of adjunction for $(X,B+M)$ over $Z$. 
Fixing a choice of $M_Z$, we can pick the $M_{Z'}$ consistently so that it also defines a b-divisor ${\bf{M}}_Z$ which we refer to as the \emph{moduli b-divisor} of adjunction for $(X,B+M)$ over $Z$. 

Moreover, if $B,M'$ are $\Q$-divisors, then ${\bf B}_Z$ is a b-$\Q$-divisor and we can choose ${\bf M}_Z$ so that it is also a b-$\Q$-divisor.

(2) The moduli b-divisor ${\bf M}_Z$ is b-$\R$-Cartier in the sense of \ref{ss-b-div} and it is nef over $S$, by [\ref{ka97}][\ref{ka98}][\ref{am05}][\ref{am04}][\ref{Filipazzi-18}] in the case of $\Q$-divisors and [\ref{CHLX-folitation-g-pair}, Theorem 11.4.4] in the case of $\R$-divisors. In particular, we can view $(Z,B_Z+M_Z)$ as a generalised pair with data $Z'\to Z\to S$ and $M_{Z'}$ where $Z'\to Z$ is any high resolution. In this paper, we are often concerned with the singularities of $(Z,B_Z+M_Z)$ (e.g. in Theorem \ref{t-mc-sh-conj}) in which case $S$ is not relevant and we may simply take $Z\to S$ to be the identity morphism.

\subsection{Image of divisors}

\begin{lem}\label{l-desced-prime-div}
Assume that $Y\to X$ is a dominant morphism of varieties, which is \'etale at a closed point $y\in Y$. Assume that $D$ is a prime divisor over $Y$ with centre passing through $y$. Then we can find resolutions $Y'\to Y$ and $X'\to X$ so that the induced map $Y'\bir X'$ is a morphism, $D$ is a divisor on $Y'$, and the image of $D$ on $X'$ is a divisor.
\end{lem}
\begin{proof}
First, pick a resolution $X'\to X$, let $Y''$ be the main component of $Y\times_XX'$, and let $y''\in Y''$ be a closed point that maps to $y$ and is contained in the centre of $D$ on $Y''$. Then the induced map $Y''\to X'$ is \'etale at $y''$, in particular, $Y''$ is smooth at $y''$. Take a resolution $Y'\to Y''$ which is an isomorphism over $y''$. Replacing $Y\to X$ and $y$ with $Y'\to X'$ and $y''$, we can assume that $Y$ is smooth at $y$. If necessary, we will replace $y$ by a general closed point of the centre of $D$ on $Y$. 

Let $C$ be the centre of $D$ on $Y$ and let $E$ be the closure of the image of $C$ on $X$. Shrinking $Y,X$, we can assume that $Y,X,C,E$ are all smooth and that $Y\to X$ is \'etale. Let $X'\to X$ be the blowup of $X$ along $E$. Shrinking $Y$ and letting $Y'=Y\times_XX'$, the induced map $Y'\to Y$ is the blowup of $Y$ along $C$. Also $Y'\to X'$ is \'etale. Replace $Y\to X,y$ with $Y'\to X',y'$ where $y'\in Y'$ is a closed point mapping to $y$ and contained in the centre of $D$ on $Y'$. Repeat this process. By [\ref{kollar-mori}, Lemma 2.45], after finitely many steps, $D$ is a divisor on $Y$. Since $Y\to X$ is \'etale, the image of $D$ on $X$ is also a divisor.  
\end{proof}

\subsection{Coordinate hyperplanes}

By coordinate hyperplanes of 
$$
\PP^n=\Proj k[\beta_0,\dots,\beta_n]
$$ 
we mean the hyperplanes 
defined by the $\beta_i$. When $Z$ is a variety, by coordinate hyperplanes on 
$\PP^n_Z=\PP^n\times Z$ we mean the pullback of the coordinate hyperplanes on 
$\PP^n$ via the projection $\PP^n_Z\to \PP^n$.


\section{\bf Couples and toroidal geometry}

\subsection{Couples} 

A \emph{couple} $(X,D)$ consists of a variety $X$ and a reduced Weil divisor $D$ on $X$.
This is more general than the definition given in [\ref{B-sing-fano-fib}] because we are not assuming $X$ to be normal 
nor projective. Also note that a couple is not necessarily a pair in the sense that we are not assuming 
$K_X+D$ to be $\Q$-Cartier. In this paper, we often consider a couple $(X,D)$ equipped with a 
\emph{surjective} projective morphism $X\to Z$ in which case we denote the couple as $(X/Z,D)$ or $(X,D)\to Z$. 
We say a couple $(X/Z,D)$ is \emph{flat} if both $X\to Z$ and $D\to Z$ are flat.

Let $\mathcal{P}$ be a set of couples. We say $\mathcal{P}$ is \emph{generically relatively bounded} if 
there exist natural numbers $d,r$ such that for each $(X/Z,D)\in \mathcal{P}$ we have the following: $\dim X-\dim Z\le d$ and there is a very ample$/Z$ divisor $A$ on $X$ such that 
$$
\deg_{A/Z}A\le r ~~\mbox{and} ~~\deg_{A/Z}D\le r.
$$
If in addition all the $(X/Z,D)\in \mathcal{P}$ are flat, we say that $\mathcal{P}$ is \emph{relatively bounded}. 

When $D=0$ for every $(X/Z,D)\in \mathcal{P}$, we say $\mathcal{P}$ is a set of generically relatively  bounded (resp. relatively bounded) varieties.

\subsection{Morphisms of couples}\label{ss-tower-couples}

(1) A \emph{morphism} $(Z,E)\to (V,C)$ of couples is a morphism $f\colon Z\to V$ such that 
$f^{-1}(C)\subseteq E$. 

(2) Assume $(V',C')\to (V,C)$ is a morphism of couples with $V'\to V$ \emph{dominant}.   
Suppose in addition that $(Z,E)\to (V,C)$ 
is a morphism of couples such that over the generic point of $Z$ we have: $Z\times_{V}V'$ is integral 
and not contained in $Z\times_{V}C'$.  
We then define the \emph{pullback} of $(V',C')$ by base change to $(Z,E)$ as 
follows. Let $Z'$ be the main component of $Z\times_{V}V'$ and let $E'$ be the 
codimension one part, with reduced structure, of the union of the inverse images of $C'$ and $E$ under 
$Z'\to V'$ and $Z'\to Z$, respectively. Note that if $C'$ and $E$ are supports of 
effective Cartier divisors, then $E'$ coincides with the union of the inverse images of $C'$ and $E$, 
with reduced structure.


\subsection{Toric geometry}\label{ss-toric-geometry}
We will follow [\ref{Cox-etal}] for concepts and results in toric geometry.
A \emph{toric variety} is a variety $X$ of dimension $d$ containing a torus $\mathbb{T}_X$ (that is, isomorphic to $(k^*)^d$) as an open subset so that the action of $\mathbb{T}_X$ on itself (induced by coordinate-wise multiplication of $(k^*)^d$) extends to an action on the whole $X$ [\ref{Cox-etal}, 3.1.1]. Here, $X$ is not necessarily normal. 
A \emph{toric morphism} $f\colon X\to Y$ between toric varieties is a morphism so that the restriction $f|_{\mathbb{T}_X}$ induces a morphism $\mathbb{T}_X\to \mathbb{T}_Y$ of algebraic groups and so that $f$ is equivariant with respect to the actions of the tori.  

A normal toric variety $X$ of dimension $d$ can also be described in terms of a fan structure $\Sigma$ in $\R^d$ [\ref{Cox-etal}, 3.1.8]. Moreover, if $D_i$ are all the prime toric (i.e. torus-invariant) divisors on $X$, then to give a $\Q$-Cartier toric divisor $D=\sum d_iD_i$ on $X$ is the same as giving its \emph{support function} $\phi_D\colon |\Sigma|\to \R$ which is linear on each cone in $\Sigma$ and $\phi_D(u_i)=-d_i$ for the primitive vector $u_i$ generating the ray corresponding to $D_i$  [\ref{Cox-etal}, 4.2.12]. If $g\colon W\to X$ is a toric morphism from another normal toric variety with fan $\Gamma$, then $g^*D$ is the divisor determined by the support function $|\Gamma|\to |\Sigma|\overset{\phi_D}\to \R$ where the first map is induced by $g$ [\ref{Cox-etal}, 6.2.7]. 

Let $X$ be a $\Q$-factorial normal toric variety given by a fan $\Sigma$ with toric prime divisors $D_i$. Assume $B=\sum b_iD_i$ and that $K_X+B$ is $\Q$-Cartier. Pick a toric prime divisor $E$ over $X$. We are interested in the log discrepancy $a(E,X,B)$. Shrinking $X$, we can assume that it is affine, say given by a cone $\sigma$ whose rays correspond to the $D_i$. We can uniquely write $e=\sum \alpha_iu_i$ where $e,u_i$ are the primitive vectors corresponding to $E,D_i$. Then  
$$
a(E,X,B)=\sum \alpha_i(1-b_i)
$$  
by [\ref{Ambro-toric-mld}, \S2]. This can be seen by taking a toric resolution $g\colon W\to X$ and 
considering the support function of $B-\Lambda$ keeping in mind that $K_W+\Lambda_W=g^*(K_X+\Lambda)$ where $\Lambda_W$ and $\Lambda$ are the sum of all the toric prime divisors on $W$ and $X$, respectively.

\subsection{Formally Cartier divisors}\label{ss-formally-cartier}

Let $X$ be a variety, $x\in X$ be a closed point, and $\widehat{X}=\Spec \widehat{\mathcal{O}}_{X,x}$ where $\widehat{\mathcal{O}}_{X,x}$ denotes the completion of the local ring ${\mathcal{O}_{X,x}}$ with respect to its maximal ideal. The local ring ${\mathcal{O}}_{{X},x}$ is a G-ring (meaning Grothendieck ring) by [\ref{Matsumura}, Corollary and Remark 1 on page 259], so by definition of G-rings, the geometric fibres of $\widehat{X} \to \Spec {\mathcal{O}}_{{X},x}$ are regular: in the language of commutative algebra, this says that the homomorphism ${\mathcal{O}}_{{X},x}\to \widehat{\mathcal{O}}_{{X},x}$ is regular. 

Now assume $X$ is normal. Then $\widehat{\mathcal{O}}_{{X},x}$ is normal by the previous paragraph and [\ref{Matsumura}, Theorem 32.2] (or by [\ref{Zariski}]), hence $\widehat{X}$ is normal.
 Let $D$ be a Weil divisor on $X$. We define $\widehat{D}$ on $\widehat{X}$ as follows.  
Let $U$ be the smooth locus of 
$X$ and let $\widehat{U}$ be its inverse image in $\widehat{X}$, and $\pi\colon \widehat{U}\to U$ the induced 
morphism. Then $D|_U$ is Cartier and its pullback $\pi^*D|_U$ is a well-defined Cartier divisor. Now 
let $\widehat{D}$ be the closure of $\pi^*D|_U$ in $\widehat{X}$. Note that the complement of $\widehat{U}$ in $\widehat{X}$ has codimension at least two.

When $X$ is normal and $D$ is an effective Weil divisor on $X$, we can view $D$ as the closed subscheme of $X$ defined by the ideal sheaf $\mathcal{O}_X(-D)$ and think of $\widehat{D}$ as the corresponding closed subscheme of $\widehat{X}$, that is, if $D$ is given by an ideal $I$ near $x$, then $\widehat{D}$ is given by $\widehat{I}$.

\begin{lem}\label{l-formally-cartier}
Let $X$ be a normal variety, $x\in X$ be a closed point, and $\widehat{X}=\Spec \widehat{\mathcal{O}}_{X,x}$. 
Let $D$ be a Weil divisor on $X$ and let $\widehat{D}$ 
be the corresponding divisor on $\widehat{X}$. Then $D$ is Cartier near $x$ if and only if $\widehat{D}$ is Cartier. 
\end{lem}
\begin{proof}
If $D$ is Cartier near $x$, then  $\widehat{D}$ is Cartier. We show the converse. 
Shrinking $X$ and changing $D$ linearly, we can assume $D$ is effective, hence $\mathcal{O}_{X}(-D)\subset \mathcal{O}_X$.
Since $X$ is normal, $\mathcal{O}_{X}(-D)$ is a reflexive coherent sheaf. Since the morphism 
$\rho\colon \widehat{X}\to X$ is flat, $\rho^*\mathcal{O}_{X}(-D)$ is reflexive too [\ref{Hartshorne}, Proposition 1.8]. 
Moreover,  $\mathcal{O}_{\widehat{X}}(-\widehat{D})$ is reflexive, actually invertible, since $\widehat{D}$ is Cartier. 
Now as observed above, denoting the smooth locus of $X$ by $U$, $D|_U$ is Cartier and so is $\widehat{D}|_{\widehat{U}}$. Therefore, 
$\rho^*\mathcal{O}_{X}(-D)$ coincides with $\mathcal{O}_{\widehat{X}}(-\widehat{D})$ on $\widehat{U}$, hence $\rho^*\mathcal{O}_{X}(-D)$ and $\mathcal{O}_{\widehat{X}}(-\widehat{D})$ are equal as both are reflexive and as the complement of $\widehat{U}\subset \widehat{X}$ has codimension at least two [\ref{Hartshorne}, Proposition 1.6].  Thus $\rho^*\mathcal{O}_{X}(-D)$ is invertible, so applying [\ref{Matsumura}, Exercise 8.3] implies $\mathcal{O}_{X}(-D)$ 
is invertible near $x$, hence $D$ is Cartier near $x$.
\end{proof}

\subsection{Toroidal couples}

A couple $(X,D)$ is \emph{toroidal} at a closed point $x\in X$ 
if there exist a \emph{normal} toric variety $W$ and a closed point $w\in W$ such that there is 
a $k$-algebra isomorphism 
$$
\widehat{\mathcal{O}}_{{X},{x}}\to \widehat{\mathcal{O}}_{{W},w}
$$ 
of completion of local rings so that the ideal of $D$ is mapped to the ideal of the toric boundary divisor 
$C\subset W$ (that is, the complement of the torus). Then  
there is a common \'etale neighbourhood of $X,x$ and $W,w$ [\ref{Artin}, Corollary 2.6].
We call $(W,C),w$ a \emph{local toric model} of $(X,D),x$.
We say $(X,D)$ is toroidal if it is toroidal at every closed point.

Now let $f\colon (X,D)\to (Y,E)$ be a morphism of couples. Let $x\in X$ be a closed point 
and let $y=f(x)$. We say $(X,D)\to (Y,E)$ 
is a \emph{toroidal morphism} at $x$ if there exist local toric models $(W,C),w $ and $(V,B),v$ of 
$(X,D),x$ and $(Y,E),y$, respectively, and a toric morphism $W\to V$ of toric varieties 
inducing a commutative diagram 
$$
\xymatrix{
\widehat{\mathcal{O}}_{{X},{x}}\ar[r] & \widehat{\mathcal{O}}_{{W},w}\\
\widehat{\mathcal{O}}_{{Y},{y}} \ar[u] \ar[r] & \widehat{\mathcal{O}}_{{V},v} \ar[u]
}
$$
where the vertical maps are induced by the given morphisms and the horizontal maps are isomorphisms 
induced by the local toric models. We say the morphism of couples $f\colon (X,D)\to (Y,E)$ is toroidal 
if it is toroidal at every closed point. 

For a systematic treatment of toroidal couples, see [\ref{KKMB}]. 

\begin{lem}\label{l-toroidal-couple-lc}
Let $(X,D)$ be a toroidal couple. Then $X$ is normal and Cohen-Macaulay, $K_X+D$ is Cartier, and  
$(X,D)$ is an lc pair.
\end{lem}
\begin{proof}
Pick a closed point $x\in X$. Let $(W,C),w$ be a local toric model of $(X,D),x$. Since $W$ is toric and normal, it is Cohen-Macaulay.  
Thus $\widehat{\mathcal{O}}_{{W},w}$ is normal and Cohen-Macaulay, hence $\widehat{\mathcal{O}}_{{X},{x}}$ 
is normal and Cohen-Macaulay which implies $X$ is normal and Cohen-Macaulay at $x$, by [\ref{Bruns-Herzog}, Corollaries 2.1.8 and 2.2.23]. Alternative argument: ${\mathcal{O}}_{{X},x}$, ${\mathcal{O}}_{{W},w}$ are G-rings by [\ref{Matsumura}, Corollary on page 259], so by definition of G-rings, the homomorphisms ${\mathcal{O}}_{{X},x}\to \widehat{\mathcal{O}}_{{X},x}$ and ${\mathcal{O}}_{{W},w}\to \widehat{\mathcal{O}}_{{W},w}$ are regular; so by [\ref{Matsumura}, Theorem 32.2], 
${\mathcal{O}}_{{X},x}$ is normal (resp. regular, resp. Cohen-Macaulay, resp. reduced) iff $\widehat{\mathcal{O}}_{{X},x}$ is normal (resp. regular, resp. Cohen-Macaulay, resp. reduced) and a similar statement holds for ${\mathcal{O}}_{{W},w}$ and its completion.

Pulling back the canonical sheaf $\mathcal{O}_X(K_X)$ to $\Spec \widehat{\mathcal{O}}_{{X},{x}}$ gives the canonical sheaf of the latter [\ref{Bruns-Herzog}, Theorem 3.3.5]. In other words, $\widehat{K_X}$ is the canonical divisor of $\Spec \widehat{\mathcal{O}}_{{X},{x}}$ which is unique up to linear equivalence. Similarly, $\widehat{K_W}$ is the canonical divisor of $\Spec \widehat{\mathcal{O}}_{{W},{w}}$. Moreover, $(W,C)$ is toric, hence $K_W+C$ is Cartier near $w$. Thus using the given isomorphism $\widehat{\mathcal{O}}_{{X},{x}}\to \widehat{\mathcal{O}}_{{W},{w}}$ to identify the corresponding spaces, we deduce that $\widehat{K_X}+\widehat{D}\sim \widehat{K_W}+\widehat{C}$ is Cartier. Therefore, $K_X+D$ is Cartier near $x$, by Lemma \ref{l-formally-cartier}. 
Additionally, $(X,D)$ is lc at $x$ because $(W,C)$ is lc and because singularities are determined locally formally.
\end{proof}

We sketch an alternative approach to the second paragraph of the proof of the lemma. Applying [\ref{Artin}, Corollary 2.6], there is a common \'etale neighbourhood $U,u$ of $X,x$ and $W,w$. Assume that the inverse images of $D$ and $C$ to $U$ coincide near $u$ (this does not follow immediately from [\ref{Artin}, Corollary 2.6] but a modification of its proof should work; in this paper, when we apply the lemma, the condition on inverse images holds). Then one can see quickly that, near $x$, $K_X+D$ is Cartier and  
$(X,D)$ is an lc pair.

\section{\bf Toroidal and toric models of fibrations over curves}

We recall the main results of [\ref{B-toroid-toric-model-fib}]. The first result is bounded toroidalisation of fibrations over curves. Although its statement looks technical but it basically says that relatively bounded families over curves can be altered to relatively bounded and toroidal families.

\begin{thm}\label{t-bnd-torification}
Let $d,r\in \N$. Then there exists $r'\in \N$ depending only on $d,r$ satisfying the 
following. Assume that 
\begin{itemize}
\item $(X,D)$ is a couple of dimension $d$, 

\item $f\colon X\to Z$ is a projective morphism onto a smooth curve, 

\item $z\in Z$ is a closed point, and 

\item $A$ is a very ample$/Z$ divisor on $X$ such that $\deg_{A/Z}A\le r$ and $\deg_{A/Z}D\le r$.
\end{itemize}
Then, perhaps after shrinking $Z$ around $z$, there exists a commutative diagram 
$$
 \xymatrix{
 (X',D') \ar[d]^{f'} \ar[r]^{\pi} &    (X,D) \ar[d]^f \\
  (Z',E') \ar[r]^{\mu} & Z
  } 
$$  
of couples and a very ample$/Z'$ divisor $A'$ on $X'$ such that 
\begin{itemize}
\item  $(X',D')\to (Z',E')$ is a toroidal morphism which factors as a good tower 
$$
(X_d',D_d') \to \cdots \to (X_1',D_1')
$$
of families of split nodal curves,

\item  $\pi$ and $\mu$ are alterations,

\item 
$
\deg_{A'/Z'}A'\le r', \ \  \deg_{A'/Z'}D'\le r', \ \ \deg \pi\le r', \ \ \mbox{and} \ \ \deg \mu\le r',
$

\item the induced morphism 
$$
\pi|_{X'\setminus D'}\colon {X'\setminus D'}\to {X\setminus D}
$$ 
is quasi-finite,

\item $D'$ contains the fibre of $X'\to Z$ over $z$,

\item there is a Cartier divisor $G'\ge 0$ on $X'$ such that $A'-G'$ is ample$/Z'$ and $\Supp G'=D'$, and 

\item $A'-\pi^*A$ is ample over $Z'$.
\end{itemize}
\end{thm}

The theorem is proved in [\ref{B-toroid-toric-model-fib}] following the technique of families of nodal curves developed by de Jong [\ref{de-jong-smoothness-semi-stability}][\ref{de-jong-nodal-family}]. In this paper, we will use the fact that $(X',D')\to (Z',E')$ is a toroidal morphism but we will not use the fact that it factors as a good tower of families of split nodal curves. So we will not define families of nodal curves and good towers. For more details, see [\ref{B-toroid-toric-model-fib}].

The above theorem is used in [\ref{B-toroid-toric-model-fib}] to prove the following result which will be used to reduce problems in toroidal settings to toric settings.

\begin{thm}\label{p-local-desc-torif-bnd-fib-II}
Let $d,r$ be natural numbers. 
Assume  $(X,D)$ and $X\to Z$ satisfy the assumptions of Theorem \ref{t-bnd-torification} with the given $d,r$. 
Then we can choose 
$(X',D')$ and $X'\to Z'$ in the theorem so that if $x'\in X'$ is a closed point and $z'\in Z'$ is its image, then  
perhaps after shrinking $Z'$ around $z'$, we can find a commutative diagram of varieties and couples
$$
\xymatrix{
& {M}'\ar[ld] \ar[rd] & &N' \ar[ll]\ar[d]\\
(X',D') \ar[rd] & & ({Y}',{L}')\ar[ld] \ar@{-->}[r] & P'=\PP^{d-1}_{Z'}\ar[lld]\\
&(Z',E') &&
}
$$
where
\begin{enumerate}
\item all arrows are projective morphisms, except that $Y'\bir P'$ is a birational map,

\item $N'\to {M}'$ is birational and $N'\to P'$ is an alteration, 

\item $M'\to X'$ and $M'\to Y'$ are \'etale at some closed point $m'$ mapping to $x'$, 

\item the inverse images of $D'$ and $L'$ to $M'$ coincide near $m'$,

\item  if $G'$ is the sum of the coordinate hyperplanes of $\PP^{d-1}_{Z'}$ and the inverse image of $E'$, then the induced map $P'\setminus G'\bir Y'$ is an open immersion,

\item $(Y',L')$ is lc near $y'$, the image of $m'$, and any lc place of $(Y',L')$ with centre at $y'$ is an lc place of $(P',G')$, and

\item there is an ample$/Z'$ Cartier divisor $H'$ on ${Y}'$ such that 
$$
\vol_{/Z'}(A'|_{N'}+H'|_{N'}+G'|_{N'})\le r'
$$ 
where $A',r'$ are as in Theorem \ref{t-bnd-torification}.
\end{enumerate}
\end{thm}

It is worth noting that the semi-stable reduction method of [\ref{KKMB}] would not work to prove Theorem \ref{t-bnd-torification} because although it produces a toroidal model but it would be far from relatively bounded, which is crucial for us. However, more general (weak) semi-stable reductions have been developed in recent years, e.g. [\ref{Qu}] (which is applied in [\ref{BQ}]), relying on log geometry. It is likely that this can used to get another proof of Theorems \ref{t-bnd-torification} and \ref{p-local-desc-torif-bnd-fib-II} but it would not be straightforward and more work is required.


\section{\bf Strict transform of divisors}

In this section we introduce strict transform of divisors on a variety to another variety essentially via a correspondence. This generalises the notion of birational transform for birational maps. We will use it to reduce problems in the setting of toroidal morphisms to the setting of toric morphisms.

\subsection{Strict transform of reduced divisors}\label{ss-strict-transform} 
Let $X$ and $Y$ be varieties of the same dimension equipped with surjective projective morphisms
 $X\to Z$ and $Y\to Z$ onto a base variety $Z$. 
Assume we are given a variety $V$ fitting into a commutative diagram 
$$
\xymatrix{
&V\ar[ld]_\phi \ar[rd]^\psi&\\
X\ar[rd] && Y\ar[ld]\\
&Z&}
$$
where $\phi, \psi$ are alterations. 
Now let $D$ be a reduced divisor on $X$. We define the \emph{strict transform} $D_Y\subset Y$ of $D$ using the above diagram as follows. 
First, we define $\phi^{[-1]}D$ to be the union of those prime divisors on $V$ which map onto some irreducible component of $D$.
Then we define the strict transform $D_Y$ on $Y$ to be the codimension one part of $\psi(\phi^{[-1]}D)$
with reduced structure. 

\begin{lem}\label{l-strict-transform-well-defined}
Under the above notation and assumptions, if $\rho \colon U\to V$ is an alteration, then $\rho^{[-1]}(\phi^{[-1]}D)=(\phi\rho)^{[-1]}D$.
In particular, $D_Y$ is the same if defined using the induced 
diagram
$$
\xymatrix{
&U\ar[ld]_{\phi\rho} \ar[rd]^{\psi\rho}&\\
X\ar[rd] && Y\ar[ld]\\
&Z.&}
$$
\end{lem} 
\begin{proof}
Pick an irreducible component $S$ of $\rho^{[-1]}(\phi^{[-1]}D)$. Then $S$ maps onto an irreducible component $T$ of $\phi^{[-1]}D$. In turn, $T$ maps onto an irreducible component of $D$. Thus, $S$ is a component of $(\phi\rho)^{[-1]}D$. 
Conversely, assume that $S$ is an irreducible component of $(\phi\rho)^{[-1]}D$. Then it maps onto an irreducible component of $D$, hence the image of $S$ on $V$, say $T$, is a prime divisor mapping onto an irreducible component of $D$, so $T$ is a component of $\phi^{[-1]}D$. This in turn implies that $S$ is a component of  $\rho^{[-1]}(\phi^{[-1]}D)$. Therefore, $\rho^{[-1]}(\phi^{[-1]}D)=(\phi\rho)^{[-1]}D$. The last claim about $D_Y$ then follows from the fact that $\rho(\rho^{[-1]}(\phi^{[-1]}D))=\phi^{[-1]}D$.
\end{proof}

Although the lemma shows that certain choices of diagrams give the same strict transform $D_Y$ 
but in general two arbitrarily chosen diagrams lead to different strict transforms. For example, 
assume $V,X,Y$ are all $\PP^1$, $Z$ is a point, and $D$ is a point on $X$. 
If $\phi\colon V\to X$ is of degree $n$ and \'etale over $D$ and if $\psi \colon V\to Y$ is the identity morphism, 
then the strict transform $D_Y=\phi^{-1}(D)$ has degree $n$. Since $n$ can be arbitrary, the strict transform is not unique. 

If $X,Y$ are normal and $\phi,\psi$ are birational, then $D_Y$ is the birational transform of $D$.

\subsection{Degree of strict transform of divisors}

Next we want to consider settings in which we can control degree of strict transform of divisors. 

\begin{lem}\label{l-strict-transform-bnd-degree}
Consider a diagram as in \ref{ss-strict-transform} and assume that  
\begin{itemize}
\item  $A,H$ are nef$/Z$ Cartier divisors on $X,Y$, respectively, and 

\item ${B}=\sum b_i {B}_i$ is an effective $\R$-divisor on ${X}$ with 
$|{A}-{B}|_{\R/Z}\neq \emptyset$. 
\end{itemize} 
Then 
$$
\deg_{H/Z} \sum b_i({B}_i)_{{Y}}\le \vol_{/Z}(A|_V+H|_V)
$$ 
where $(B_i)_{{Y}}$ is the strict transform of ${B}_i$.

\end{lem}
\begin{proof}
Since $|{A}-{B}|_{\R/Z}\neq \emptyset$, we can write $A\sim_\R B+G/Z$ for some $G\ge 0$. 
Replacing $B$ with $B+G$, we can assume $A\sim_\R B/Z$ which in particular means we can assume $B$ is 
$\R$-Cartier. On the other hand, we can assume that $V\to X$ factors through the normalisation of $X$, and then we can replace $X$ with its normalisation and replace $A,B$ with their pullbacks. 

Now $\sum b_i \phi^{[-1]}B_i\le \phi^*B$ which can be checked locally over each irreducible component of $B$ keeping in mind that $X$ is smooth near the generic point of the $B_i$. On the other hand, the codimension one part of $\psi(\phi^{[-1]}B_i)$ is less than or equal to $\psi_*\phi^{[-1]}B_i$ where $\psi_*$ stands for cycle-theoretic direct image (cf. [\ref{Fulton}, Section 1.4] for definition of direct image). Then
$$
\sum b_i({B}_i)_{{Y}}\le \sum b_i \psi_*\phi^{[-1]}B_i=\psi_*\sum b_i \phi^{[-1]}B_i \le \psi_*\phi^*B.
$$ 
 Thus it is enough to show that
$$
\deg_{H/Z}\psi_*\phi^*B\le \vol_{/Z}(A|_V+H|_V). 
$$ 
This is equivalent to showing 
$$
\deg_{H/Z}\psi_*\phi^*A\le \vol_{/Z}(A|_V+H|_V)
$$
because $\psi_*\phi^*B\sim_\R\psi_*\phi^*A$ over the generic point of $Z$ (this $\R$-equivalence can be derived from [\ref{Fulton}, Theorem 1.4]). 

Let $z\in Z$ be a general closed point and let $S,T$ be the fibres of $V\to Z$ and $Y\to Z$ over $z$. 
Let $d=\dim S=\dim T$ and $c=\dim Z$. Perhaps after shrinking $Z$ around $z$, we can find effective Cartier divisors $L_1,\dots,L_c$ on $Y$ such that $T=L_1\cap \cdots \cap L_c$ where the intersections are scheme-theoretic.
Then 
$$
T\cdot H^{d-1}\cdot \psi_*\phi^*A\sim L_1\cdot \cdots \cdot L_c \cdot H^{d-1} \cdot \psi_*\phi^*A
$$
$$
\sim \psi_*(\psi^*L_1\cdot \cdots \cdot \psi^*L_c \cdot \psi^*H^{d-1}\cdot  \phi^*A)\sim 
\psi_*(S\cdot \psi^*H^{d-1}\cdot  \phi^*A)
$$ 
by applying the cycle-theoretic projection formula [\ref{Fulton}, Proposition 2.5(c)] repeatedly, where $\sim$ denotes rational equivalence of zero-cycles.
Therefore, we have  
$$
\deg_{H/Z}\psi_*\phi^*A=\deg(H|_T^{d-1} \cdot (\psi_*\phi^*A)|_T)
=\deg(T\cdot H^{d-1}\cdot \psi_*\phi^*A)
$$
$$
=\deg(S\cdot \psi^*H^{d-1}\cdot  \phi^*A)=\deg(\psi^*H|_S^{d-1}\cdot \phi^*A|_S)
$$
$$
\le  \vol(\psi^*H|_S+\phi^* A|_S)=\vol_{/Z}(H|_V+A|_V).
$$ 
To get the third equality, we are using 
the fact that we are working over an algebraically closed field so degree of a zero-cycle and degree of its direct image are the same. To get the inequality, we are using the assumption that $A,H$ are nef over $Z$.
\end{proof}

\subsection{Log discrepancies after strict transform}\label{ss-log-disc-strict-transform}
We will define a kind of strict transform of boundary divisors under suitable conditions, and then 
compare log discrepancies.  First, we fix some notation. Let $(W,B_W)$ 
be a pair and let $\pi\colon W'\to W$ be an alteration from a normal variety. Then 
we can write 
$$
K_{W'}+B_{W'}=\pi^*(K_W+B_W)
$$ 
where $B_{W'}$ is uniquely determined by choosing $K_{W'}$ to be compatible with $K_W$. This follows from the Riemann-Hurwitz formula if $\pi$ is finite. It is obvious when $\pi$ is birational. The general case is reduced to these two cases. We will refer to  
$B_{W'}$ as \emph{the sub-boundary given by crepant pullback of $K_W+B_W$}. 

Let $D$ be a prime divisor over $W$ and $D'$ be a prime divisor over $W'$ mapping to $D$. If the ramification index along $D'$ is $r$ (the ramification index can be computed on appropriate resolutions of $W',W$), then the log discrepancies satisfy 
$$
a(D',W',B_{W'})=ra(D,W,B_W)
$$
(cf. [\ref{kollar-mori}, proof of Proposition 5.20] or [\ref{Shokurov-log-flips}]).
In particular, if $D$ is a component of $B_W$ and if $\mu_DB_W\ge \frac{e-1}{e}$, where $e=\deg \pi$, then one can check that $\mu_{D'}B_{W'}\ge 0$ as $r\le e$. On the other hand, if $(W,B_W)$ is $\epsilon$-lc, then $(W',B_{W'})$ is sub-$\epsilon$-lc.

\begin{prop}\label{p-disc-after-str-transform}
Assume that 
\begin{itemize}
\item we have a commutative diagram of normal varieties
$$
\xymatrix{
&{U}\ar[ld]_{{\phi}} \ar[rd]^{{\psi}}&\\
X\ar[rd]_f && Y\ar[ld]^g\\
&Z&}
$$ 
where $f,g,\phi,\psi$ are surjective projective morphisms, $Z$ is a curve,

\item $\phi,\psi$ are \'etale at a closed point $u\in U$ mapping to $x\in X$ and $y\in Y$, 

\item $(W,B_W)$ is an $\epsilon$-lc pair ($\epsilon>0$) and $W\to X$ is a birational contraction, and 

\item $K_W+B_W$ is nef$/X$.
\end{itemize}
Then there exist an $\epsilon$-lc pair $(V,C_{V})$ and a birational contraction $V\to Y$ such that 
\begin{itemize}

\item if $B=\sum b_iB_i$ is the pushdown of $B_W$ to $X$ and $C_Y$ is the pushdown of $C_V$ 
to $Y$, then 
$$
C_Y=(\sum b_i (B_i)_{Y})^{\le 1-\epsilon}
$$ 
where the strict transforms are defined with respect to the diagram above, 

\item $K_{V}+C_{V}$ is nef$/Y$, and 

\item if $T$ is any prime divisor over $X$ whose centre passes through $x$ 
and if $S$ is a corresponding prime divisor over $Y$ (induced by the diagram above), then 
$$
a(T,W,B_W)\ge a(S,V,C_{V}).
$$
\end{itemize}
\end{prop}
\begin{proof}
\emph{Step 1.}
Let 
$$
C_Y:=(\sum b_i (B_i)_{Y})^{\le 1-\epsilon}
$$
where the notation $()^{\le 1-\epsilon}$ is as in \ref{ss-divs-deg}.
Let $Y'\to Y$ be a log resolution of $(Y,C_Y)$. Let 
$$
C_{Y'}=C_Y^\sim+(1-\epsilon)E
$$ 
where $C_Y^\sim$ is the birational transform of $C_Y$ and 
$E$ is the reduced exceptional divisor of $Y'\to Y$. Then $(Y',C_{Y'})$ is $\epsilon$-lc. Run an MMP$/Y$ on $K_{Y'}+C_{Y'}$ with scaling of some ample divisor. This terminates as $K_{Y'}+C_{Y'}$ is big over $Y$, by [\ref{BCHM}]. Let $V$ be the minimal model obtained. By construction, $(V,C_{V})$ is $\epsilon$-lc 
and $K_{V}+C_{V}$ is nef$/Y$.

\emph{Step 2.}
By assumption, $\phi,\psi$ are \'etale on some open neighbourhood $U'$ of $u$.
Let $\overline{W}$ be the normalisation of the main component of $W\times_X{U}$. 
Similarly define $\overline{V}$ to be the  
normalisation of the main component of $V\times_Y{U}$. 
Let ${W}'$ be the inverse image of 
${U'}$ under $\overline{W}\to {U}$ and similarly let ${V}'$ be the inverse image of 
${U'}$ under $\overline{V}\to {U}$. 
Since ${U}'\to X$ and ${U}'\to Y$ are \'etale, ${W}'\to W$ and 
${V}'\to V$ are also \'etale. Since $\overline{W}\to {U}$ and $\overline{V}\to {U}$ are both birational,  
we can take a common resolution $P\to \overline{W}$ and $P\to \overline{V}$. 
Let ${P'}$ be the inverse image of ${U'}$ under $P\to {U}$. 

Let $B_P$ and $B_{\overline{W}}$ be the sub-boundaries given 
by crepant pullback of $K_W+B_W$ to $P$ and $\overline{W}$. Similarly, 
let $C_P$ and $C_{\overline{V}}$ be the sub-boundaries given 
by crepant pullback of  $K_{V}+C_{V}$ to $P$ and $\overline{V}$.
Since $\overline{W}\to W$ is \'etale on ${W}'$, we deduce that 
$B_{\overline{W}}|_{W'}$ is effective and its coefficients are 
among the coefficients of $B_W$. Similarly, 
$C_{\overline{V}}|_{{V}'}$ is effective and its coefficients are 
among the coefficients of $C_{V}$.

\emph{Step 3.}
We claim that $B_P\le C_P$  holds on ${P'}$. 
Since $K_W+B_W$ is nef$/X$, $K_P+B_P$ is nef$/\overline{V}$, hence from 
$$
C_P-B_P=K_P+C_P-(K_P+B_P)\equiv -(K_P+B_P)/\overline{V},
$$
we see that $C_P-B_P$ is anti-nef$/\overline{V}$. Thus by the negativity lemma, it is enough to show that 
$\pi_*C_P-\pi_*B_P\ge 0$ holds on $V'$ 
where $\pi$ denotes $P\to \overline{V}$. 

Let $Q$ be a prime divisor on $\overline{V}$ intersecting $V'$. We want to show that 
$\mu_Q\pi_*C_P\ge \mu_Q\pi_*B_P$. Let $R$ be the birational transform of $Q$ on $P$. Equivalently, we want to show that 
$\mu_RC_P\ge \mu_RB_P$.
Since $V'\to V$ is \'etale, we see that  
$Q$ (and so $R$) maps to a divisor $D$ on $V$. 
First assume $Q$ is exceptional over ${U}$. Then $D$ is exceptional over $Y$, hence 
by construction, $\mu_DC_V=1-\epsilon$ which in turn 
implies $\mu_RC_P=\mu_QC_{\overline{V}}=1-\epsilon$ as $V'\to V$ is \'etale. On the other hand, 
since $(W,B_W)$ is $\epsilon$-lc, $(P,B_P)$ is sub-$\epsilon$-lc, so $\mu_RB_P\le 1-\epsilon$. Thus $\mu_RC_P\ge \mu_RB_P$.

Now assume $Q$ is not exceptional over ${U}$. Then $Q$ (hence $R$) is not exceptional over either $X,Y$ 
because the image of $Q$ on $U$ intersects ${U'}$ and ${U}'\to X$ and 
${U'}\to Y$ are both \'etale. This implies that $P\to W$ and $P\to V$ are both \'etale 
at the generic point of $R$. Thus if $M,N$ are the images of $R$ on $X,Y$, respectively, then 
we have 
$$
\mu_RB_P=\mu_MB~~~ \mbox{and}~~~ \mu_RC_P=\mu_{N}C_Y.
$$ 
Thus if $M$ is not a component of $B$, 
then $\mu_RC_P\ge \mu_RB_P$ holds trivially. 

Assume then that $M$ is a component of $B$, say 
with coefficient $b>0$. Then $bM_Y\le C_Y$ by definition of $C_Y$ and because $b\le 1-\epsilon$. 
Let $L$ be the image of $R$ on $U$. Then $L$ is a component of $\phi^{[-1]}M$ as $L$ maps onto $M$, hence $N$ is a component of $M_Y$ as $L$ maps onto $N$.
Therefore, $N$ is a component of $C_Y$ 
with coefficient at least $b$. This implies that 	
$$
\mu_RB_P=\mu_MB=b\le \mu_{N}C_Y=\mu_RC_P.
$$
Therefore, we have shown that 
$\pi_* C_P\ge \pi_*B_P$ holds on $V'$ as required.

\emph{Step 4.}
Finally, let $T$ be a prime divisor over $X$ with centre passing through $x$. 
Let $\tilde{W}\to W$ and $\tilde V\to V$ be resolutions and assume $T$ is a divisor on $\tilde{W}$. 
Replacing $P$, we can assume both induced maps $P\bir \tilde{W}$ and $P\bir \tilde{V}$ are morphisms.
Since $U\to X$ is \'etale at $u$ and since $u$ maps to $x$, there is a prime divisor $R$ on $P$ 
whose centre on $U$ passes through $u$ and $R$ maps onto $T$. Replacing $P$ and $\tilde{V}$, 
we can assume $R$ maps onto a divisor $S$ on $\tilde{V}$ whose centre on $Y$ passes through $y$, by Lemma \ref{l-desced-prime-div}.

Moreover, 
$P\to \tilde{W}$ factors through $\tilde{W}\times_XU\to \tilde{W}$ which is \'etale on the open subset 
$\tilde{W}\times_XU'$, and $R$ is mapped to a divisor on $\tilde{W}\times_XU$ intersecting 
$\tilde{W}\times_XU'$. Thus $P\to \tilde{W}$ is \'etale near the generic 
point of $R$ and $P'$ intersects $R$. Similarly, $P\to \tilde{V}$ is \'etale near the generic 
point of $R$. In particular, 
$$
a(T,W,B_W)=a(T,\tilde{W},B_{\tilde{W}})=a(R,P,B_P)=1-\mu_RB_P
$$
and 
$$
a(S,V,C_{V})=a(S,\tilde{V},C_{\tilde{V}})=a(R,P,C_P)=1-\mu_RC_P
$$
where $B_{\tilde{W}}$ and $C_{\tilde{V}}$ are the sub-boundaries determined by the pullbacks of 
$K_W+B_W$ and $K_{V}+C_V$ to $\tilde{W}$ and $\tilde{V}$, respectively.

Now since $R$ intersects $P'$, by Step 3 we have 
$$
\mu_RB_P \le \mu_RC_P
$$
proving the claim 
$$
a(T,W,B_W)=1-\mu_RB_P\ge 1-\mu_RC_P=a(S,V,C_{V}).
$$
\end{proof}


\section{\bf Multiplicities along lc places}

In this section we study multiplicities of fibres along lc places in a quite general setting. 
We work towards proving Theorem \ref{t-bnd-mult-lc-places-fib-main}. 
First, we will reduce the theorem to the case when $(X,\Lambda)$ is toroidal. Next, we will reduce to the case when $X=\PP^{d-1}_Z$ and $\Lambda$ is the sum of the coordinate hyperplanes and the fibre over $z$.  In the next section we will further reduce the theorem to a toric problem, then solve the problem, hence complete the proof of the theorem.

\subsection{Reduction to the toroidal case}
We start with using Theorem \ref{t-bnd-torification} to make $(X,\Lambda)$ toroidal over $Z$.

\begin{prop}\label{p-reduction-to-toroidal}
Assume Theorem \ref{t-bnd-mult-lc-places-fib-main} holds when $(X,\Lambda)\to (Z,E)$ is toroidal where $E=z$. Then the theorem holds in general. 
\end{prop}
\begin{proof}
\emph{Step 1.}
Since $a(T,W,B_W)\le 1$, we may extract $T$, hence assume $T$ is a prime divisor on $W$.
Replacing $A$ with $rA$ (and replacing $r$ with $r^{d}$), we can assume $A-r\Lambda$ is pseudo-effective over $Z$, so 
$$
\deg_{A/Z}r\Lambda\le \deg_{A/Z}A\le r.
$$
Now applying Theorem \ref{t-bnd-torification} to $(X,D:=\Supp r\Lambda),A, X\to Z$, and $z\in Z$, we get a commutative diagram 
$$
 \xymatrix{
 (X',D') \ar[d] \ar[r]^{\pi} &    (X,D) \ar[d] \\
  (Z',E') \ar[r] & Z
  } 
$$  
satisfying the properties listed in that theorem, perhaps after shrinking $Z$ around $z$. We will use the notation 
of that theorem. In particular, $(X',D')\to (Z',E')$ is toroidal and $\deg \pi\le r'$ where we can assume $r'\ge 2$.

\emph{Step 2.}
We will modify $(W,B_W)$ so that  the coefficient of $T$ in $B_W$, say $b$, is at least $\frac{r'-1}{r'}$. 
Assume $b<\frac{r'-1}{r'}$. Let $K_{W}+\Lambda_{W}$ be the pullback of $K_{X}+\Lambda$. 
 The coefficient of $T$ in $\Lambda_{W}$ is $1$ because 
 $a(T,X,\Lambda)=0.$
Consider 
$$
\Delta_W=\alpha B_W+(1-\alpha)\Lambda_W
$$ 
where $\alpha\in (0,1)$ is the number so that 
the coefficient of $T$ in $\Delta_W$ is $\frac{r'-1}{r'}$. More precisely, $\alpha=\frac{1}{r'(1-b)}$ which is 
bounded from below by $\frac{1}{r'}$. 
Moreover, $(W,\Delta_W)$ is sub-$\alpha \epsilon$-lc. 

In general, $\Delta_W$ may have negative coefficients but we will remedy this as follows. 
Take a log resolution $V\to W$, let $K_V+\Delta_V$ be the 
pullback of $K_{W}+\Delta_{W}$, and write $\Delta_V=\Omega_V-J_V$ where $\Omega_V,J_V\ge 0$ have no common component (so $\Omega_V$ is the effective part of $\Delta_V$).
Since $(W,\Delta_W)$ is sub-$\alpha \epsilon$-lc, the coefficients of $\Omega_V$ do not exceed $1-\alpha\epsilon$, 
hence  $(V,\Omega_V)$  is $\alpha\epsilon$-lc.
Let  $(V',\Omega_{V'})$ be a minimal model of $(V,\Omega_V)$ over $X$. 
Note that since $K_V+\Omega_V=K_V+\Delta_V+J_V$ and since $K_V+\Delta_V$ is nef$/X$, any divisor contracted by $V\bir V'$ is a component of $J_V$. In particular, 
since $T$ has positive coefficient in $\Delta_W$, its birational transform is not a component of $J_V$, hence not contracted over $V'$.  Then 
$(V',\Omega_{V'})$ is $\alpha\epsilon$-lc and 
$$
a(T,V',\Omega_{V'})=a(T,V,\Omega_{V})=a(T,W,\Delta_{W})=\frac{1}{r'} \le 1.
$$ 
Also the pushdown of $\Omega_{V'}$ to $X$ is $\alpha B+(1-\alpha)\Lambda$.

Now replace $\epsilon$ with $\alpha\epsilon$ and   
$(W,B_{W})$ with $(V',\Omega_{V'})$ so that we can assume $\mu_TB_W\ge \frac{r'-1}{r'}$. Note that the condition of $A-B$ being pseudo-effective over $Z$ is preserved. 

\emph{Step 3.}
Let $W'$ be the normalisation of the main component of $X'\times_XW$. 
Let $T'$ be a prime divisor on $W'$ mapping onto $T$. Since $W\to X$ and $W'\to X'$ 
are birational, 
$$
\deg(W'\to W)=\deg \pi\le r'.
$$ 
Then  the ramification index of $W'\to W$ along $T'$ is at most $r'$. 
Thus if $K_{W'}+B_{W'}$ is the pullback of $K_W+B_W$, then the coefficient of $T'$ in $B_{W'}$ is 
non-negative because the coefficient of $T$ in $B_W$ is at least $\frac{r'-1}{r'}$ (see \ref{ss-log-disc-strict-transform}). 

In general, $B_{W'}$ is not effective. Let $\Theta_{W'}$ be the effective part of $B_{W'}$. 
The idea is to construct a suitable model over $X'$ using $\Theta_{W'}$ but first we need to establish 
some properties of this divisor.
Let $\Theta'$ be the pushdown of $\Theta_{W'}$ to $X'$ and write $\Theta'=C_1'+C_2'$ where $C_2'$ 
is the exceptional$/X$ part of $\Theta'$ and $C_1'$ is the non-exceptional part. 
Since $D'$ is 
reduced and $\deg_{A'/Z'} D'\le r'$, we see that $|pA'-D'|_{\R/Z'}\neq \emptyset$ for some fixed natural number $p$ depending only on $d,r'$. Now 
$C_2'\le D'$ as $X'\setminus D'\to X\setminus D$ is quasi-finite by Theorem \ref{t-bnd-torification}, hence
$|pA'-C_2'|_{\R/Z'}\neq \emptyset$.

On the other hand, writing $B=\sum b_iB_i$ where $B_i$ are the irreducible components (recall $B$ is the pushdown of $B_W$), we claim that 
$$
C_1'\le \sum b_i\pi^{[-1]}B_i.
$$ 
Assume $S'$ is a component of $C_1'$, say with coefficient $b'$. Then its birational transform on $W'$, which we still denote by $S'$,  
is a component of $B_{W'}$ with coefficient $b'$. So $S'$ maps onto a component $S$ of $B_W$ which is the 
birational transform of a component of $B$, say with coefficient $b$ which we again denote by $S$. Thus 
$$
b'=\mu_{S'}\Theta'=\mu_{S'}B_{W'}\le \mu_SB_W=\mu_SB=b
$$  
where the inequality follows from the discussion of \ref{ss-log-disc-strict-transform}. This implies $b'S'\le  b\pi^{[-1]}S$, and proves the claim $C_1'\le \sum b_i\pi^{[-1]}B_i$. 

Now since $A-B$ is pseudo-effective over $Z$, $|2A-B|_{\R/Z}\neq \emptyset$, so we have $2A\sim_\R B+G/Z$ where $G$ is effective. 
Thus over $Z'$ we have 
$$
\pi^*2A\sim_\R \pi^*(B+G)\ge \sum b_i\pi^{[-1]}B_i 
$$
giving 
$$
|\pi^*2A-\sum b_i\pi^{[-1]}B_i |_{\R/Z'} \neq \emptyset
$$ 
which in turn implies 
$$
|\pi^*2A-C_1'|_{\R/Z'} \neq \emptyset
$$ 
by the previous paragraph. 
By Theorem \ref{t-bnd-torification}, $A'-\pi^*A$ is ample over $Z'$ from which we get 
$|2A'-C_1'|_{\R/Z'} \neq \emptyset$.
Therefore, replacing $p$, we can assume $|pA'-\Theta'|_{\R/Z'} \neq \emptyset$.

\emph{Step 4.}
Consider a log resolution $U'\to W'$, 
let $K_{U'}+B_{U'}$ be the pullback of $K_{W'}+B_{W'}$, and let $\Theta_{U'}$ be the effective part of $B_{U'}$. 
Then $(U',\Theta_{U'})$ is $\epsilon$-lc as $(U',B_{U'})$ is sub-$\epsilon$-lc by \ref{ss-log-disc-strict-transform}.
Next, let $(U'',\Theta_{U''})$ be a minimal model of $(U',\Theta_{U'})$ over $X'$. Then $T'$ 
is not contracted over $U''$ as $K_{U'}+B_{U'}$ is nef over $X'$ and as $T'$ is not a component of $\Theta_{U'}-B_{U'}$ which in turn follows from  the fact that the coefficient of $T'$ in $B_{W'}$ is non-negative. In particular,
$$
a(T',U'',\Theta_{U''})\le 1.
$$

On the other hand, let $K_{X'}+\Lambda'$ be the pullback of $K_X+\Lambda$. Then any component of 
$\Lambda'$ with positive coefficient is either exceptional$/X$  
 or it maps onto a component of $\Lambda$ with positive coefficient. 
Thus by Theorem \ref{t-bnd-torification} such a component is a component of $D'$, hence $\Lambda'\le D'$. 
Therefore, 
$$
0\le a(T',X',D')\le a(T',X',\Lambda')=0
$$
where the equality follows from the assumption $a(T,X,\Lambda)=0$: indeed, the coefficient of 
$T$ in $\Lambda_W$ is $1$, hence the coefficient of $T'$ in $\Lambda_{W'}$ is also $1$ where 
$K_{W'}+\Lambda_{W'}$ is the pullback of $K_W+\Lambda_W$ which is the same as the pullback 
of $K_{X'}+\Lambda'$. 

Finally, let $F,G'$ be the fibres of $X\to Z$ and $W'\to Z$ over $z$, and let $F'$ be the fibre 
of $X'\to Z'$ over $z'$ where $z'$ is the image of $T'$ on $Z'$. Then 
$$
\mu_TF\le \mu_{T'}G'\le \deg (Z'\to Z)\mu_{T'}F'\le r' \mu_{T'}F'
$$ 
which means it is enough to bound $\mu_{T'}F'$. Therefore, we can replace 
$$
(W,B_W), (X,\Lambda), Z, A, T, r,
$$ 
with 
$$
(U'',\Theta_{U''}), (X',D'), Z', pA', T', p^{d-1}r',
$$ 
so we are reduced to the toroidal case.
\end{proof}

\subsection{Reduction to the relative projective space case}
Next, we will use Theorem \ref{p-local-desc-torif-bnd-fib-II} to reduce Theorem \ref{t-bnd-mult-lc-places-fib-main} to the relative projective space case.

\begin{prop}\label{p-reduction-to-proj-space}
Assume Theorem \ref{t-bnd-mult-lc-places-fib-main} holds in dimension $\le d-1$. Also assume the theorem holds when $X=\PP^{d-1}_Z$, $\Lambda$ is the sum of the 
coordinate hyperplanes  plus the fibre over $z$, and $X\to Z$ is projection. Then the theorem holds in dimension $d$ in general. 
\end{prop}
\begin{proof}
\emph{Step 1.}
We can assume $T$ is a divisor on $W$. 
Let $x$ be the generic point of the centre of $T$ on $X$. Assume the closure $\bar{x}$ is positive dimensional. 
Let $H\sim A/Z$ be general and let $G$ be its pullback to $W$. Let 
$$
K_H+\Lambda_H=(K_X+\Lambda+H)|_H \ \mbox{and} \ K_G+B_G=(K_W+B_W+G)|_G
$$ 
and $A_H=A|_H$. Then $(G,B_G)$, $(H,\Lambda_H)$, $A_H$, and $H\to Z$ satisfy the properties listed in Theorem \ref{t-bnd-mult-lc-places-fib-main}, that is, $(G,B_G)$ is $\epsilon$-lc, $(H,\Lambda_H)$ is lc, $K_G+B_G$ is nef over $H$, etc. 
Thus applying the induction hypothesis we deduce that the multiplicity of any component of $T\cap G$ in the fibre of $G\to Z$ over $z$ is bounded from above. Thus the multiplicity of $T$ in the fibre of $W\to Z$ over $z$ is bounded from above. We can then assume $x$ is a closed point. 

\emph{Step 2.}
By Proposition \ref{p-reduction-to-toroidal}, we can assume $(X,\Lambda)\to (Z,E)$ is toroidal where $E=z$.
In the proof of Proposition \ref{p-reduction-to-toroidal} we applied Theorem \ref{t-bnd-torification} to reduce to the toroidal situation. Using Theorem \ref{p-local-desc-torif-bnd-fib-II}, we can assume that stronger conditions hold in our setting. That is, if we apply \ref{p-local-desc-torif-bnd-fib-II} already in the proof of \ref{p-reduction-to-toroidal}, 
then perhaps after shrinking $Z$ around $z$, we can assume that there exists a commutative diagram 
$$
\xymatrix{
&M\ar[ld] \ar[rd] & &N \ar[ll]\ar[d]\\
(X,\Lambda) \ar[rd] & & (Y,L)\ar[ld] \ar@{-->}[r] & P=\PP^{d-1}_{Z}\ar[lld]\\
&(Z,E) &&
}
$$
satisfying the properties listed in  \ref{p-local-desc-torif-bnd-fib-II} with primes $'$ removed and 
$D$ replaced with $\Lambda$. For convenience, we recall those properties:
\begin{enumerate}
\item all arrows are projective morphisms, except $Y\bir P$ which is birational,

\item $N\to M$ is birational and $N\to P$ is an alteration, 

\item $M\to X$ and $M\to Y$ are \'etale at some closed point $m$ mapping to $x$, 

\item the inverse images of $\Lambda$ and $L$ to $M$ coincide near $m$,

\item  if $G$ is the sum of the coordinate hyperplanes of $P$ and the fibre over $E=z$, 
then the induced map $P\setminus G\bir Y$ is an open immersion,

\item $(Y,L)$ is lc near $y$, the image of $m$, and any lc place of $(Y,L)$ with centre at $y$ is an lc place of $(P,G)$, and
\item there is an ample$/Z$ Cartier divisor $H$ on $Y$ such that  
$$
\vol_{/Z}(A|_{N}+H|_N+G|_{N})\le r.
$$ 
\end{enumerate}

Since $(Y,L)$ is lc near $y$, $Y$ is normal at $y$. Moreover, $M$ is normal at $m$ as $M\to X$ (and $M\to Y$) is \'etale at $m$. Taking the normalisations of $M,Y,N$, we can assume they are all normal varieties. This does not affect the properties listed above. 

Also since $A-B$ is pseudo-effective over $Z$ by assumption (recall that $B$ on $X$ is the pushdown of $B_W$), 
replacing $A$ with $2A$ and replacing $r$ accordingly, we can assume $|A-B|_{\R/Z}\neq \emptyset$ which will be used when we apply Lemma \ref{l-strict-transform-bnd-degree} below.

\emph{Step 3.}
By (3), there is a prime divisor $R$ over $M$ with centre $m$ mapping onto $T$. The image $S$ of $R$ is then a divisor over $Y$ with centre $y$, by Lemma \ref{l-desced-prime-div}. If $K_M+\Lambda_M$ and $K_M+L_M$ are the pullbacks of $K_X+\Lambda$ and $K_Y+L$ near $m$, respectively, then by (4), $\Lambda_M=L_M$ near $m$. Thus we get
$$
0=a(T,X,\Lambda)=a(R,M,\Lambda_M)=a(R,M,L_M)=a(S,{Y},{L})
$$
where the first equality is by assumption.
This implies that $a(S,P,G)=0$ by (6).

\emph{Step 4.}
Now applying Proposition \ref{p-disc-after-str-transform} to $(W,B_W), X,{Y}$ over $Z$, we find an $\epsilon$-lc 
pair $(V,C_V)$ and a projective birational morphism $V\to {Y}$ such that 
\begin{itemize}
\item if $B=\sum b_iB_i$ is the pushdown of $B_W$ to $X$ and $C_{{Y}}$ is the pushdown of $C_V$ 
to ${Y}$, then 
$$
C_Y=(\sum b_i (B_i)_{{Y}})^{\le 1-\epsilon},
$$

\item $K_{V}+C_{V}$ is nef$/{Y}$, and 

\item we have
$$
1\ge a(T,W,B_W)\ge a(S,V,C_{V}).
$$
\end{itemize}
Here $(B_i)_{{Y}}$ is the strict transform of $B_i$ as defined in \ref{ss-strict-transform} using $M$   
in the diagram of Step 2. 
By Lemma \ref{l-strict-transform-well-defined}, these strict transforms are the same if defined using $N$ in the same diagram.
We will also use $N$ to define strict transform of divisors from $X$ to $P$.
However, for a divisor $E$ on $V$ or on ${Y}$, we will use the usual notation $E^\sim$ to denote its birational transform on $P$. 

\emph{Step 5.}
Our next goal is to show that $\deg_{G/Z}(C_V)^\sim$ is bounded from above. 
By (5), $G$ contains all the exceptional divisors of $P\bir {Y}$ because $P\setminus G\bir Y$ is an open immersion, hence $(C_V)^\sim\le (C_{{Y}})^\sim+G$. 
Thus it is enough to show  $\deg_{G/Z} (C_{{Y}})^\sim$ is bounded. 

Since ${Y}\bir P$ is birational, $(B_i)_{{Y}}^\sim\le (B_i)_P$ where $(B_i)_{{Y}}$ and $(B_i)_P$ are the strict transforms of 
$B_i$ on ${Y}$ and $P$: this follows from the fact that $(B_i)_{{Y}}$ and $(B_i)_P$ are both defined using the image of a certain divisor on $N$.
Thus 
$$
(C_Y)^\sim \le ( \sum b_i(B_i)_{{Y}})^\sim=\sum b_i(B_i)_{{Y}}^\sim \le  \sum b_i(B_i)_P.
$$
Then recalling (7) and applying Lemma \ref{l-strict-transform-bnd-degree} to $X,A,B$ and $P,G$, we get  
$$
\deg_{G/Z} (C_Y)^\sim \le \deg_{G/Z} \sum b_i(B_i)_P\le \vol_{/Z}(A|_N+G|_N)  \le r.
$$
We then have shown that $\deg_{G/Z} (C_V)^\sim$ is bounded from above. Therefore, since $P=\PP^{d-1}_Z$, there is a 
bounded natural number $t$ such that $tG-(C_V)^\sim$ is pseudo-effective over $Z$.

\emph{Step 6.}
 Let $V'$ be a common resolution of $V,P$, and let $K_{V'}+C_{V'}$ be the pullback of $K_V+C_V$. 
Let $I_Y\sim_\Q 2dH$ be general, and let $I_{V'}$ be its pullback to $V'$.
Then by boundedness of extremal rays applied on $V$, $K_{V'}+C_{V'}+I_{V'}$ is nef over $Z$.  
Now let $\Theta_{V'}$ be the effective part of $C_{V'}$ and let 
$$
\Delta_{V'}=\Theta_{V'}+I_{V'}.
$$ 
Let $\Theta_P,\Delta_P$ be the pushdowns of $\Theta_{V'},\Delta_{V'}$ to $P$. Then 
$$
\Theta_P\le (C_V)^\sim+G
$$
as $G$ contains all the exceptional divisors of $P\bir {Y}$ (and so of $P\bir V$). Moreover,   
$$
\Delta_P\le (C_V)^\sim+G+I_Y^\sim.
$$ 
In particular, $\deg_{G/Z} \Delta_P$ is bounded from above, by Step 5 and (7) of Step 2 (using Lemma \ref{l-strict-transform-bnd-degree} to bound $\deg_{G/Z} I_Y^\sim$).  
Thus replacing $t$, we can assume that $tG-\Delta_P$ is pseudo-effective over $Z$.

Additionally, we can assume that $(V',\Delta_{V'})$ is $\epsilon$-lc.
Moreover, we can assume that $S$ is a divisor on $V'$, and since $a(S,V,C_V)\le 1$, it is a component of $C_{V'}$ with non-negative coefficient, so it is not a component of $\Theta_{V'}-C_{V'}$. Therefore, running an MMP$/P$ on $K_{V'}+\Delta_{V'}$ ends with a minimal model $V''$ which does not contract $S$ because $K_{V'}+C_{V'}+I_{V'}$ is nef over $Z$ and 
$$
0\le \Theta_{V'}-C_{V'}=K_{V'}+\Delta_{V'}-(K_{V'}+C_{V'}+I_{V'}).
$$

\emph{Step 7.}
We are now ready to finish the proof.
To summarise the constructions above, we have the following setting: 
\begin{itemize}
\item $(V'',\Delta_{V''})$ is an $\epsilon$-lc pair and $(P=\PP^{d-1}_Z,G)$ is an lc pair, 
where $G$ is the sum of the coordinate hyperplanes and the fibre over $z$,

\item $V''\to P$ is birational and $P\to Z$ is projection onto a smooth curve, 

\item $K_{V''}+\Delta_{V''}$ is nef$/P$, 

\item  there is a fixed $t>0$ such that $tG-\Delta_P$ is pseudo-effective over $Z$  
where $\Delta_P$ is the pushdown of $\Delta_{V''}$,

\item $S$ is a prime divisor on $V''$ mapping to the closed point $z\in Z$ with 
$$
a(S,V'',\Delta_{V''})\le 1 ~~~\mbox{and}~~~ a(S,P,G)=0.
$$
\end{itemize} 

Finally, let $F_X,F_Y, F_M,F_P$ be the fibres of $X\to Z$, $Y\to Z$, $M\to Z$, and $P\to Z$ over $z$, respectively.  Since $M\to X$ is \'etale at $m$ by (3), we have 
$\mu_TF_X=\mu_RF_M$ where $R$ is the prime divisor over $M$ which maps onto $T$ and onto the prime divisor $S$ over $Y$ (as in Step 3). Similarly, 
$\mu_RF_M=\mu_SF_Y$, hence $\mu_TF_X=\mu_SF_Y$.
On the other hand, since $Y\bir P$ is birational, $\mu_SF_Y=\mu_SF_P$. 

By assumption, Theorem \ref{t-bnd-mult-lc-places-fib-main} holds in the relative projective case $\PP^{d-1}_Z\to Z$, so applying the theorem to the above setting, $\mu_SF_P$ is bounded, hence $\mu_TF_X$ is also bounded.
\end{proof}


\section{\bf Multiplicities along lc places: reduction to a toric problem}

In this section we continue our study of multiplicities of fibres along lc places. We will reduce Theorem \ref{t-bnd-mult-lc-places-fib-main} to a purely toric problem and then solve this problem.

\subsection{Models of divisors}\label{ss-V-Y-models}
Consider $(X,\Lambda)\to (Z,E)$ where $Z$ is a smooth curve, $E=z$ is a closed point,  $X=\PP^{d-1}_Z\to Z$ is projection, and $\Lambda$ is the sum of the coordinate hyperplanes plus the fibre over $z$. We will construct certain varieties and maps associated to toroidal divisors over $X$. 

(1) 
Assume $T$ is a prime divisor over $X$, toroidal with respect to $(X,\Lambda)$, such that $T$ maps to $z\in Z$. If $T$ is not exceptional over $X$, let $X'\to X$ be the identity map and let $T'\subset X'$ be the image of $T$. Otherwise, let $X'\to X$ be the extremal contraction that extracts (the birational transform of) $T$ and again let $T'\subset X'$ be the image of $T$. Then $X'$ is of Fano type over $Z$  (by the relative version of [\ref{B-Fano}, 2.13(7)]). So we can run a $-T'$-MMP over $Z$ ending with a model $V/Z$. By a slight abuse of notation, we denote the pushdown of $T'$ to $V$ by $T$.  We call $V\to Z$ \emph{the model associated to $T$}.

We will argue that 
\begin{itemize}
\item $V$ is a minimal model for $-T'$ over $Z$, 
\item $X\bir V$ is an isomorphism over $Z\setminus E$, 
\item $T$ is the only component of the fibre of $V\to Z$ over $z$, and
\item $V\to Z$ is a $K_V$-negative extremal contraction. 
\end{itemize}

Note that $T'$ is not contracted by the MMP as we are running a $-T'$-MMP.
Also since $T'$ is vertical over $Z$, $-T'$ is pseudo-effective over $Z$, so $V$ is a minimal model of $-T'$ over $Z$. Moreover, by construction, both $X'\to X$ and $X'\bir V$ are isomorphisms over $Z\setminus E$. On the other hand, since the fibre of $V\to Z$ over $z$ is connected and since $-T$ is nef over $Z$, $T$ cannot intersect any other component of this fibre, so $T$ is the only component of the fibre. 

If $T$ is not exceptional over $X$, then $V=X$ hence $V\to Z$ is a $K_V$-negative extremal contraction. Assume then that $T$ is exceptional over $X$ in which case $X'$ has Picard number two over $Z$. Then since $T$ is the only irreducible component of the fibre of $V\to Z$ over $z$, we see that the fibre of $X\to Z$ over $z$ is contracted over $V$. Therefore, since $X'$ has Picard number two over $Z$, $V$ has Picard number one over $Z$, and since $-K_V$ is ample over the generic point of $Z$, it is ample over the whole $Z$. Thus indeed $V\to Z$ is a $K_V$-negative extremal contraction.

(2) 
Suppose now that $T,D$ are distinct prime divisors over $X$ mapping to $z$ and toroidal with respect to $(X,\Lambda)$. Let $V\to Z$ be the model associated to $T$ as in (1). Since $T$ is the only component of the fibre of $V\to Z$ over $z$, $D$ is exceptional over $V$.
Also if $K_V+\Lambda_V$ is the pullback of $K_X+\Lambda$, then since $K_X+\Lambda\sim 0/Z$,  $(V,\Lambda_V)$ is lc and $a(D,V,\Lambda_V)=0$. Then we can extract $D$ via an extremal contraction $\psi\colon Y\to V$.
We call $Y\to V\to Z$ \emph{the model associated to $T,D$}.

Also given $r\in \N$ and $\epsilon\in (0,1]$, we define 
$$
\Theta_Y=(1-\epsilon)D+\psi^*r\Lambda_V.
$$

(3)
Under the assumptions of (1), assume $Z=\A^1$ and that $z$ is the origin. Then we are in the toric setting, that is, $(X,\Lambda)\to (Z,E)$ is a toric morphism and $T$ is a toric divisor over $X$. 
Let 
$$
\mathfrak{v}_1=(1,0,\dots,0),\dots,\mathfrak{v}_d=(0,\dots,0,1)
$$ 
be the standard basis of $\Z^d$ and let 
$$
\mathfrak{c}=-\sum_2^d \mathfrak{v}_i=(0,-1,\dots,-1).
$$ 
Then $X$ is the toric variety associated to the fan whose maximal cones are generated by $\mathfrak{v}_1$ and subsets of $\{\mathfrak{v}_2,\dots,\mathfrak{v}_d,\mathfrak{c}\}$ of size $d-1$, and $X\to Z$ is induced by the projection $\Z^d\to \Z$ onto the first factor. The support of the fan of $X$ is $\R^{\ge 0}\times \R^{d-1}$.

Now $T$ corresponds to a primitive vector $\mathfrak{n}=(n_1,\dots,n_d)\in \Z^d$ with $n_1>0$. Then the model $V\to Z$ 
associated to $T$ has a toric interpretation.  
Indeed, $V$ is the toric variety associated to the fan whose maximal cones are generated by $\mathfrak{n}$ and subsets of $\{\mathfrak{v}_2,\dots,\mathfrak{v}_d,\mathfrak{c}\}$ of size $d-1$, and $V\to Z$ is induced by the same projection $\Z^d\to \Z$. Moreover, the multiplicity of $T$ in the fibre of $V\to Z$ over $z$ is just $n_1$: to see this, let $\rho\colon \R^{\ge 0}\to \R$ be the support function of the divisor $E$ on $Z$, that is, $\rho$ is the linear map sending $1$ to $-1$. Now composing $\rho$ with the projection $\R^{\ge 0}\times \R^{d-1}\to \R$ onto the first factor we get the support function of $f^*E$ on $X$ where $f$ denotes $X\to Z$.  The value of this function at $\mathfrak{n}$ is $-n_1$ which implies that the multiplicity of $T$ in the pullback of $f^*E$ to any resolution of $X$ where $T$ appears, is $n_1$ (or we can use $V\to Z$ to calculate the multiplicity). 

If $D$ is another toric divisor over $X$, then similar reasoning shows that the varieties and morphisms in the model $Y\to V\to Z$ are toric. 

\subsection{Reduction of Theorem \ref{t-bnd-mult-lc-places-fib-main} to a toric problem}
We have already reduced Theorem \ref{t-bnd-mult-lc-places-fib-main} to the relative projective case $X=\PP^{d-1}_Z\to Z$. But even if $(X,\Lambda)\to (Z,E)$ is toric, we still have $(W,B_W)$ which may be far from toric. We will use the constructions of the previous subsection to reduce Theorem \ref{t-bnd-mult-lc-places-fib-main} to the following problem which only has toric ingredients.

\begin{prop}\label{p-Nsigma-Y-model}
Let $d,r\in \N$ and $\epsilon\in (0,1]$. Then there exists $\epsilon'\in \R^{>0}$ depending only on $d,r,\epsilon$ satisfying the following.
Assume $(X,\Lambda)\to (Z,E)$ and $T,D$ are as in the toric setting \ref{ss-V-Y-models}(3), and  
that $Y\to V\to Z$ and $\Theta_Y$ are as in \ref{ss-V-Y-models}(1),(2) constructed for $T,D,r,\epsilon$. If 
$
a(D,V,0)<\epsilon',
$ 
 then 
$$
T^\sim\subseteq N_\sigma(K_Y+\Theta_Y/Z)
$$
where $T^\sim\subset Y$ is the birational transform of $T\subset V$.
\end{prop}

We will prove the proposition later. The proof shows that we can take $\epsilon'=\frac{\epsilon}{3dr}$. Also note that $N_\sigma$ is well-defined here as $Y$ is of Fano type over $Z$.

\begin{lem}\label{l-reduction-to-toric-case}
Assume that Theorem \ref{t-bnd-mult-lc-places-fib-main} holds in dimension $\le d-1$ and that Proposition \ref{p-Nsigma-Y-model} holds in dimension $d$. Then Theorem \ref{t-bnd-mult-lc-places-fib-main} holds in dimension $d$. 
\end{lem}
\begin{proof}
By Proposition \ref{p-reduction-to-proj-space}, it is enough to treat Theorem \ref{t-bnd-mult-lc-places-fib-main} in the relative projective case, that is, when $X=\PP^{d-1}_Z$ and $\Lambda$ is the sum of the coordinate hyperplanes plus the fibre over $z$. Adding the pullback of a general divisor $0\le J\sim_\Q 2\Lambda/Z$ to $B_W$, we can assume that $K_W+B_W$ is nef over $Z$.
Since $a(T,X,\Lambda)=0$, $T$ is toroidal with respect to $(X,\Lambda)$. We can assume that $T$ is exceptional over $X$ otherwise the multiplicity of $T$ in the fibre of $X\to Z$ over $z$ is 1.
Let $V\to Z$ be the model associated to $T$ as in \ref{ss-V-Y-models}(1). We need to show that the multiplicity of the fibre of $V\to Z$ along $T$ is bounded from above.

Consider $(X'',\Lambda'')\to (Z'',E'')$ where $Z''=\A^1$, $E''=z''$ is the origin, and $X''=\PP_{Z''}^{d-1}$ with $\Lambda''$ its toric boundary. Perhaps after shrinking $Z$ around $z$, there is an \'etale morphism $Z\to Z''$ mapping only $z$ to $z''$. Then $(X,\Lambda)\to (Z,E)$ is just the pullback of $(X'',\Lambda'')\to (Z'',E'')$ by base change to $(Z,E)$ (as in \ref{ss-tower-couples}).  
Moreover, corresponding to $T$, there is a prime divisor $T''$ over $X''$, toroidal with respect to $(X'',\Lambda'')$. If $V''\to Z''$ is the model associated to $T''$, as in \ref{ss-V-Y-models}(1), then 
$V\to Z$ is the pullback of $V''\to Z''$ by base change to $Z$ because taking base change to $Z$ of the steps of the construction of $V''$ gives the steps of the construction of $V$. 

Let $\epsilon'<0$ be as in Proposition \ref{p-Nsigma-Y-model} for the given data $d,r,\epsilon$.
Assume that $V''$ is $\epsilon'$-lc. Then the multiplicity of $T''$ in the fibre of $V''\to Z''$ over $z''$ is bounded, by [\ref{BC-toric}, Corollary 1.5], because $V''\to Z''$ is a toric Fano fibration. This in turn implies that the multiplicity of the fibre of $V\to Z$ along $T$ is bounded. 

From now on assume that $V''$ is not $\epsilon'$-lc which is equivalent to saying that $V$ is not $\epsilon'$-lc. We will derive a contradiction. Pick a prime divisor $D''$ over $V''$ with $a(D'',V'',0)<\epsilon'$. Let $D$ be the corresponding divisor over $V$. Then $a(D,V,0)<\epsilon'$, in particular, $D$ is toroidal with respect to $(X,\Lambda)$.     
Let ${\psi}\colon {Y}\to V$ and $\Theta_{{Y}}$ be as in \ref{ss-V-Y-models}(2) constructed for $D$.
 Let $K_{{Y}}+B_{{Y}}$ be the pullback of $K_W+B_W$ under ${Y}\bir W$. Then since $K_W+B_W$ is nef over $Z$, $K_{{Y}}+B_{{Y}}$ is a limit of movable$/Z$ $\R$-divisors. Thus  
$$
T^\sim\not\subset N_\sigma(K_{{Y}}+B_{{Y}}/Z),
$$
where $T^\sim\subset {Y}$ is the birational transform of $T\subset V$. Note that $N_\sigma$ is well-defined here as $Y$ is of Fano type over $Z$.

Write $B_Y=B_Y'+B_Y''$ where $B_Y'$ is the exceptional$/V$ part of $B_Y$ (so $B_Y',B_Y''$ have no common components). Let $B_V$ be the pushdown of $B_Y$, which coincides with the pushdown of $B_Y''$. By construction, the only possible exceptional divisor of $V\bir W$ is $T$ because  the only exceptional divisor of $V\bir X$ is $T$. So $B_V\ge 0$ as 
$$
\mu_TB_V=1-a(T,W,B_W)\ge 0.
$$ 
Thus $B_Y''\ge 0$ as $B_Y''$ is the birational transform of $B_V$. 
Since $B_Y''$ does not contain the exceptional divisor of $\psi$, we deduce that $\psi^*B_V-B_Y''\ge 0$. Thus 
$$
B_Y'+\psi^*B_V=B_{Y}'+B_Y''+\psi^*B_V-B_Y''=B_Y+\psi^*B_V-B_Y''
$$
where $\psi^*B_V-B_Y''$ is effective and exceptional/$V$ which in particular means that its support does not contain $T^\sim$. Therefore,
$$
T^\sim\not\subset N_\sigma(K_Y+B_Y+\psi^*B_V-B_Y''/Z)=N_\sigma(K_Y+B_Y'+\psi^*B_V/Z). 
$$

On the other hand, $B_V\equiv \alpha r\Lambda_V/Z$ for some number $\alpha$. Then $\alpha\le 1$: indeed, since $V\to Z$ is a Mori fibre space, $\alpha$ is determined over the generic point $\eta_Z$; but $V\bir X$ is an isomorphism over $\eta_Z$, so $\alpha$ satisfies $B\equiv \alpha r\Lambda/Z$ (recall that $B$ is the pushdown of $B_W$); then 
$$
\alpha r\le \alpha r\deg_{A/Z}\Lambda=\deg_{A/Z}B\le \deg_{A/Z}A\le r
$$  
where we use the assumption that $A-B$ is pseudo-effective over $Z$; 
thus $\alpha\le 1$ as claimed. Therefore,  
$$
T^\sim\not\subset N_\sigma(K_Y+B_Y'+\psi^*r\Lambda_V/Z).
$$

Since $(W,B_W)$ is $\epsilon$-lc and $D$ is the only exceptional divisor of $Y\to V$ and $B_Y'$ is exceptional over $Y$, we deduce that $B_Y'\le (1-\epsilon)D$. 
Thus 
$$
T^\sim\not\subset N_\sigma(K_Y+(1-\epsilon)D+\psi^*r\Lambda_V/Z)=N_\sigma(K_Y+\Theta_Y/Z) 
$$
because $(1-\epsilon)D-B_Y'$ is effective and exceptional$/V$ not containing $T^\sim$. 

Now if $Y''\to V''$ and $\Theta_{Y''}$ are constructed for $D''$ as in \ref{ss-V-Y-models}(2), then $Y\to V$ is the pullback of $Y''\to V''$ by base change to $Z$, and $K_Y+\Theta_Y$ is the pullback of $K_{Y''}+\Theta_{Y''}$. Thus 
$$
{T''^\sim}\not\subset N_\sigma(K_{Y''}+\Theta_{Y''}/Z) 
$$ 
where ${T''^\sim}$ is the birational transform of $T''\subset V''$ because running an MMP$/Z''$ on $K_{Y''}+\Theta_{Y''}$ corresponds to running an MMP$/Z$ on $K_{Y}+\Theta_{Y}$, and the latter does not contract $T^\sim$.
Therefore, applying Proposition \ref{p-Nsigma-Y-model}, we get a contradiction.
\end{proof}

We should remark that although we reduced Theorem \ref{t-bnd-mult-lc-places-fib-main} to the toric problem \ref{p-Nsigma-Y-model} by a relatively easy argument but in practice realising that \ref{p-Nsigma-Y-model} is the right setting to consider was perhaps the hardest part of completing the proof of \ref{t-bnd-mult-lc-places-fib-main}. Intuitively, there really is no reason why one should have considered that setting. A much more natural thing to do is to extract all the prime divisors $E_i$ over $V$ with $a(E_i,V,0)<\epsilon$, say via $\rho\colon U\to V$, and then consider 
$$
\Theta_U=\sum (1-\epsilon)E_i+\rho^*r\Lambda_V
$$ 
to proceed. But then it would be extremely hard to keep track of all the possibly numerous divisors $E_i$. Even working out examples in this setting using computers is quite time consuming. Hopefully we will write more about this story elsewhere.

\subsection{Proof of the toric problem}

\begin{proof}(Proof of Proposition \ref{p-Nsigma-Y-model})
\emph{Step 1.}
Assume that $T$ corresponds to the primitive vector 
$$
\mathfrak{n}=(n_1,\dots,n_d)\in \Z^d
$$ 
and that $D$ corresponds to the primitive vector 
$$
\mathfrak{l}=(l_1,\dots,l_d)\in \Z^d.
$$
 Note that $n_1,l_1\ge 0$ because the support of the fan of $X$ (hence of $V,Y$) is equal to $\R^{\ge 0}\times \R^{d-1}$ (see \ref{ss-V-Y-models}(3)). Moreover, since both $T,D$ map to $z$, we see that $n_1,l_1>0$. 
 
 Consider a maximal cone $\tau$ in the fan of $V$ which contains $\mathfrak{l}$. Then $\mathfrak{n}$ necessarily generates one of the edges of $\tau$, by the construction of $V$ or by noting that $l_1>0$.
In other words, $\tau$ is generated by $\mathfrak{n}$ and a subset 
$$
\{\mathfrak{e}_1,\dots, \mathfrak{e}_{d-1}\}\subset \{\mathfrak{v}_2,\dots,\mathfrak{v}_d,\mathfrak{c}\}.
$$
 We can write 
 \begin{equation}\label{eq-proof-toric-formula}
    \mathfrak{l}=\gamma \mathfrak{n}+\sum_1^{d-1} \alpha_j\mathfrak{e}_j
\end{equation}
where $\gamma, \alpha_j\ge 0$.
Since the first coordinate of each $\mathfrak{e}_j$ is $0$, we see that  $l_1=\gamma n_1$, so $\gamma=\frac{l_1}{n_1}$.

\emph{Step 2.}
By \ref{ss-toric-geometry},   
$$
a(D,V,0)=\gamma +\sum_1^{d-1} \alpha_j.
$$
We then have 
$$
K_Y+(1-a(D,V,0))D=\psi^*K_V.
$$
Let $S_V$ be the horizontal$/Z$ part of $\Lambda_V$, where as before $K_V+\Lambda_V$ is the pullback of $K_X+\Lambda$. Then 
$$
a(D,V,S_V)=\gamma
$$ 
again by \ref{ss-toric-geometry}. So 
$$
K_Y+S_V^\sim+(1-a(D,V,S_V))D=\psi^*(K_V+S_V)
$$
where $S_V^\sim$ denotes the birational transform of $S_V$.
Thus 
$$
\psi^*S_V=(a(D,V,0)-a(D,V,S_V))D+S_V^\sim=(\sum_1^{d-1} \alpha_j)D+S_V^\sim.
$$
This can also be seen by considering the support function of $S_V$, on the open subset $V_\tau\subset V$ defined by $\tau$, which assigns $-1$ to each  $\mathfrak{e}_j$ and assigns $0$ to $\mathfrak{n}$.

To ease notation, we denote $a:=a(D,V,0)$.
Since $K_V+S_V+T\sim 0/Z$ and $T\equiv 0/Z$, we have $K_V+S_V\equiv 0/Z$. Thus 
$$
K_Y+(1-a)D+\psi^*S_V=\psi^*K_V+\psi^*S_V\equiv 0/Z. 
$$
On the other hand,  
$$
n_1T^\sim+l_1D\sim 0/Z
$$ 
because $n_1T^\sim+l_1D$ is the fibre of $Y\to Z$ over $E=z$ (see \ref{ss-V-Y-models}(3) above for an explanation in terms of the pullback of the support function of $E$ on $Z$).

\emph{Step 3.}
Let 
$$
u:=(r-1)(\sum_1^{d-1} \alpha_j).
$$ 
Then making use of $\Lambda_V\equiv S_V/Z$ and the steps above, over $Z$ we have 
$$
K_Y+\Theta_Y\equiv K_Y+(1-\epsilon)D+\psi^*rS_V
$$
$$
=K_Y+(1-a)D+\psi^*S_V+(a-\epsilon)D+\psi^*(r-1)S_V
$$
$$
\equiv (a-\epsilon)D+\psi^*(r-1)S_V
$$
$$
=(a-\epsilon)D+(r-1)((\sum_1^{d-1} \alpha_j)D+S_V^\sim)
$$
$$
=(a-\epsilon+u)D+(r-1)S_V^\sim
$$
$$
\equiv (a-\epsilon+u)\frac{-n_1}{l_1}T^\sim+(r-1)S_V^\sim.
$$
$$
=(-a+\epsilon-u)\frac{n_1}{l_1}T^\sim+(r-1)S_V^\sim.
$$

\emph{Step 4.}
Let $W\to Z$ be the model associated to $D$ constructed as in \ref{ss-V-Y-models}(1). The support of the fibre of $W\to Z$ over $z$ is the birational transform of $D\subset Y$. The maximal cones of the fan of $W$ are generated by $\mathfrak{l}$ and subsets of $\{\mathfrak{v}_2,\dots,\mathfrak{v}_d,\mathfrak{c}\}$ of size $d-1$.
 Let $\pi\colon U\to W$ be the extremal toric birational contraction which extracts the birational transform of $T$. So $U\to W\to Z$ is the model associated to $D,T$ (first $D$, then $T$) as in \ref{ss-V-Y-models}(2) but with different notation. 
We then have a diagram 
$$
\xymatrix{
Y\ar[d] \ar@{-->}[rr]&  & U\ar[d]\\
V\ar[rd]& & W\ar[ld]\\
&Z.&
}
$$ 
 
 From the formula (\ref{eq-proof-toric-formula}) of Step 1, we get 
$$ 
\mathfrak{n}=\frac{1}{\gamma}\mathfrak{l}-\frac{1}{\gamma}\sum_1^{d-1} \alpha_j\mathfrak{e}_j.  
$$
Consider a maximal cone in the fan of $W$ containing $\mathfrak{n}$. Similar to Step 1, this cone is  generated by 
$\mathfrak{l}$ and a  subset 
$$
\{\mathfrak{f}_1,\dots, \mathfrak{f}_{d-1}\}\subset \{\mathfrak{v}_2,\dots,\mathfrak{v}_d,\mathfrak{c}\}.
$$
Then we can write 
$$
\mathfrak{n}=\lambda\mathfrak{l}+\sum_1^{d-1} \beta_k\mathfrak{f}_k
$$
where $\lambda, \beta_k\ge 0$.
Since the first coordinate of each $\mathfrak{f}_k$ is $0$, we have $n_1=\lambda l_1$, hence $\lambda=\frac{n_1}{l_1}=\frac{1}{\gamma}$. Thus multiplying by $\gamma$, we get 
$$
\gamma\mathfrak{n}=\mathfrak{l}+\sum_1^{d-1} \gamma\beta_k\mathfrak{f}_k,
$$
so
$$
\sum_1^{d-1} \gamma\beta_k\mathfrak{f}_k=\gamma\mathfrak{n}-\mathfrak{l}=
-\sum_1^{d-1} \alpha_j\mathfrak{e}_j.
$$ 
 
\emph{Step 5.} 
In this step assume that   
$$
-a+\epsilon-u>(r-1)(\sum_1^{d-1} \gamma\beta_k).
$$
Then  
$$
(-a+\epsilon-u)\frac{1}{\gamma}>(r-1)(\sum_1^{d-1} \beta_k)
$$
meaning 
$$
(-a+\epsilon-u)\frac{n_1}{l_1}>(r-1)(\sum_1^{d-1} \beta_k).
$$

On the other hand, similar to the formula for $\psi^*S_V$ in Step 2, we have 
$$
(r-1)\pi^*S_W=(r-1)(\sum_1^{d-1} \beta_k)T^\sim+(r-1)S_W^\sim
$$
where $S_W$ is the sum of the horizontal$/Z$ toric prime divisors on $W$ and $S_W^\sim$ is its birational transform.
So on $U$ we have 
$$
T^\sim\subseteq N_\sigma((-a+\epsilon-u)\frac{n_1}{l_1}T^\sim+(r-1)S_W^\sim/Z).
$$

Now by construction, $Y\bir U$ is an isomorphism in codimension one and $Y\to V$ and $V\bir W$ are isomorphisms over the generic point of $Z$.  By Step 3,
$$
K_Y+\Theta_Y \equiv (-a+\epsilon-u)\frac{n_1}{l_1}T^\sim+(r-1)S_V^\sim/Z,
$$
so considering birational transform of both sides on $U$, we have  
$$
K_U+\Theta_Y^\sim\equiv (-a+\epsilon-u)\frac{n_1}{l_1}T^\sim+(r-1)S_W^\sim/Z
$$
where we use the fact that the birational transforms of $S_V,S_W$ on $U$ are the same. Therefore, 
$$
T^\sim\subseteq N_\sigma(K_U+\Theta_Y^\sim/Z)
$$
which in turn implies 
$$
T^\sim\subseteq N_\sigma(K_Y+\Theta_Y/Z).
$$

\emph{Step 6.}
In this step we show that we can take $\epsilon'=\frac{\epsilon}{3dr}$. 
In view of the previous step, it is enough to ensure that 
$$
-a+\epsilon-u>(r-1)(\sum_1^{d-1} \gamma\beta_k).
$$
Recall the following equality from Step 4,
$$
\sum_1^{d-1} \gamma\beta_k\mathfrak{f}_k=-\sum_1^{d-1} \alpha_j\mathfrak{e}_j.
$$ 

We claim that $\gamma\beta_k<2a$ for each $k$. 
First, we examine the coordinates of $-\sum_1^{d-1} \alpha_j\mathfrak{e}_j$ (here we mean coordinates in terms of the $\mathfrak{v}_1,\dots, \mathfrak{v}_d$). The first coordinate is clearly zero, so we ignore it  and discuss the other coordinates. If no $\mathfrak{e}_j=\mathfrak{c}$, then each coordinate is as $-\alpha_j$ for some $j$; but if say $\mathfrak{e}_1=\mathfrak{c}$, then there is $i\ge 2$ such that $\mathfrak{e}_j\neq \mathfrak{v}_i$ for every $j$, so the $i$-th coordinate is $\alpha_1$ and the other coordinates are as $\alpha_1-\alpha_j$ for $j\ge 2$. So in any case, the absolute value of each coordinate of $-\sum_1^{d-1} \alpha_j\mathfrak{e}_j$ is $<a$ because $a=\gamma +\sum_1^{d-1} \alpha_j$.

Next, we examine the coordinates of $\sum_1^{d-1} \gamma\beta_k\mathfrak{f}_k$ (again in terms of the $\mathfrak{v}_i$, and again the first coordinate is zero and we ignore it). If no $\mathfrak{f}_k=\mathfrak{c}$, then the coordinates of $\sum_1^{d-1} \gamma\beta_k\mathfrak{f}_k$ are the $\gamma\beta_k$; but if say $\mathfrak{f}_1=\mathfrak{c}$, then there is $i\ge 2$ such that $\mathfrak{f}_k\neq \mathfrak{v}_i$ for every $k$, so the $i$-th coordinate is
 $-\gamma\beta_1$ and the others are as $-\gamma\beta_1+\gamma\beta_k$ for $k\ge 2$. Thus in any case, $\gamma\beta_k<2a$ for each $k$, by the previous paragraph, proving the claim. 
 
Now
$$
(r-1)(d-1)2a>(r-1)\sum_1^{d-1} \gamma\beta_k.
$$
Moreover, taking $\epsilon'=\frac{\epsilon}{3dr}$, the assumption $a<\epsilon'$ of the proposition ensures that 
$$
\epsilon-ra  > (r-1)(d-1)2a.
$$
Also
$$
-a+\epsilon-u=-a+\epsilon-(r-1)(\sum_1^{d-1} \alpha_j)\ge -a+\epsilon-(r-1)a=\epsilon-ra.
$$
Therefore, 
$$
-a+\epsilon-u>(r-1)(d-1)2a> (r-1)\sum_1^{d-1} \gamma\beta_k.
$$
\end{proof}

It is worth pointing out an alternative way in Step 6 to ensure that $\gamma\beta_k$ is sufficiently small (which would be enough to prove the proposition but perhaps with a smaller $\epsilon'$) is to observe that both the $\mathfrak{e}_j$ and the $\mathfrak{f}_k$ form a basis of $\Z^{d-1}$, so one can convert coordinates in terms of the $\mathfrak{e}_j$ into coordinates in terms of the $\mathfrak{f}_k$ using finitely many possible matrices.

\begin{exa}
We illustrate the proof of the proposition with this two-dimensional example.
Let $Z=\A^1$ and consider $(X=\PP^1_{Z},\Lambda)$ with its standard toric structure given by the fan generated by the vectors $(1,0), (0,1), (0,-1)$ in $\Z^2$. Let $T,D$ be the toric divisors over $X$ corresponding to $(n,1)$ and $(1,0)$, respectively, where $2\le n\in \N$. Then using notation in the proof above, $W=X$. Let $S_1,S_2$ be the horizontal components of $\Lambda$ corresponding to the vectors $(0,1),(0,-1)$. We can obtain $T$ by blowing up $S_1\cap D$ first, then blowing up the intersection of $D^\sim$ and the exceptional divisor and so on where in each step we blow up the intersection of $D^\sim$ and the newest exceptional divisor. We do $n$ blowups in total. The last exceptional divisor is $T$. At this point, $T$ is the only exceptional$/W$ divisor intersecting $D^\sim$. Next, we contract all the exceptional divisors except $T$. This gives $Y=U$. By the previous sentence, $D\cdot T^\sim=1$
(here we are abusing notation as we denote the birational transform of $D$ on $Y$ again by $D$). Contracting $D$ gives $Y\to V$. 

We can calculate  
$$
a(T,W,0)=n+1
$$   
by observing $(n,1)=n(1,0)+(0,1)$, so 
$$
K_Y-nT^\sim=\pi^*K_W. 
$$
Moreover, using this and $nT^\sim+D\sim 0/Z$, we can calculate 
$$
(K_Y+D)\cdot T^\sim=(K_Y-nT^\sim)\cdot T^\sim=0
$$
and so from $\Lambda_V\cdot T=\frac{2}{n}$ we get
$$
(K_Y+\Theta_Y)\cdot T^\sim=(-\epsilon D+\psi^*r\Lambda_V) \cdot T^\sim=
-\epsilon+\frac{2r}{n}.
$$
Therefore, if $n$ is large enough, then this intersection number is negative, and so 
$$
T^\sim\subseteq N_\sigma(K_Y+\Theta_Y/Z).
$$
Note that $n$ being large also means that 
$$
a(D,V,0)=\frac{2}{n}
$$   
is small, where we derive the equality from $(1,0)=\frac{1}{n}(n,1)+\frac{1}{n}(0,-1)$. 
\end{exa}

\subsection{Proof of Theorem \ref{t-bnd-mult-lc-places-fib-main}}

\begin{proof}[Proof of Theorem \ref{t-bnd-mult-lc-places-fib-main}]
By induction on dimension, we can assume the theorem holds in lower dimension. Now apply Lemma \ref{l-reduction-to-toric-case} and Proposition \ref{p-Nsigma-Y-model}.
\end{proof}


\section{\bf Singularities and complements on Fano fibrations}

In this section we prove Theorem \ref{t-mc-sh-conj} and its variant in the context of generalised pairs, and we prove  Corollaries \ref{cor-mckernan-conj}, \ref{cor-mult-fib} and \ref{cor-klt-comp}, and Theorem \ref{t-gonality}. The strategy of the proof of \ref{t-mc-sh-conj} is similar to the one employed in the proof of the main result of [\ref{B-sing-fano-fib}] which goes via reducing to the case $\dim Z=1$ and to boundedness of multiplicities of certain fibrations where we can apply Theorem \ref{t-bnd-mult-lc-places-fib-main}.

\subsection{Proof of Theorem \ref{t-mc-sh-conj} and Corollaries \ref{cor-mckernan-conj}, \ref{cor-mult-fib}}

\begin{proof}[Proof of Theorem \ref{t-mc-sh-conj}]
We can assume $\epsilon$ is sufficiently small and that it is of the form $\frac{1}{p}$ for some $p\in\N$. 
By taking hyperplane sections of birational models of $Z$, the theorem can be reduced to the case $\dim Z=1$, for example, see [\ref{B-sing-fano-fib}, Lemma 3.2] and its proof. We will then assume that $Z$ is a curve.
It is enough to show that the $\frac{\epsilon}{2}$-lc threshold $t$ of $f^*z$ with respect to $(X,B)$ is bounded from below, for each closed point $z\in Z$. Assuming $t$ is arbitrarily small, we will derive a contradiction. 

We can find a prime divisor $T$ over $X$ so that 
$$
a(T,X,B+tf^*z)=\frac{\epsilon}{2}.
$$  
Then since $(X,B)$ is $\epsilon$-lc, 
$$
\mu_Ttf^*z\ge \frac{\epsilon}{2}
$$ 
which means that $\mu_Tf^*z$ is arbitrarily large. It is enough to show that $\mu_Tf^*z$ is bounded from above.
For the rest of the proof we focus on showing this boundedness.
Replacing $B$ with $B+tf^*z$, replacing $\epsilon$ with $\frac{\epsilon}{2}$, and extracting $T$, we can assume that $T$ is a component of $B$ with coefficient $1-\epsilon$. Also we can assume $X$ is $\Q$-factorial. 

Run an MMP on $-(K_X+(1-\epsilon)T)$ over $Z$ which ends with a minimal model $X'$, that is, $-(K_{X'}+(1-\epsilon)T')$ is nef and big over $Z$. The bigness follows from the fact that $-K_X$ is big over $Z$ and $T$ is vertical over $Z$. Note however that the MMP might contract $T$ but this does not cause any problem.
Since $K_X+B\equiv 0/Z$ and $B\ge (1-\epsilon)T$, $(X',(1-\epsilon)T')$ is $\epsilon$-lc.
By Theorem \ref{t-bnd-compl-usual-local}, there is an $n$-complement $K_{X'}+\bar{\Lambda}'$ of $K_{X'}+(1-\epsilon)T'$ over $z$ with $\bar{\Lambda}'\ge (1-\epsilon)T'$, for some fixed $n\in \N$ depeding only on $d$.
Note that here we are making use of the assumption that $\epsilon$ is of the form $\frac{1}{p}$ so that $1-\epsilon=1-\frac{1}{p}$ is a standard number; the number $n$ does not depend on the choice of $p$ as it works for all standard numbers. Then we get an induced $n$-complement $K_X+\Lambda$ over $z$ with $\Lambda\ge (1-\epsilon)T$. Since $\epsilon$ is sufficiently small and $n$ does not depend on $\epsilon$, 
we deduce $\mu_T\Lambda=1$.

Run an MMP on $-K_X$ over $Z$ and replace $X$ with the resulting model so that $-K_X$ is nef and big over $Z$. 
This may contract $T$ but we will not need $T$ to be a divisor on $X$ any more: we only need the properties $a(T,X,B)\le 1$ and $a(T,X,\Lambda)=0$. Moreover, replacing $X$ with the ample model of $-K_X$, we can further assume that $-K_X$ is ample over $Z$. We may lose the $\Q$-factorial property of $X$ but we will not need it any more.

Now the general fibres $F$ of $X\to Z$ are $\epsilon$-lc Fano varieties. Then by [\ref{B-BAB}], $F$ belongs to a bounded family, and perhaps after replacing $n$ with a bounded multiple, we can assume that $-nK_F$ is very ample. Thus shrinking $Z$ around $z$ we can assume that $-nK_X$ is very ample over $Z\setminus\{z\}$. After shrinking $Z$, this defines an embedding $U\to \PP^q_{Z\setminus\{z\}}$ for some fixed $q$  where $U$ is the inverse image of $Z\setminus\{z\}$. Let $Y$ be the closure of the image of $U$ in $\PP^q_Z$. Then we get a birational map $X\bir Y/Z$ which is an isomorphism over $Z\setminus\{z\}$. Denote $Y\to Z$ by $g$. 

Let $D\subset Y$ be the support of the birational transform of the horizontal$/Z$ part of $\Lambda$. Then $(Y,D)\to Z$ is relatively bounded: indeed, $(Y,D)\to Z$ is flat as $Z$ is a smooth curve and $Y$ and each component of $D$ maps onto $Z$; moreover, if $A$ is the divisor on $\PP^q_Z$ which is pullback of a hyperplane on $\PP^q_k$, then $A$ is very ample$/Z$ and 
$$
\deg_{A/Z}A=(-nK_F)^{\dim F} \ \ \mbox{and} \ \ \deg_{A/Z}D=(-nK_F)^{\dim F-1}\cdot D_F\le (-nK_F)^{\dim F}
$$ 
are bounded from above. For the inequality we use the fact that 
$$
-nK_F\sim n\Lambda_F\ge D_F=D|_F.
$$ 

The rest of the proof is somewhat similar to the proof of Proposition \ref{p-reduction-to-toroidal}.
Now applying Theorem \ref{t-bnd-torification} to $(Y,D)\to Z$, $A$, $z$, perhaps after shrinking $Z$ around $z$, there exists a diagram 
$$
 \xymatrix{
 (Y',D') \ar[d]^{g'} \ar[r]^{\pi} &    (Y,D) \ar[d]^g \\
  (Z',E') \ar[r]^{\mu} & Z
  } 
$$  
satisfying the properties listed in the theorem. Pick $z'\in Z'$ mapping to $z$. 

Let $W'\to Y'$ be a log resolution so that the induced map $W'\bir X$ is a morphism. Let $K_{W'}+B_{W'}$ and $K_{W'}+\Lambda_{W'}$ be the pullbacks of $K_X+B$ and $K_X+\Lambda$, respectively. Then $(W',B_{W'})$ is sub-$\epsilon$-lc and $(W',\Lambda_{W'})$ is sub-lc. We can assume that 
$$
(W',\Supp B_{W'}\cup \Supp \Lambda_{W'})
$$ 
is log smooth. We then have the diagram 
$$
\xymatrix{
&& W'\ar[d]\ar[drr]& \\
& & {Y'} \ar[r] \ar[d] & Y \ar[d]\ar@{-->}[r]& X\ar[ld]\\
&& Z'\ar[r] &  Z.
}
$$

Let $\Lambda'$ be the pushdown of $\Lambda_{W'}$ to $Y'$. We claim that $\Lambda'\le D'$ holds over a neighbourhood of $z'$. Indeed, assume $R$ is a component of $\Lambda'$ with positive coefficient which is not a component of $D'$. Then $R$ is not mapped to $z'$ because by Theorem \ref{t-bnd-torification}, $D'$ contains the support of the fibre of $X'\to Z$ over $z$. So we can assume $R$ is horizontal over $Z'$. Since $Y'\setminus D'\to Y\setminus D$ is quasi-finite, $R$ is not exceptional over $Y$. But then $R$ is mapped onto a component of $\Lambda$ by $Y'\bir X$ because its coefficient in $\Lambda'$ is positive and because $X\bir Y$ is an isomorphism over $\eta_Z$ (see \ref{ss-log-disc-strict-transform}), so $R$ is mapped into $D$, hence $R$ is a component of $D'$, a contradiction. We have then proved the claim.

There is a prime divisor $T'$ over $Y'$ so that $T'$ maps to $z'$ and that $T'$ maps onto $T$ (on some resolution of $X$). We can assume that $T'$ is a divisor on $W'$. Moreover, by construction, $K_{W'}+\Lambda_{W'}\sim_\Q 0/Z'$, so $K_{W'}+\Lambda_{W'}$ is the pullback of $K_{Y'}+\Lambda'$. Then  
$$
a(T',Y',\Lambda')=a(T',W',\Lambda_{W'})=0
$$
where the second equality follows from $a(T,X,\Lambda)=0$ (see \ref{ss-log-disc-strict-transform}). Therefore, from 
$\Lambda'\le D'$ we get, 
$$
a(T',Y',D')\le a(T',Y',\Lambda')=0
$$ 
which in turn gives $a(T',Y',D')=0$ where we use the fact that $(Y',D')$ is lc, by Lemma \ref{l-toroidal-couple-lc}, as it is toroidal. 

On the other hand, 
$$
\deg(W'\to X)=\deg(Y'\to Y)\le r'
$$ 
where $r'$ is given by Theorem \ref{t-bnd-torification}. Letting $\Delta=\frac{1}{r'}B+(1-\frac{1}{r'})\Lambda$, we can see that $(X,\Delta)$ is $\frac{\epsilon}{r'}$-lc and that   
$$
a(T,X,\Delta)\le \frac{1}{r'}.
$$  
This implies that the coefficient of $T'$ in $\Delta_{W'}$ is non-negative where $K_{W'}+\Delta_{W'}$ is the pullback of $K_X+\Delta$ (by \ref{ss-log-disc-strict-transform}). 

Let $\Gamma_{W'}$ be the effective part of $\Delta_{W'}$. Since $(W',\Delta_{W'})$ is sub-$\frac{\epsilon}{r'}$-lc and since $(W',\Delta_{W'})$ is log smooth, $(W',\Gamma_{W'})$ is $\frac{\epsilon}{r'}$-lc. 
Let $\Gamma'$ be the pushdown of $\Gamma_{W'}$ to $Y'$. Recall $A'$ from Theorem \ref{t-bnd-torification}. We claim that $A'-\Gamma'$ is pseudo-effective over $Z'$, perhaps after replacing $A'$ with a bounded multiple. We can work over the generic point $\eta_{Z'}$. Over this point, $\rho\colon Y'\bir X$ is a morphism and letting $K_{Y'}+\Delta_{Y'}=\rho^*(K_X+\Delta)$, we see that $\Gamma'$ is the effective part of $\Delta_{Y'}$. So 
$\Gamma'\le \rho^*\Delta+D'$ because $D'$ contains all the exceptional prime divisors of $Y'\to Y$. We can assume that $A'-D'$ is pseudo-effective over $Z'$, so it is enough to show $A'-\rho^*\Delta$ is pseudo-effective over $Z'$. Since $-\Delta\equiv K_X/Z$, in turn it is enough to show $A'+\rho^*K_X$ is pseudo-effective over $Z'$. By construction, $A+K_X$ is nef over $\eta_{Z}$ and by Theorem \ref{t-bnd-torification}, $A'-\pi^*A$ is ample over $Z'$, so $A'+\rho^*K_X$ is pseudo-effective over $Z'$ as required. 

Now $K_{W'}+\Delta_{W'}\equiv 0/Z'$ and $T'$ is not a component of $\Gamma_{W'}-\Delta_{W'}\ge 0$. 
So running the MMP on $K_{W'}+\Gamma_{W'}$ over $Y'$ ends with a minimal model $(W'',\Gamma_{W''})$ and the MMP does not contract $T'$. Therefore, applying Theorem \ref{t-bnd-mult-lc-places-fib-main} to $(W'',\Gamma_{W''})$, $(Y',D')\to Z'$, $W''\to Y'$, $A'$, $z'$, we deduce that $\mu_{T'}G'$ is bounded from above where $G'$ is the fibre of $Y'\to Z'$ over $z'$. Since $\deg(Z'\to Z)\le r'$, we deduce that the coefficient of $T'$ in the fibre of $Y'\to Z$ over $z$ is bounded, and this in turn implies that $\mu_TG$ is bounded for the fibre $G$ of $Y\to Z$ over $z$.
Note that $T$ may not be a divisor on $Y$, so $\mu_TG$ is calculated on some resolution of $Y$ where $T$ appears (a similar remark applies to $\mu_{T'}G'$). Therefore, the multiplicity of $T$ with respect to the fibre of $X\to Z$ over $z$ is also bounded as desired, and we are done.
\end{proof}

\begin{proof}[Proof of Corollary \ref{cor-mckernan-conj}] 
Take a boundary $B$ so that $(X,B)$ is $\epsilon$-lc and $K_X+B\sim_\Q 0/Z$. Applying Theorem \ref{t-mc-sh-conj}, the generalised pair $(Z,B_Z+M_Z)$ on the base given by adjunction is $\delta$-lc where $\delta$ depends only on $d,\epsilon$. Since $K_Z$ is $\Q$-Cartier, it follows that $Z$ is $\delta$-lc. 
\end{proof}

\begin{proof}[Proof of Corollary \ref{cor-mult-fib}] 
Take $B$ so that $(X,B)$ is $\epsilon$-lc and $K_X+B\sim_\Q 0/Z$. Now by Theorem \ref{t-mc-sh-conj}, the lc threshold $t$ of $f^*z$ with respect to $(X,B)$ is bounded from below. Therefore, the coefficients of $f^*z$ are bounded from above.
\end{proof}

\subsection{Generalised Fano type contractions}

The next result extends Theorem \ref{t-mc-sh-conj} in several different directions by considering generalised pairs, real coefficients, and a base variety $S$. 

\begin{thm}\label{t-sh-sing-gen-fib}
Let $d\in \N$ and $\epsilon\in \R^{>0}$. 
 Then there is $\delta \in \R^{>0}$ depending only on $d,\epsilon$ satisfying the following. 
 Assume that $(X,B+M)$ is a generalised pair with data $X'\to X\to S$ and $M'$ 
such that 
\begin{itemize}
\item $(X,B+M)$ is generalised $\epsilon$-lc,

\item we have a contraction $f\colon X\to Z/S$ with $d=\dim X-\dim Z$, 

\item $K_X+B+M\sim_\R 0/Z$, and  

\item $-K_X$ is big over $Z$.
\end{itemize}
Then the generalised pair $(Z,B_Z+M_Z)$ over $S$ given by generalised adjunction 
$$
K_X+B+M\sim_\R f^*(K_Z+B_Z+M_Z),
$$
is generalised $\delta$-lc. 
\end{thm}
\begin{proof}
Note that here we view $(Z,B_Z+M_Z)$ as a generalised pair with data $Z'\to Z\to S$ and $M_{Z'}$ where $Z'\to Z$ is a high resolution and $M_{Z'}$ is the moduli part of adjunction on $Z'$. However, to show that $(Z,B_Z+M_Z)$ is generalised $\delta$-lc is a local condition over $Z$, so we may ignore $S$ and view $(Z,B_Z+M_Z)$ as a generalised pair with data $Z'\to Z\to Z$ and $M_{Z'}$ where $Z\to Z$ is the identity morphism.
  
First, we reduce the theorem to the case of usual pairs, that is, when $M'=0$. Using MMP, we can reduce to the situation where $B+M$ is ample over $Z$: indeed, initially taking a $\Q$-factorialisation we can assume $B+M$ is $\R$-Cartier; since $B+M$ is big over $Z$, it has a minimal model on which it is semi-ample over $Z$ (note that $X$ is of Fano type over $Z$, so we can run MMP on any $\R$-Cartier divisor); replacing $X$ with the ample model of $B+M$ over $Z$, we can assume $B+M$ is ample over $Z$.  

Pick a resolution $Z'\to Z$ and a prime divisor $D$ on $Z'$. We can assume $\phi\colon X'\to X$ is a resolution and that the induced map $f'\colon X'\bir Z'$ is a morphism. Let $K_{X'}+B'+M'$ be the pullback of $K_{X}+B+M$. We need to show that the generalised lc threshold of $f'^*D$ with respect to $(X',B'+M')$ over the generic point of $D$ is bounded from below away from $0$. Since $M'$ is nef over $Z$, $t$ is the lc threshold of $f'^*D$ with respect to $(X',B')$ over the generic point of $D$, i.e. $t$ is the largest number such that the coefficients of $B'+tf'^*D$ over the generic point of $D$ do not exceed $1$. 

Pick $s\in(0,1)$ sufficiently close to $1$. 
Then we can write 
$$
K_{X'}+B_s'+sM'=\phi^*(K_X+sB+sM)
$$
where $B'-B_s'$ has sufficiently small coefficients. 
Since $B+M$ is ample over $Z$, by taking a general 
$$
0\le L\sim_\R (1-s)(B+M)/Z
$$ 
and adding its pullback to $B_s'+sM'$, 
we can find $\Delta'\sim_\R B'+M'/Z$ so that 
\begin{itemize}
\item $(X',\Delta')$ is sub-$\frac{\epsilon}{2}$-lc, 
\item $\Delta'\ge B_s'$, and
\item if $\Delta$ is the pushdown of $\Delta'$ to $X$, then $\Delta\ge 0$. 
\end{itemize}
In particular, $(X,\Delta)$ is $\frac{\epsilon}{2}$-lc, and we can assume that if $u$ is the lc threshold of $f'^*D$ with respect to $(X',\Delta')$ over the generic point of $D$, then $u-t$ is sufficiently small (positive or negative). Therefore, we can reduce the problem to the case of usual pairs by considering adjunction for $(X,\Delta)$ over $Z$. From now on then we assume $M'=0$. 
 
On the other hand, to reduce to Theorem \ref{t-mc-sh-conj} we need to reduce to the case when $B$ is a $\Q$-boundary. Using approximation, we can find real numbers $r_i\ge 0$ and $\Q$-boundaries $B_i$ so that $\sum r_i=1$ and $B=\sum r_iB_i$. Moreover, we can ensure that $K_{X}+B_i\sim_\Q 0/Z$ and that $(X,B_i)$ are $\frac{\epsilon}{2}$-lc. 

We use the notation introduced above in the third paragraph. 
Let $K_{X'}+B_i'$ be the pullback of $K_{X}+B_i$. Let $t_i$ be the lc threshold of $f'^*D$ with respect to $(X',B_i')$ over the generic point of $D$. Then since $B'=\sum r_iB_i'$, we deduce $t\ge \sum r_it_i$. Applying Theorem \ref{t-mc-sh-conj}, there is a positive real number $\delta$ depending only on $d,\epsilon$ such that $t_i\ge \delta$. But then $t\ge \sum r_i\delta=\delta$. So the discriminant b-divisor ${\bf B}_Z$  has coefficients $\le 1-\delta$, hence $(Z,B_Z+M_Z)$ is generalised $\delta$-lc.
\end{proof}

\subsection{Klt complements}

Next, we aim to prove Corollary \ref{cor-klt-comp}.

\begin{lem}\label{l-fixed-movable-div-fib}
Assume that 
\begin{itemize}
\item $f\colon X\to Z$ is a contraction from a normal variety onto a smooth curve, 
\item  $D\ge 0$ is a $\Q$-Cartier integral divisor on $X$, and 
\item for the general fibres $F$ of $f$, $D|_F$ is base point free. 
\end{itemize}
 Then there exists a log  resolution $\phi\colon W\to X$ so that we can write $\phi^*D=S+M$ where $0\le S\le \phi^*D$ is vertical over $Z$ and $M$ is base point free over $Z$.  
\end{lem}
\begin{proof}
Let $\psi\colon X'\to X$ be a resolution. Let $F$ be a general fibre of $f$ and $F'$ the corresponding fibre of $X'\to Z$. Then $(\psi^*D)|_{F'}$ is the pullback of $D|_F$. Since the latter is base point free, it is Cartier, hence the former is Cartier too which implies that $\psi^*D$ is integral and Cartier over the generic point $\eta_Z$. 
It is enough to prove the lemma for $X'\to Z$, $\rddown{\psi^*D}$ instead of $X\to Z,D$. Thus we can assume that $X$ is smooth. Moreover, letting $\overline{Z}$ be a smooth compactification of $Z$, we can assume that $X$ is an open subset of a smooth projective variety $\overline{X}$ over $\overline{Z}$. 
Let $\overline{D}$ be the closure of $D$ in $\overline{X}$. It is enough to prove the lemma for $\overline{X}\to \overline{Z}$, $\overline{D}$. Thus replacing $X\to Z$, $D$ with $\overline{X}\to \overline{Z}$, $\overline{D}$, we can assume that $X,Z$ are smooth and projective. 

Let $F$ be a general fibre of $f$, say over a point $z$. We can assume that $R^1f_*\mathcal{O}_X(D)$ is locally free near $z$. 
 Then taking direct image of the exact sequence 
$$
0\to \mathcal{O}_X(D-F)\to  \mathcal{O}_X(D)\to \mathcal{O}_F(D|_F)\to 0
$$
we get an exact sequence 
$$
0\to f_*\mathcal{O}_X(D-F)\to  f_*\mathcal{O}_X(D)\to f_*\mathcal{O}_F(D|_F)\to 0
$$
because the sequence is clearly exact outside $z$ and because 
$$
R^1f_*\mathcal{O}_X(D-F)=R^1f_*\mathcal{O}_X(D)\otimes_{\mathcal{O}_Z} \mathcal{O}_Z(-z) \to R^1f_*\mathcal{O}_X(D)
$$ 
is injective near $z$ as $R^1f_*\mathcal{O}_X(D)$ is locally free near $z$.

Now let $m\gg 0$ be a natural number and consider the induced exact sequence
$$
0\to f_*\mathcal{O}_X(D+(m-1)F)\to  f_*\mathcal{O}_X(D+mF)\to f_*\mathcal{O}_F(D|_F)\to 0.
$$ 
Then 
$$
H^1(f_*\mathcal{O}_X(D+(m-1)F))=H^1(f_*\mathcal{O}_X(D)\otimes_{\mathcal{O}_Z} \mathcal{O}_Z((m-1)z))=0
$$ 
by Serre vanishing as $m\gg 0$, so we get a surjection 
$$
H^0(f_*\mathcal{O}_X(D+mF))\to H^0(f_*\mathcal{O}_F(D|_F))
$$
which means we get a surjection 
$$
H^0(D+mF)\to H^0(D|_F).
$$
Then the base locus of the linear system $|D+mF|$ is disjoint from $F$, hence this base locus is vertical over $Z$. Therefore, there exists a log resolution $\phi\colon W\to X$ so that we can write $\phi^*(D+mF)=S+N$ where
$S$ is the fixed part of $\phi^*(D+mF)$ and $N$ is free. Since $D+mF$ is effective, 
$S\le \phi^*(D+mF)$. Moreover, since the base locus of $|D+mF|$ is vertical over $Z$ and it does not intersect $F$, $S$ is vertical over $Z$ with no common component with $\phi^*F$, hence in particular $S\le \phi^*D$. 

Now let $M=N-\phi^*mF$. Then $M$ is base point free over $Z$ and $\phi^*D=S+M$.
\end{proof}

\begin{proof}[Proof of Corollary \ref{cor-klt-comp}]
By Theorem \ref{t-mc-sh-conj} and Corollary \ref{cor-mult-fib}, there is a fixed natural number $r>0$ such that $(X,B:=\frac{1}{r}f^*z)$ is klt and the coefficients of $B$ belong to a fixed finite set. Then by Theorem \ref{t-bnd-compl-usual-local}, there is an $n$-complement $K_X+\Lambda$ for $K_X+B$ over $z$, for some fixed number $n\ge 2$ depending only on $d$ and the coefficients of $B$, where $\Lambda\ge B$. We can assume that $n(K_X+B)$ is integral, hence $n(\Lambda-B)$ is integral. Also shrinking $Z$ around $z$, we have 
$$
n(\Lambda-B)\sim -n(K_X+B).
$$

Since $X$ is $\epsilon$-lc, a general fibre $F$ of $f$ is an $\epsilon$-lc Fano variety of dimension $d-1$, hence it belongs to a fixed bounded family, by the BAB [\ref{B-BAB}, Theorem 1.1]. Thus we can choose $n$ so that $-nK_F=-n(K_X+B)|_F$ is base point free which also means that $n(\Lambda-B)|_F$ is base point free. Therefore, by Lemma \ref{l-fixed-movable-div-fib}, there exists a log  resolution $\phi\colon W\to X$ so that we can write 
$$
\phi^*n(\Lambda-B)=S+M
$$ 
where $0\le S\le \phi^*n(\Lambda-B)$ is vertical over $Z$ and $M$ is base point free over $Z$. 

Let $K_W+B_W=\phi^*(K_X+B)$. Then 
$$
K_W+B_W+\phi^*(\Lambda-B)=\phi^*(K_X+\Lambda)
$$  
and since $(X,\Lambda)$ is lc over $z$, the coefficients of $B_W+\phi^*(\Lambda-B)$ over $z$ do not exceed $1$. 
Thus since $0\le S\le \phi^*n(\Lambda-B)$, the coefficients of $B_W+\frac{1}{n}S$ over $z$ do not exceed $1$. 

Now choose a general $L\sim M/Z$. Then the coefficients of $B_W+\frac{1}{n}(S+L)$ over $z$ do not exceed $1$. Moreover, 
$$
n(K_W+B_W+\frac{1}{n}(S+L))\sim n\phi^*(K_X+\Lambda)\sim 0/Z.
$$
Therefore, letting 
$$
\Lambda':=\phi_*(B_W+\frac{1}{n}(S+L)),
$$ 
we get an $n$-complement $K_X+\Lambda'$ of $K_X+B$ over $z$ which is klt over the generic point $\eta_Z$. Replacing $\Lambda$ with $\Lambda'$, we can then assume that  
 $(X,\Lambda)$ is klt over $\eta_Z$. Now 
$K_X+\Lambda-B$ is a klt $rn$-complement of $K_X$ over $z$. We can then replace $n$ with $rn$.
\end{proof}

\subsection{Special fibres of del Pezzo fibrations}

\begin{proof}[Proof of Theorem \ref{t-gonality}]
Let $z$ be the image of $S$ on $Z$. 
By Theorem \ref{t-mc-sh-conj}, there is $t>0$ depending only on $\epsilon$ such that $(X,tf^*z)$ is klt. Since $-(K_X+tf^*z)$ is ample over $Z$, we can find $B\ge tf^*z$ such that $(X,B)$ is klt and $K_X+B\sim_\Q 0/Z$. Now apply [\ref{BL}, Theorem 1.3].
\end{proof}


\section{\bf Rationally connected varieties with nef $-K$}

In this section we prove Theorems \ref{t-mc-pro-conj} and \ref{t-bnd-rc-gen-pairs}. The latter implies the former. The strategy of the proof of \ref{t-bnd-rc-gen-pairs} is similar to the one given in [\ref{BDS}][\ref{CDCH+}] in dimension $3$ which splits the theorem into two cases: when $K_X\not \equiv 0$ and when $K_X\equiv 0$. In the first case, we use the MMP to reduce to the situation when we have a Fano type fibre structure in which case we can apply Theorem \ref{t-mc-sh-conj}, [\ref{B-FT-fib}, Theorem 2.3] and induction on dimension. In the second case, we need a different argument than the one given in [\ref{BDS}][\ref{CDCH+}] and this is the next proposition.

\subsection{Coefficients of certain pairs}

\begin{prop}\label{p-bnd-coeff-bnd-fam-kappa=0}
Let $\mathcal{P}$ be a bounded set of projective varieties. 
Then there is a real number $\alpha>0$ depending only on 
$\mathcal{P}$ satisfying the following. 
Assume that 
\begin{itemize}
\item $(X,B)$ is a projective klt pair with $X\in \mathcal{P}$, 
\item $K_X+B\equiv 0$ and $B\neq 0$, and
\item $\kappa_\sigma(-K_X)=\kappa_\sigma(B)=0$. 
\end{itemize}
Then some coefficient of $B$ is $\ge \alpha$.
\end{prop}
\begin{proof}
\emph{Step 1.}
If the proposition does not hold, then there is a sequence of pairs $(X_i,B_i)$ as in the 
proposition such that the coefficients of the $B_i$ converge to $0$. Replacing the sequence, we can assume that there is a flat projective morphism  
$V\to T$ of normal varieties such that $X_i$ is the fibre over a closed point $t_i\in T$ 
and such that $\{t_i\}$ is dense in $T$. In addition, taking a resolution of $V$ and using MMP, we can assume that $V$ is $\Q$-factorial and $V$ has klt singularities: indeed, take a resolution $W\to V$ and let $V'$ be a minimal model of $W$ over $V$; denote the fibre of $V'\to T$ over a general $t_i$ by $X_i'$ and let $K_{X_i'}+B_i'$ be the pullback of $K_{X_i}+B_i$ under the induced morphism $X_i'\to X_i$; then $K_{X_i'}$ is nef over $X_i$, so by the negativity lemma, we have $B_i'\ge 0$; moreover, $\kappa_\sigma(-K_{X_i'})\le \kappa_\sigma(-K_{X_i})$, hence the former is also zero; so we can replace $(X_i,B_i)$ with $(X_i',B_i')$ and replace $V$ with $V'$ (and replace $\mathcal{P}$ accordingly).

\emph{Step 2.} In this step we show that $-K_V$ is pseudo-effective over $T$.
Let $F$ be the generic fibre of $V\to T$. It is enough to show that $-K_F$ is pseudo-effective. 
Assume this is not the case. Pick a very ample$/T$ divisor $A$ on $V$ and let $A_F=A|_F$. There is $p\in \N$ such that $-pK_F+A_F$ is not pseudo-effective. Then $-rK_F+A_F$ is not pseudo-effective for any real number $r\ge p$.

For the rest of this step, we fix a sufficiently large natural number $l$ (e.g. $l\ge 3\dim F$). Let $A_i:=A|_{X_i}$. Since $-K_{X_i}$ is pseudo-effective 
and $A_i$ is very ample, the divisor $-mK_{X_i}+lA_i$ is potentially birational (see 
[\ref{HMX1}, Definition 2.3.3]) for any 
natural number $m$ and any $i$. Thus the linear system 
$$
|K_{X_i}-mK_{X_i}+lA_i|
$$ 
defines a birational map [\ref{HMX1}, Lemma 2.3.4] for any $m,i$. But then 
for each $m$ such that $(m-1)K_V$ is Cartier, we can find an $i$ such that  
$$
h^0(K_F-mK_{F}+lA_F)=h^0(K_{X_i}-mK_{X_i}+lA_i)>0
$$
where the left hand $h^0$ is dimension over the function field of $T$ while the right hand $h^0$ is dimension over $k$. 
The equality follows from base change of cohomology and from the assumption that $\{t_i\}$ is dense in $T$: indeed, we can apply [\ref{Hartshorne-AG}, Chapter III, Theorem 12.11] noting that all the sheaves $R^if_*\mathcal{O}_V(K_V-mK_V+lA)$ are locally free on some open subset of $T$ depending on $m$, where $f$ denotes $V\to T$.
Thus 
$$
-(m-1)K_F+lA_F
$$ 
is pseudo-effective contradicting the 
previous paragraph as soon as $\frac{m-1}{l}\ge p$.

\emph{Step 3.} Pick a natural number $n$ such that $nK_V$ is Cartier. We show that 
$$
h^0(-mnK_{F}+lA_F)
$$
is a bounded function of $m$.
For each $i$, by assumption $\kappa_\sigma(-K_{X_i})=0$, hence  
$$
h^0(-mnK_{X_i}+lA_i)
$$
is a bounded function of $m$ for each natural number $l$: indeed, let $\phi\colon Y_i\to X_i$ be a resolution; then $\kappa_\sigma(\phi^*(-K_{X_i}))=0$, hence $\kappa_\sigma^{-}(\phi^*(-K_{X_i}))=0$ [\ref{Nakayama}, Chapter V, Remark 2.6(4)] which in turn implies $\phi^*(-K_{X_i})\equiv N_\sigma(\phi^*(-K_{X_i}))$ [\ref{Nakayama}, Chapter V, Proposition 2.7(8)], so the function $h^0(mn\phi^*(-K_{X_i}))+lA_{Y_i})$ is bounded [\ref{Nakayama}, Chapter V, Corollary 1.12] where $A_{Y_i}$ is a sufficiently ample divisor so that $A_{Y_i}-\phi^*A_i$ is very ample; but then $h^0(-mnK_{X_i}+lA_i)$ is bounded as claimed. 
Therefore, by upper semi-continuity of cohomology, 
$
h^0(-mnK_{F}+lA_F)
$
is also a bounded function of $m$. 

\emph{Step 4.}
We show that $-K_F\not\equiv 0$. Assume not. Then $K_F\sim_\Q 0$ and there is $l$ such that 
$$
h^0(mnK_{F}+lA_F)>0
$$
for infinitely many $m$. Thus for each $i$, by upper semi-continuity of cohomology, we get 
$$
h^0(mnK_{X_i}+lA_i)>0
$$
for the same set of $m$. This implies that $K_{X_i}$ is pseudo-effective contradicting 
the assumption $B_i\neq 0$.
 
\emph{Step 5.}
Let $K=k(T)$ and let $\overline{K}$ be its algebraic closure. Let $\overline{F}$ be the base change of $F$ to $\overline{K}$. Then  
$
h^0(-mnK_{\overline{F}}+lA_{\overline{F}})
$
is a bounded function of $m$. Let $\psi\colon \overline{F}'\to \overline{F}$ be a resolution. Then   
$$
\psi^*(-K_{\overline{F}})\equiv {N}_{\overline{F}'}:=N_\sigma(\psi^*(-K_{\overline{F}}))\ge 0
$$
by arguing similar to Step 3 and using [\ref{Nakayama}, Chapter V, Corollary 1.12]. Therefore, 
$$
-K_{\overline{F}}\equiv N_{\overline{F}}:=\psi_*N_{\overline{F}'}\ge 0
$$
where we use the fact that $\overline{F}$ has klt singularities so the numerical equivalence on $\overline{F}'$ can be pushed down to $\overline{F}'$ (this is a consequence of MMP and the cone theorem).

 Here $N_{\overline{F}}$ is an $\R$-divisor, whose irreducible components are determined by finitely many equations over $\overline{K}$. So $N_{\overline{F}}$ can also be defined over a finite extension of $K$, hence after a finite base change of $V\to T$, we can assume that ${N}_{\overline{F}}$ is the pullback of an $\R$-divisor $N_F\ge 0$ on $F$. Then $-K_F\equiv N_F$. 
 
 We can assume that the base change of the previous paragraph is \'etale hence preserving the klt property of $V$. Taking a $\Q$-factorialisation of $V$, we can assume it is still $\Q$-factorial (this replaces  $X_i$ by a small birational modification preserving our assumptions).

\emph{Step 6.}
 Let $N$ on $V$ be the closure of $N_F$. Pick a $\Q$-divisor $0\le E\le N$ on $V$ so that $E_F:=E|_F\neq 0$. Then $-K_F-E_F:=(-K_V-E)|_F$ is pseudo-effective. 
Shrinking $T$, we can assume no $X_i$ is contained in the support of $E$. 
In particular, $E_i:=E|_{X_i}\ge 0$ for each $i$ and the coefficients of $E_i$ 
are $\ge \alpha>0$ where $\alpha$ is independent of $i$.
 
Replacing $n$, we can assume $nE$ is Cartier. Then, since $-K_F-E_F$ is pseudo-effective, 
we can choose $l$ such that 
$$
h^0(mn(-K_F-E_F)+lA_F)>0
$$
for infinitely many $m$. Thus by upper semi-continuity of cohomology, for each $i$, we get 
$$
h^0(mn(-K_{X_i}-E_i)+lA_i)>0
$$
for the same set of $m$. This implies that $-K_{X_i}-E_i$ is pseudo-effective.

\emph{Step 7.} 
Since $\kappa_\sigma(-K_{X_i})=0$ and since $-K_{X_i}\sim_\R B_i\ge 0$, we have 
$N_\sigma(-K_{X_i})=B_i$.
Moreover, since $-K_{X_i}-E_i$ is pseudo-effective and $E_i\ge 0$, $\kappa_\sigma(-K_{X_i}-E_i)=0$. 
Thus 
$$
-K_{X_i}-E_i\equiv N_\sigma(-K_{X_i}-E_i)\ge 0
$$ 
by arguments similar to the steps above, which in turn implies   
$$
E_i+N_\sigma(-K_{X_i}-E_i)\equiv -K_{X_i} \equiv B_i.
$$ 
Therefore, since $\kappa_\sigma(-K_{X_i})=0$, we get $E_i+N_\sigma(-K_{X_i}-E_i)= B_i$,  hence some 
coefficient of $B_i$ is $\ge \alpha$, a contradiction.
\end{proof}

\subsection{Proof of Theorems \ref{t-mc-pro-conj} and \ref{t-bnd-rc-gen-pairs}}

\begin{lem}\label{l-bnd-fam-bnd-mmodel}
Let $\epsilon\in \R^{>0}$ and $\Phi \subset [0,1]$ be a finite set of rational numbers. 
Assume that $\mathcal{P}$ is a bounded set of couples. Consider pairs $(X,B)$ where   
\begin{itemize}
\item $(X,B)$ is projective $\epsilon$-lc,
\item $(X,\Supp B)\in \mathcal{P}$, 
\item the coefficients of $B$ belong to $\Phi$, and 
\item $K_X+B$ is pseudo-effective with $\kappa_\sigma(K_X+B)=0$.
\end{itemize} 
Then there exists a bounded set of minimal models for such $(X,B)$.
\end{lem}
\begin{proof}
Taking a bounded log resolution, we can assume that $(X,B)$ is log smooth: more precisely, we can take a log resolution $\phi\colon W\to X$ so that letting $B_W=B^\sim+(1-e)E$, $(W,\Supp B_W)$ is bounded, where $B^\sim$ is the birational transform of $B$, $E$ is the reduced exceptional divisor of $\phi$, and $e\in (0,\epsilon)$ is a fixed rational number. 
Any minimal model $(Y,B_Y)$ of $(W,B_W)$ is also a minimal model of $(X,B)$ which follows from the arguments in [\ref{B-lc-flips}, Remark 2.8]; but since the definition of minimal models is slightly different in this paper, we present a sketch here;  taking a common resolution $\alpha\colon V\to X$ and $\beta \colon V\to Y$ and applying the negativity lemma shows that 
$$
\alpha^*(K_X+B)-\beta^*(K_Y+B_Y)
$$ 
is effective; in particular, any prime divisor $D$ contracted by $Y\bir X$ satisfies 
$$
\epsilon\le a(D,X,B)\le a(D,Y,B_Y)=e
$$
which is not possible, hence $Y\bir X$ does not contract divisors; also, any prime divisor $G$ contracted by $X\bir Y$ satisfies 
$$
a(G,X,B)=a(G,W,B_W)<a(G,Y,B_Y)
$$
which implies that $(Y,B_Y)$ is a minimal model of $(X,B)$.  
Now $(W,B_W)$ is log smooth, $e$-lc, and the coefficients of $B_W$ belong to $\Phi\cup \{1-e\}$. Therefore, we may replace $\epsilon,\Phi, (X,B)$ with $e$, $\Phi\cup \{1-e\}$, $(W,B_W)$, respectively, and replace $\mathcal{P}$ accordingly.

Shrinking $\mathcal{P}$, we can assume that there is a relatively log smooth family $(V,C)\to T$ so that each $(X,B)$ is isomorphic to the log fibre $(V_t,C_t)$ of $(V,C)\to T$ over a closed point $t\in T$ and such that the set of such $t$ are dense in $T$. It is enough to show that $(V,C)$ has a minimal model over $T$, possibly after a finite base change and shrinking $T$. 

Since $K_{V_t}+C_t$ is pseudo-effective for a dense set of points $t$, $K_F+C_F$ is pseudo-effective for the generic fibre $F$ (otherwise we get a Mori fibre space for $(V,C)$ over $T$). Moreover, since $\kappa_\sigma(K_{V_t}+C_t)=0$ for some $t$, $\kappa_\sigma(K_F+C_F)=0$ by arguing similar to the proof of Proposition \ref{p-bnd-coeff-bnd-fam-kappa=0} using upper semi-continuity of cohomology. 

Let $K=k(T)$ and let $\overline{K}$ be its algebraic closure. Let $\overline{F}$ be the base change of $F$ to $\overline{K}$. Then  
$
\kappa_\sigma(K_{\overline{F}}+C_{\overline{F}})=0,
$
so applying [\ref{Gongyo}], $(\overline{F},C_{\overline{F}})$ has a good minimal model $(\overline{F}',C_{\overline{F}'})$ satisfying $K_{\overline{F}'}+C_{\overline{F}'}\sim_\Q 0$. But then this good minimal model can be defined over a finite extension of $K$. Therefore, taking the corresponding finite base change and replacing  $T$, we can assume that $(F,C_F)$ has a good minimal model, say $(F',C_{F'})$, which satisfies $K_{F'}+C_{F'}\sim_\Q 0$. Taking a closure of $(F',C_{F'})$ over $T$, and shrinking $T$, gives a good minimal model $(V',C')$ of $(V,C)$ over $T$ with $K_{V'}+C'\sim_\Q 0/T$. 
\end{proof}

\begin{proof}[Proof of Theorem \ref{t-bnd-rc-gen-pairs}]
We can assume that $X$ is $\Q$-factorial.
If $K_X\not\equiv 0$, then the theorem follows from the arguments of [\ref{BDS}](also see [\ref{CDCH+}]) and the results of [\ref{B-FT-fib}] together with Theorem \ref{t-mc-sh-conj} by reducing the theorem to lower dimension. We give a sketch for convenience. Run an MMP on $K_X$ which ends with a Fano fibration 
$f''\colon X''\to Z''$. If $\dim Z''=0$, then $X''$ is an $\epsilon$-lc Fano variety, so we apply [\ref{B-BAB}] to deduce that $X''$ is bounded. Assume $\dim Z''>0$. By generalised adjunction (\ref{s-adjunction}), we can write 
$$
K_{X''}+B''+M''\sim_\R f''^*(K_{Z''}+B_{Z''}+M_{Z''})
$$
where we consider $({Z''},B_{Z''}+M_{Z''})$ as a projective generalised pair. 
By Theorem \ref{t-sh-sing-gen-fib}, this pair is generalised $\delta$-lc for some 
$\delta>0$ depending only on $d,\epsilon$. Moreover, $Z''$ is rationally connected.

By induction, $Z''$ is bounded up to isomorphism in codimension one, say it is isomorphic 
in codimension one to $Z'''$ which belongs to some bounded family. 
One then shows that there exists a diagram 
$$
\xymatrix{
X''\ar[d] \ar@{-->}[r]& X'''\ar[d]\\
Z'' \ar@{-->}[r] & Z'''
}
$$
where the horizontal maps are isomorphisms in codimension one and 
$X'''\to Z'''$ is a Fano type contraction. Replacing $X'',Z''$ with $X''',Z'''$ we can assume 
that $Z''$ belongs to a bounded family.  
Now by [\ref{B-FT-fib}, Theorem 2.3], $X''$ belongs to a bounded family. 

Thus we have shown that in both cases $\dim Z''=0$ and $\dim Z''>0$, $X''$ belongs to a bounded family. On the other hand, 
extracting the exceptional prime divisors of $X\bir X''$, we can construct a birational 
contraction $\overline{X}\to X''$ so that the induced map $X\bir \overline{X}$ is 
an isomorphism in codimension one. Applying [\ref{B-FT-fib}, Theorem 2.3] once more, this time to $\overline{X}\to X''$, we 
deduce that $\overline{X}$ belongs to a bounded family.

We can then assume $K_X\equiv 0$. Since $X$ is rationally connected, the singularities of $X$ are worse than canonical, that is, there is a prime divisor $D$ over $X$ such that $a(D,X,0)<1$: if not, then taking a resolution $\phi\colon W\to X$, we have $K_W=\phi^*K_X+E$ for some $E\ge 0$; but then $K_W\sim_\Q E\ge 0$ which contradicts the fact that $W$ is rationally connected [\ref{Debarre}, Corollary 4.18(a)].

Let $X'\to X$ be the birational 
contraction which extracts $D$ and let $K_{X'}+B'$ be the pullback of $K_X$. Then $B'\neq 0$.
By the above arguments, $X'$ is isomorphic in codimension one to a 
normal projective variety $X''$ which belongs to a bounded family $\mathcal{P}$.
By [\ref{B-FT-fib}, theorem 1.2], we can assume that $X''$ is $\Q$-factorial. Note that 
$$
\kappa_\sigma(-K_{X''})=\kappa_\sigma(B'')=0
$$
because $B''$ is exceptional over $X$.

By Proposition \ref{p-bnd-coeff-bnd-fam-kappa=0}, 
the coefficient of $B''$ is bounded from below by some fixed $\alpha>0$. In particular, 
$(X'',\Supp B'')$ belongs to a bounded family. Computing intersection 
numbers on $X''$ shows that the coefficient of $B''$ actually belongs to some fixed finite set. 

There is a fixed rational number $t>0$ such that $(X'',B''+t B'')$ is $\frac{\epsilon}{2}$-lc (cf. [\ref{B-BAB}, Theorem 1.8]).
Therefore,  there is a minimal model $\tilde{X}$ of $(X'',B''+t B'')$ 
which belongs to a bounded family, by Lemma \ref{l-bnd-fam-bnd-mmodel}. However, since $K_{X''}+B''\equiv 0$ and $B''$ is exceptional over $X$, 
$$
N_\sigma(K_{X''}+B''+t B'')=t B'',
$$ 
so the divisor contracted 
by the map $X''\bir \tilde{X}$ is $\Supp B''$. In other words, $X$ is isomorphic in 
codimension one with $\tilde{X}$.
\end{proof}

\begin{proof}[Proof of Theorem \ref{t-mc-pro-conj}]
This follows from Theorem \ref{t-bnd-rc-gen-pairs} by putting $M:=-(K_X+B)$ and considering $(X,B+M)$ as a generalised pair.

\end{proof}


\subsection{Proof of Corollary \ref{t-cy-index-conj}}

\begin{proof}[Proof of Corollary \ref{t-cy-index-conj}]
From the global ACC result [\ref{HMX2}, Theorem 1.5] it can be deduced that the coefficients of $B$ belong to a fixed finite subset of $\Phi$, and that there exists a fixed $\epsilon>0$ such that $(X,B)$ is $\epsilon$-lc (for existence of $\epsilon$, see [\ref{B-Fano}, Lemma 2.48]; note that the lemma assumes $(X,0)$ to be klt but this is unnecessary as we can replace $(X,B)$ in the lemma with a $\Q$-factorial dlt model).

Now applying Theorem \ref{t-mc-pro-conj}, $X$ is bounded up to isomorphism in codimension one, that is, there exists a bounded projective variety $X'$ so that $X$ is isomorphic to $X'$ in codimension one. We may replace $X'$ with its normalisation, hence assume it is normal. If $B'$ is the birational transform of $B$, then $(X',B')$ is klt Calabi-Yau. Thus replacing $(X,B)$ with $(X',B')$, we can assume that $X$ belongs to a bounded family.

We can find a very ample divisor $A$ on $X$ so that $A^d$ and $-K_X\cdot A^{d-1}=B\cdot A^{d-1}$ are bounded. Thus $(X,B)$ belongs to a bounded family of pairs. Applying [\ref{B-FT-fib}, Lemma 3.16], $I(K_X+B)$ is Cartier for some fixed natural number $I$. But then $I(K_X+B)\sim 0$ because $I(K_X+B)\sim_\Q 0$ and because $X$ is rationally connected.
\end{proof}


\section{\bf Singularities on log Calabi-Yau fibrations}

In this section we prove Theorem \ref{t-cb-sing-usual-fib}. We 
actually prove a generalised version of this. 

\begin{thm}\label{t-cb-sing-gen-fib}
Let $d,v\in \N$ and $\epsilon\in \R^{>0}$. 
 Then there is a positive real number $\delta\in \R^{>0}$ depending only on $d,v,\epsilon$ satisfying the following. 
 Assume that $(X,B+M)$ is a generalised pair with data $X'\to X\to S$ and $M'$ such that
\begin{itemize}
\item $(X,B+M)$ is generalised $\epsilon$-lc,

\item we have a contraction $f\colon X\to Z/S$ with $\dim X-\dim Z\le d$,

\item $K_X+B+M\sim_\R 0/Z$, and  

\item there is an integral divisor $N$ on $X$ which is big over $Z$ with $\vol(N|_F)<v$  
for the general fibres $F$ of $f$.
\end{itemize}
Then the generalised pair $(Z,B_Z+M_Z)$ over $S$ given by generalised adjunction 
$$
K_X+B+M\sim_\R f^*(K_Z+B_Z+M_Z),
$$
is generalised $\delta$-lc.
\end{thm} 
\begin{proof}
Here we view $(Z,B_Z+M_Z)$ as a generalised pair with data $Z'\to Z\to S$ and $M_{Z'}$ where $Z'\to Z$ is a high resolution and $M_{Z'}$ is the moduli part of adjunction on $Z'$. However, since the conclusion is local over $Z$, we may replace $S$ and assume $Z\to S$ is the identity morphism.

\emph{Step 1.}
Taking a $\Q$-factorialisation, we can assume that $X$ is $\Q$-factorial. Similar to Theorem \ref{t-sh-sing-gen-fib}, we can reduce the theorem to the case $\dim Z=1$. Replacing $X$ with the minimal model of $N$ over $Z$, we can assume that $N$ is nef and big over $Z$. The minimal model exists by [\ref{BZh}, Lemma 4.4] as $N$ is big over $Z$. Moreover, 
applying [\ref{B-geom-pol-var}, Theorem 1.1] to the geometric generic fibre of $f$ (or applying to the general fibres and use base change of cohomology) and replacing $N$ with a bounded multiple, we can assume that $N\ge 0$. 

First, assume that $B+M\equiv 0$ over the generic point $\eta_Z$. Then $K_X\equiv 0$ over $\eta_Z$, hence $K_X\sim_\Q 0$ over $\eta_Z$, so $B+M\sim_\R 0$ over $\eta_Z$. Thus $B$ is vertical over $Z$ and the nef part $M'\sim_\R 0$ over $\eta_Z$. But then since $Z$ is a curve, $M'\sim_\R 0/Z$: indeed, $M'\sim_\R L'/Z$ where $L'\le 0$ is vertical over $Z$ and its support does not contain some component of each fibre of $X'\to Z$; this is possible only if $L'=0$ because $L'$ is nef over $Z$. This implies that $M'$ does not contribute to the singularities of $(Z,B_Z+M_Z)$, so we can assume $M'=0$. 
Now apply [\ref{B-lcyf}, Theorem 1.9].

\emph{Step 2.}
So from now on we assume $B+M\not\equiv 0$ over $\eta_Z$. In particular, $K_X$ is not pseudo-effective over $Z$.
By the previous step, $N$ is nef and big over $Z$. Let $t$ be the smallest number such that $K_X+tN$ is pseudo-effective over $Z$. By [\ref{B-geom-pol-var}, Lemma 4.11], $t$ is bounded from above. Moreover, $t$ is a rational number because $(X,tN)$, considered as a generalised pair with nef part $tN$, has a good minimal model over $Z$ [\ref{BZh}, Lemma 4.4] which is not of general type, so considering intersection numbers shows $t$ is rational.

We reduce the theorem to the case when $t\ge 1$. Let $l\in \Z^{\ge 0}$ be the largest number such that $L:=lK_X+N$ is big over $Z$. Note that 
$$
\vol(lK_F+N|_F)=\vol(-lB|_F-lM|_F+N|_F)\le \vol(N|_F)<v
$$ 
where $F$ is a general fibre of $X\to Z$.
Let $X''$ be the minimal model of $L$ over $Z$. If $X\bir X''$ contracts a divisor, then we replace $(X,B+M),N$ with $(X'',B''+M''),L''$ and repeat the process. After finitely many times, we can assume that $X\bir X''$ does not contract any divisor. Note that it is possible that after this process we get $B+M\equiv 0$ over $\eta_Z$ in which case we apply Step 1. So we can assume that we still have $B+M\not\equiv 0$ over $\eta_Z$.

Since $K_X+L$ is not big over $Z$, $K_{X''}+L''$ is also not big over $Z$ where we use the fact that $X\bir X''$ is an isomorphism in codimension one. Assume $s\in \R^{> 0}$ is the smallest number such that 
$K_{X''}+sL''$ is pseudo-effective over $Z$. Then
$$
\frac{1}{s}K_{X''}+L''=\frac{1}{s}K_{X''}+lK_{X''}+N''
$$ 
is pseudo-effective over $Z$. If $\frac{1}{s}+l>1+l$, then since $lK_X+N$ is big over $Z$, $(1+l)K_X+N$ is big over $Z$ which contradicts 
our choice of $l$; again we use the fact that $X\bir X''$ is an isomorphism in codimension one. Thus $\frac{1}{s}+l\le 1+l$, so 
$s\ge 1$. Therefore, replacing $(X,B+M),N$ with $(X'',B''+M''),L''$, we can assume that the threshold $t\ge 1$.
But then the proof of [\ref{B-geom-pol-var}, Lemma 4.11] shows that there are finitely many possibilities for $t$.
 
\emph{Step 3.}
Again view $(X,tN)$ as a generalised pair over $Z$ with nef part $tN$, which is generalised $\epsilon$-lc. Let $X''$ be the minimal model of $K_X+tN$ over $Z$ (to keep the notation simple we are again using $X''$; not to be confused with $X''$ of the previous step). Since $N''$ is big over $Z$, $K_{X''}+tN''$ is semi-ample over $Z$ defining a non-birational contraction $f''\colon X''\to Y''/Z$. Also $(X'',tN'')$ is generalised $\epsilon$-lc. Since $N''$ is big over $Z$, $-K_{X''}$ is big over $Y''$. By \ref{s-adjunction}, we can write a generalised adjunction formula 
$$
K_{X''}+tN''\sim_\Q f''^*(K_{Y''}+C_{Y''}+R_{Y''}),
$$
and by Theorem \ref{t-sh-sing-gen-fib}, $({Y''},C_{Y''}+R_{Y''})$ is generalised $\tau$-lc for some fixed $\tau>0$ depending only on $d,\epsilon$. In particular, multiplicities of the fibres of $f''$ over codimension one points of $Y''$ are bounded. 
 
\emph{Step 4.} 
There is a fixed $p\in \N$ such that $p(K_{X''}+tN'')$ is integral. We claim that, perhaps after replacing $p$ with a bounded multiple, $p(K_{X''}+tN'')\sim 0$ over the generic point of $Y''$.  The general fibres $G''$ of $f''$ are bounded by BAB [\ref{B-BAB}]. Since $K_{G''}+tN''|_{G''}\sim_\Q 0$ and $N''|_{G''}\ge 0$ is integral and $t$ is in a fixed finite set, 
$(G'',\Supp N''|_{G''})$ belongs to a bounded family and the coefficients of $tN''|_{G''}$ are in a fixed finite set. Thus replacing $p$, we can assume that $p(K_{G''}+tN''|_{G''})$ is Cartier, by [\ref{B-Fano}, Lemma 2.24]. But then $p(K_{G''}+tN''|_{G''})\sim 0$ because $G''$ being of Fano type implies $\Pic^0(G'')=0$. This implies the claim.

\emph{Step 5.}
We can write the above adjunction formula so that  
$$
p(K_{X''}+tN'')\sim pf''^*(K_{Y''}+C_{Y''}+R_{Y''}).
$$
Since $p(K_{X''}+tN'')$ is integral and since multiplicities of the fibres of $f''$ over codimension one points are bounded, we can assume that 
$$
J'':=p(K_{Y''}+C_{Y''}+R_{Y''})
$$ 
is integral, perhaps after replacing $p$ with a bounded multiple. 

By construction, $X''\to Y''$ is the contraction defined by the semi-ample$/Z$ divisor $K_{X''}+tN''$, hence $J''$ is ample over $Z$. We show that $\vol_{/Z}(J'')$ is bounded from above. Let $\phi\colon W\to X$ and $\psi\colon W\to X''$ be a common resolution. Pick a general point of $Z$ and let $F_W,F_X,F_{X''},F_{Y''}$ be the corresponding fibres over this point. We want to show that $\vol(J''|_{F_{Y''}})$ is bounded from above.

By [\ref{B-geom-pol-var}, Theorem 1.1], $|mN|_{F_X}|$ defines a birational map for some fixed $m\in \N$. This implies that $|\phi^*mN|_{F_W}|$ defines a birational map. Thus, if $R$ is a general fibre of $F_W\to F_{Y''}$, then $|\phi^*mN|_{R}|$ defines a birational map, hence 
$(\phi^*mN|_{R})^{c-e}\ge 1$ where $c=\dim F_W$ and $e=\dim F_{Y''}$. Therefore, 
\begin{equation}\label{eq-intersection-volume}
(\phi^*mN|_{F_W})^{c-e}\cdot (\psi^*p(K_{X''}+tN'')|_{F_{W}})^e=(\phi^*mN|_{F_W})^{c-e}\cdot (\psi^*f''^*J''|_{F_{W}})^e
\end{equation}
$$
\hspace{1.5cm} =(\phi^*mN|_{F_W})^{c-e}\cdot (\vol(J''|_{F_{Y''}})R)=(\phi^*mN|_{R})^{c-e}\vol(J''|_{F_{Y''}}) \ge \vol(J''|_{F_{Y''}}).
$$
To get the second equality we are using the fact that the zero-cycle $(J''|_{F_{Y''}})^e$ is $\Q$-linearly equivalent to a zero-cycle, say $\sum \lambda_j y_j$, of degree equal to $\sum \lambda_j=\vol(J''|_{F_{Y''}})$ where $y_j$ are general closed points; so to calculate the above intersection number we can  
replace $(\psi^*f''^*J''|_{F_{W}})^e$ with $\sum \lambda_jR_j$ where $R_j$ is the fibre of $F_W\to F_{Y''}$ over $y_j$; in turn we can replace each $R_j$ with the fixed general fibre $R$, so we can replace $\sum \lambda_jR_j$ with $\vol(J''|_{F_{Y''}})R$.  

It is then enough to show that the left hand side intersection number of (\ref{eq-intersection-volume})
is bounded from above. But this intersection number does not exceed the left hand side of the following inequalities: 
$$
\vol(\phi^*mN|_{F_W}+\psi^*p(K_{X''}+tN'')|_{F_{W}})\le \vol(\phi^*(pK_{X}+(m+pt)N)|_{F_{W}})
$$
$$
\le \vol(\phi^*(m+pt)N)|_{F_{W}})=\vol((m+pt)N|_{F_X})<(m+pt)^cv.
$$

\emph{Step 6.}
Now consider the generalised adjunction formula 
$$
K_{X''}+B''+M''\sim_\R f''^*(K_{Y''}+B_{Y''}+M_{Y''}).
$$
Since $-K_{X''}$ is big over $Y''$, applying Theorem \ref{t-sh-sing-gen-fib}, $({Y''},B_{Y''}+M_{Y''})$ is generalised $\zeta$-lc for some fixed $\zeta>0$ depending only on $d,\epsilon$. 
Moreover, by Step 5, $J''$ is integral and ample over $Z$ with bounded relative volume. Therefore, applying induction on the relative dimension we deduce that the discriminant divisor $\tilde{B}_{Z}$ of adjunction for $(Y'',B_{Y''}+M_{Y''})\to Z$
has coefficients $\le 1-\delta$ for some fixed $\delta>0$ (recall that $Z$ is a curve, so we only need to consider the discriminant divisor rather than the discriminant b-divisor). But $\tilde{B}_{Z}$ is also the discriminant divisor of adjunction for $(X'',B''+M'')\to Z$ (see the arguments of the proof of [\ref{B-lcyf}, Lemma 6.10]). 
Moreover, the discriminant divisor of adjunction for $(X'',B''+M'')\to Z$ coincides 
with the discriminant divisor ${B}_Z$ of adjunction for $(X,B+M)\to Z$: indeed, since 
$$
K_X+B+M\sim_\R 0/Z,
$$
on the common resolution $W$ of the previous step, we have 
$$
\phi^*(K_X+B+M)=\psi^*(K_{X''}+B''+M'')
$$ 
and the two generalised pairs $(X,B+M)$ and $(X'',B''+M'')$ have the same nef part, hence  
for any closed point $z\in Z$  and number $u$, $(X,B+uf^*z+M)$ is generalised lc iff 
$(X'',B''+ug''^*z+M'')$ is generalised lc where $g''$ denotes $X''\to Z$.  
Therefore, the coefficients of $B_Z$ do not exceed $1-\delta$, so $(Z,B_Z+M_Z)$ is generalised $\delta$-lc.
\end{proof}

\begin{proof}[Proof of Theorem \ref{t-cb-sing-usual-fib}]
This is a special case of Theorem \ref{t-cb-sing-gen-fib}.

\end{proof}


\vspace{2cm}

\small
\textsc{Yau Mathematical Sciences Center, JingZhai, Tsinghua University, Hai Dian District, Beijing, China 100084  } \endgraf
\vspace{0.5cm}
\email{Email: birkar@tsinghua.edu.cn\\}

\pagebreak

\end{document}